\documentclass[10pt]{amsart}
\usepackage{amssymb}
\usepackage{amscd}
\usepackage[all]{xy}

\numberwithin{equation}{section}

\def\today{\number\day\space\ifcase\month\or   January\or February\or
   March\or April\or May\or June\or   July\or August\or September\or
   October\or November\or December\fi\   \number\year}

\theoremstyle{definition}
\newtheorem{thm}{Theorem}[section]
\newtheorem{lem}[thm]{Lemma}
\newtheorem{prp}[thm]{Proposition}
\newtheorem{dfn}[thm]{Definition}
\newtheorem{cor}[thm]{Corollary}

\newtheorem{rmk}[thm]{Remark}
\newtheorem{ntn}[thm]{Notation}
\newtheorem{exa}[thm]{Example}

\newtheorem{qst}[thm]{Question}

\newcommand{\beq}{\begin{equation}}
\newcommand{\eeq}{\end{equation}}
\newcommand{\beqr}{\begin{eqnarray*}}
\newcommand{\eeqr}{\end{eqnarray*}}
\newcommand{\bal}{\begin{align*}}
\newcommand{\eal}{\end{align*}}
\newcommand{\bei}{\begin{itemize}}
\newcommand{\eei}{\end{itemize}}

\newcommand{\af}{\alpha}
\newcommand{\bt}{\beta}
\newcommand{\gm}{\gamma}
\newcommand{\dt}{\delta}
\newcommand{\ep}{\varepsilon}

\newcommand{\et}{\eta}
\newcommand{\ch}{\chi}
\newcommand{\io}{\iota}

\newcommand{\ld}{\lambda}
\newcommand{\sm}{\sigma}
\newcommand{\kp}{\kappa}
\newcommand{\ph}{\varphi}
\newcommand{\ps}{\psi}
\newcommand{\rh}{\rho}
\newcommand{\om}{\omega}
\newcommand{\ta}{\tau}

\newcommand{\Q}{{\mathbb{Q}}}
\newcommand{\Z}{{\mathbb{Z}}}
\newcommand{\R}{{\mathbb{R}}}
\newcommand{\C}{{\mathbb{C}}}
\newcommand{\N}{{\mathbb{Z}}_{> 0}}
\newcommand{\Nz}{{\mathbb{Z}}_{\geq 0}}

\newcommand{\inv}{{\mathrm{inv}}}

\newcommand{\id}{{\mathrm{id}}}

\newcommand{\sint}{{\mathrm{int}}}

\newcommand{\supp}{{\mathrm{supp}}}

\newcommand{\spn}{{\mathrm{span}}}
\newcommand{\card}{{\mathrm{card}}}
\newcommand{\Aut}{{\mathrm{Aut}}}
\newcommand{\Ad}{{\mathrm{Ad}}}

\newcommand{\ran}{{\mathrm{ran}}}

\newcommand{\andeqn}{\,\,\,\,\,\, {\mbox{and}} \,\,\,\,\,\,}

\newcommand{\ts}[1]{{\textstyle{#1}}}
\newcommand{\ds}[1]{{\displaystyle{#1}}}
\newcommand{\ssum}[1]{{\ts{ {\ds{\sum}}_{#1} }}}
\newcommand{\sssum}[2]{{\ts{ {\ds{\sum}}_{#1}^{#2} }}}

\newcommand{\Dirlim}{\varinjlim}

\newcommand{\Wolog}{Without loss of generality}

\newcommand{\ifo}{if and only if}

\newcommand{\ca}{C*-algebra}

\newcommand{\hm}{homomorphism}

\newcommand{\fd}{finite dimensional}

\newcommand{\ct}{continuous}
\newcommand{\cfn}{continuous function}

\newcommand{\cms}{compact metric space}

\newcommand{\hme}{homeomorphism}
\newcommand{\mh}{minimal homeomorphism}

\newcommand{\cp}{crossed product}

\newcommand{\rpn}{representation}

\newcommand{\msp}{measure space}
\newcommand{\sft}{$\sm$-finite}
\newcommand{\sfm}{$\sm$-finite measure space}

\newcommand{\XBM}{(X, {\mathcal{B}}, \mu)}
\newcommand{\YCN}{(Y, {\mathcal{C}}, \nu)}
\newcommand{\cB}{{\mathcal{B}}}
\newcommand{\cC}{{\mathcal{C}}}
\newcommand{\LLp}{L (L^p (X, \mu))}

\newcommand{\OP}[2]{{\mathcal{O}}_{#1}^{#2}}
\newcommand{\MP}[2]{M_{#1}^{#2}}

\renewcommand{\S}{\subset}
\newcommand{\ov}{\overline}
\newcommand{\SM}{\setminus}
\newcommand{\I}{\infty}
\newcommand{\E}{\varnothing}
\newcommand{\Lem}[1]{Lemma~\ref{#1}}
\newcommand{\Def}[1]{Definition~\ref{#1}}

\title[$L^p$~operator crossed products]{Crossed products
   of $L^p$~operator algebras
   and the K-theory of Cuntz algebras on $L^p$~spaces}

\author{N.~Christopher Phillips}

\date{24~September 2013}

\address{Department of Mathematics, University  of Oregon,
       Eugene OR 97403-1222, USA.}

\email[]{ncp@darkwing.uoregon.edu}

\thanks{This material is based upon work supported by the
  US National Science Foundation under Grants
  DMS-0701076 and DMS-1101742.
  It was also partially supported
  by the Research Institute for Mathematical Sciences
  of Kyoto University through a visiting professorship
  in 2011--2012.}

\begin{document}

\begin{abstract}
For $p \in [1, \infty),$ we define and study
full and reduced crossed products
of algebras of operators on $\sigma$-finite $L^p$~spaces
by isometric actions of second countable locally compact groups.
We give universal properties for both crossed products.
When the group is abelian,
we prove the existence of a dual action on the
full and reduced $L^p$~operator crossed products.
When the group is discrete,
we construct a conditional expectation to the original algebra
which is faithful in a suitable sense.
For a free action of a discrete group
on a compact metric space~$X,$
we identify all traces on the reduced $L^p$~operator crossed product,
and if the action is also minimal we show that
the reduced $L^p$~operator crossed product is simple.
We prove that the full and reduced $L^p$~operator crossed products
of an amenable $L^p$~operator algebra by a discrete amenable group
are again amenable.
We prove a Pimsner-Voiculescu exact sequence
for the K-theory of reduced $L^p$~operator crossed products by~$\Z.$
We show that the $L^p$ analogs ${\mathcal{O}}_d^p$
of the Cuntz algebras~${\mathcal{O}}_d$
are stably isomorphic to reduced $L^p$~operator crossed products
of stabilized $L^p$~UHF algebra by~${\mathbb{Z}},$
and show that
$K_0 \big( {\mathcal{O}}_d^p \big) \cong \Z / (d - 1) \Z$
and $K_1 \big( {\mathcal{O}}_d^p \big) = 0.$
\end{abstract}

\maketitle

\indent
This paper is an initial investigation of
full and reduced crossed products
of algebras of operators on $L^p$~spaces
by isometric actions of locally compact groups,
for $p \in [1, \infty).$
The original motivation was the computation of the K-theory
of the analogs of Cuntz algebras on $L^p$~spaces,
introduced in~\cite{PhLp1}.
The result is the same as in the C*~case:
$K_0 \big( \OP{d}{p} \big) \cong \Z / (d - 1) \Z$
and $K_1 \big( \OP{d}{p} \big) = 0.$
The choice of material in this paper is largely dictated
by what is needed for this calculation,
but we carry the work out in the greatest generality
not requiring significant extra work,
and we present some results unrelated to this computation
but which can be done with little additional work.
As steps toward the computation,
we prove that on the reduced $L^p$~operator crossed product
by a countable group,
there is a conditional expectation to the original algebra
which is faithful in a suitable sense,
and we prove a Pimsner-Voiculescu exact sequence
for the K-theory of reduced $L^p$~operator crossed products by~$\Z.$
We also construct the dual action on the
full and reduced $L^p$~operator crossed products
by a second countable locally compact abelian group.
The main unrelated results are as follows.
Let $G$ be a countable discrete group.
If $X$ is a free minimal compact metrizable $G$-space,
then the reduced $L^p$~operator crossed product by the corresponding
action on $C (X)$ is simple.
If $X$ is a free compact metrizable $G$-space,
not necessarily minimal,
then the traces on the reduced $L^p$~operator crossed product
are in one to one correspondence with the
invariant Borel probability measures on~$X.$
If $G$ is amenable and acts on an $L^p$~operator algebra
which is amenable as a Banach algebra,
then both the full and reduced $L^p$~operator crossed products
are amenable Banach algebras.

We also mention a separate result of Pooya and Hejezian~\cite{PH}.
Let $p \in (1, \I),$
let $G$ be a Powers group,
and let $\af \colon G \to \Aut (A)$
be an isometric action of~$G$ on an
$L^p$~operator algebra $A$
such that $A$ is $G$-simple.
Then the reduced $L^p$~operator crossed product
of $A$ by~$G$ is simple.
This is not true for $p = 1,$
as follows from Proposition~\ref{P_3217_L1OpGpAlg}.

We originally hoped to compute $K_* \big( \OP{d}{p} \big)$
directly,
without developing the theory of $L^p$~operator crossed products.
However, the methods of the original computation,
in~\cite{Cu2},
do not seem to work.
See the discussion at the beginning of Section~\ref{Sec_OpdCP}.

There are many interesting questions which we do not address.
We exclude the case $p = \I$ in most of the paper,
since it seems to require special treatment.
Whenever convenient, we restrict to discrete groups.
We make no effort to decide when the full and reduced
$L^p$~operator crossed products are the same,
beyond the simple observation that they are at least
isomorphic when the group is finite.
We do not investigate the crossed product by the dual action
(Takai duality).
We do not try to give any general criteria for simplicity
of crossed products;
even for actions on spaces,
our result is weaker than in the C*~case,
where one only needs the action to be essentially free.
There are many other interesting questions
which we do not even mention.

This paper is organized as follows.
In Section~\ref{Sec_LpOpAlg},
we define $L^p$~operator algebras
(closed subalgebras of $L (L^p (X, \mu))$),
give some examples,
and prove a few elementary facts.
Section~\ref{Sec_GAlg}
contains the definition of (isometric) $G$-$L^p$ operator algebras,
that is, $L^p$~operator algebras
with (isometric) actions of the group~$G,$
and their covariant \rpn{s}.
We give examples of $G$-$L^p$ operator algebras,
and prove the existence of regular covariant \rpn{s}.
Full and reduced $L^p$~operator crossed products
are introduced in Section~\ref{Sec_CP}.
Analogs of the usual universal properties for C*~crossed products
are given.
When the group is abelian,
dual actions are defined
on both the full and reduced $L^p$~operator crossed products.
Section~\ref{Sec_DiscCP}
contains the construction of the conditional expectation
on a reduced $L^p$~operator crossed product by a discrete group.
In Section~\ref{Sec_FpCX},
we prove the results mentioned above on $L^p$~operator crossed products
by free actions on compact spaces,
and on amenability of $L^p$~operator crossed products.
Section~\ref{Sec_PV} contains the proof of the analog
of the Pimsner-Voiculescu exact sequence for the K-theory
of reduced $L^p$~operator crossed products by~$\Z.$
In Section~\ref{Sec_OpdCP},
we show how to realize the stabilization of the algebra $\OP{d}{p}$
as the reduced $L^p$~operator crossed product of an
action of $\Z$ on a stabilized $L^p$~UHF algebra,
and we combine this realization with the
Pimsner-Voiculescu exact sequence to compute
$K_* \big( \OP{d}{p} \big).$
In Section~\ref{Sec_Problems}, we state a few of the many problems
left open in this paper.

We warn the reader that many facts
which are either automatic
or well known for \ca{s}
are false, unknown, or require additional work
for $L^p$~operator algebras.

Integration of Banach space values \cfn{s}
with compact support will be taken to be as in Section~2.5
of~\cite{DDW}.
For more details, see Section~1.5 of~\cite{Wl}.

I am grateful to Guillermo Corti\~{n}as,
Mar\'{\i}a Eugenia Rodr\'{\i}guez,
and Sanaz Pooya for finding a number of misprints
in earlier drafts of this paper.

\section{$L^p$~operator algebras}\label{Sec_LpOpAlg}

\indent
Crossed products will be taken with respect to the category
of norm closed subalgebras
of algebras of the form $\LLp,$
for a fixed value of~$p.$
We call such algebras $L^p$~operator algebras.
We will mostly restrict to nondegenerate algebras
and \sfm{s}.
In this section,
we present the basic definitions
related to $L^p$~operator algebras
and some examples.

\begin{dfn}\label{D-LpOpAlg}
Let $A$ be a Banach algebra,
and let $p \in [1, \I].$
We say that $A$ is an {\emph{$L^p$~operator algebra}}
if there is a measure space $\XBM$ such that
$A$ is isometrically isomorphic to a norm closed subalgebra
of $\LLp.$
\end{dfn}

We do not assume that $A$ is unital.

When $p = 2,$ an $L^p$~operator algebra
is a Banach algebra which is isometrically isomorphic
to a norm closed but not necessarily selfadjoint subalgebra
of the bounded operators on some Hilbert space.
(We do not know if there is a useful $L^p$~analog
of a C*-algebra.)
Results such as the characterization in~\cite{BRS}
suggest that nonselfadjoint operator algebras
are better behaved when matrix norms are included in the structure.
In the $L^p$~situation,
there is an obvious way to define matrix norms.
In~\cite{LM} there is a representation theorem
for matrix normed operator algebras on the collection
of quotients of subspaces of $L^p$~spaces,
for a fixed value of~$p.$
However, for $p \neq 2,$
this collection is much bigger than the collection of $L^p$~spaces,
and therefore does not meet our needs.
We do not need matrix norms for the purposes of this paper,
essentially because the algebras we consider have
a unique $L^p$~operator matrix normed structure.
Therefore we do not pursue this idea here.
We do make scattered remarks.
For example, as discussed at the beginning
of Section~\ref{Sec_CP},
the completely isometric theory of crossed products
is essentially a special case of what we do here.

When $p = \I,$
it may well be more appropriate to consider
norm closed subalgebras
of spaces of the form $L (C_0 (X))$
for locally compact Hausdorff spaces~$X.$
Since we will eventually have to exclude $p = \I$ anyway,
we do not pursue this approach here.

The algebra $\LLp$ is an obvious example
of an $L^p$~operator algebra.
Here are some others.

\begin{exa}\label{E-Mpd}
Let $p \in [1, \I],$
let $d \in \N,$
set $Z = \{ 0, 1, \ldots, d - 1 \},$
and let $\ld$ be normalized counting measure on~$Z,$
that is,
$\ld (S) = d^{-1} \card (S)$ for all $S \S Z.$
Then $L (L^p ( Z, \ld))$ is an $L^p$~operator algebra,
algebraically isomorphic to the algebra $M_d$ of $d \times d$
complex matrices.
We call it $M^p_d.$

We write its standard matrix units as
$e_{j, k}$ for $j, k \in Z.$
Thus,
\[
e_{j, k} (\ch_{ \{ l \} } ) = \begin{cases}
   0 & l \neq k
       \\
   \ch_{ \{ j \} } & l = k.
\end{cases}
\]
\end{exa}

We use $\{ 0, 1, \ldots, d - 1 \}$
rather than $\{ 1, 2, \ldots, d \},$
and normalized counting measure,
for notational convenience in Section~\ref{Sec_OpdCP}.

\begin{rmk}\label{R-UnNormalize}
Let the notation be as in Example~\ref{E-Mpd},
and let $\nu = d \cdot \ld$ be ordinary counting measure on~$Z.$
Then there is an isometric isomorphism $\ph$ from
$M^p_d$ as defined there to $L (l^p (Z, \nu))$
which sends matrix units to matrix units.
Indeed, let $u \in L \big( L^p (Z, \mu), \, L^p (Z, \ld) \big)$
be the isometric isomorphism $u \xi = d^{1/p} \xi.$
Then set $\ph (a) = u a u^{-1}$ for $a \in M^p_d.$

More generally, replacing the measure $\mu$
by a strictly positive multiple of $\mu$
does not change $\LLp.$
(See Lemma~2.11 of~\cite{PhLp2b} for the formal statement.)
\end{rmk}

The algebra $M^p_d$ plays a key role in \cite{PhLp1} and~\cite{PhLp2a}.

\begin{exa}\label{L_3911_UTr}
Let the notation be as in Example~\ref{E-Mpd},
and set
\[
T = \spn
 \big( \big\{ e_{j, k} \colon 0 \leq j \leq k \leq d - 1 \big\} \big),
\]
the algebra of upper triangular matrices.
Then $T$ is an $L^p$~operator algebra.
\end{exa}

Unlike the other examples we present here,
for $p = 2$ we do not get a \ca.

For any Banach space~$E,$
we let $K (E)$ denote the algebra of compact operators on~$E.$

\begin{exa}\label{E-CptOps}
Let $p \in [1, \I],$
and let $\XBM$ be a \msp.
Then the algebra $K (L^p (X, \mu))$ of compact operators
on $L^p (X, \mu)$ is an $L^p$~operator algebra.
\end{exa}

\begin{exa}\label{E-MatpS}
Let $S$ be a set,
and let $p \in [1, \I].$
For a finite subset $F \S S,$
we define $M_F^p$ to be the set of all $a \in L (l^p (S))$
such that $a \xi = 0$ whenever $\xi |_F = 0$
and such that $a \xi \in l^p (F) \S l^p (S)$
for all $\xi \in l^p (S).$
We then define
\[
M_S^p = \bigcup_{ {\mbox{$F \S S$ finite}} } M_F^p,
\]
and we define ${\ov{M}}_S^p$
to be the closure of $M_S^p$ in the operator norm on $L (l^p (S)).$
It follows from \Lem{L-MpIsAlg} below that
${\ov{M}}_S^p$ is an $L^p$~operator algebra which is contained
in $K (l^p (S)).$
By Corollary~\ref{C-UMnIsK} below, ${\ov{M}}_S^p = K (l^p (S))$
for $p \in (1, \I),$
but Example~\ref{E-UMn1} shows that this is not true for $p = 1.$
When $S = \{ 0, 1, \ldots, d - 1 \}$
or $S = \{ 1, 2, \ldots, d \},$
we recover the algebra $M^p_d$ of Example~\ref{E-Mpd},
including its norm,
by Remark~\ref{R-UnNormalize}.

As in Example~\ref{E-Mpd},
for $j, k \in S$
we let $e_{j, k} \in M^p_S \S {\ov{M}}^p_S$
be the standard matrix unit, given by
\[
e_{j, k} (\ch_{ \{ l \} } ) = \begin{cases}
   0 & l \neq k
       \\
   \ch_{ \{ j \} } & l = k.
\end{cases}
\]
\end{exa}

\begin{lem}\label{L-MpIsAlg}
Let $S$ be a set,
and let $p \in [1, \I).$
Then $M_S^p$ is a subalgebra of $L (l^p (S))$
and ${\ov{M}}_S^p$ is a closed subalgebra of $L (l^p (S))$
which is contained in $K (l^p (S)).$
\end{lem}

\begin{proof}
It is easy to check that $M_S^p$ is a subalgebra of $L (l^p (S))$
and all its elements have finite rank.
The rest of the statement follows.
\end{proof}

The following proposition and its corollary
will not be formally used.
They are included to show how ${\ov{M}}_S^p$
is related to the more usual algebra $K (l^p (S)).$

\begin{prp}\label{P-UMnDense}
Let $p \in [1, \I).$
Let $S$ be a set, and let ${\mathcal{F}}$ be the collection of
finite subsets of~$S,$
ordered by inclusion.
For $T \in {\mathcal{F}}$
let $e_T \in L (l^p (S))$ be multiplication by~$\ch_T.$
Let $a \in K (l^p (S)).$
If $p \in (1, \I),$ then
$\lim_{T \in \mathcal{F}} \| e_T a e_T - a \| = 0,$
and if $p = 1$ then $\lim_{T \in \mathcal{F}} \| e_T a - a \| = 0.$
\end{prp}

\begin{proof}
We begin with the easily checked observation
that for all $T \in {\mathcal{F}},$
we have
\begin{equation}\label{Eq:NormeF}
\| e_T \| \leq 1
\andeqn
\| 1 - e_T \| \leq 1.
\end{equation}
(We have equality unless $T = \varnothing$ or $T = S.$)

Now let $a \in K (l^p (S))$ and let $\ep > 0.$
We will find $T_0 \in {\mathcal{F}}$ such that
for all $T \in {\mathcal{F}}$ containing~$T_0,$
we have $\| e_T a - a \| < \tfrac{1}{3} \ep.$
For $p \neq 1$ we will further find $T_1 \in {\mathcal{F}}$
(containing~$T_0$) such that
for all $T \in {\mathcal{F}}$ containing~$T_1,$
we have $\| e_T a e_T - a \| < \ep.$

Define
\[
B = {\ov{ \big\{ a \xi \colon
  {\mbox{$\xi \in l^p (S)$ and $\| \xi \|_p \leq 1$}} \big\}  }}
 \S l^p (S).
\]
Then $B$ is compact.
For $T \in {\mathcal{F}},$
define an open set $U_T \S l^p (S)$ by
\[
U_T = \big\{ \et \in l^p (S) \colon
     \| (1 - e_T) \et \|_p < \tfrac{1}{6} \ep \big\}.
\]
If $T_1, T_2 \in {\mathcal{F}}$ satisfy
$T_1 \S T_2,$
then $U_{T_1} \S U_{T_2},$
since, by~(\ref{Eq:NormeF}), for $\et \in l^p (S)$ we have
\[
\| (1 - e_{T_2}) \et \|_p
 = \| (1 - e_{T_2}) (1 - e_{T_1}) \et \|_p
 \leq \| (1 - e_{T_1}) \et \|_p.
\]
Also,
for all $\et \in l^p (S),$
we have $\lim_{T \in \mathcal{F}} \| (1 - e_T) \et \|_p = 0,$
so $\et \in \bigcup_{T \in {\mathcal{F}} } U_T.$
Since $B$ is compact,
there is $T_0 \in {\mathcal{F}}$ such that
$B \S U_{T_0}.$

We now claim that whenever $T \in {\mathcal{F}}$ satisfies $T_0 \S T,$
then $\| e_T a - a \| < \tfrac{1}{3} \ep.$
Indeed, if $\xi \in l^p (S)$ satisfies
$\| \xi \|_p \leq 1,$
then $a \xi \in B,$
so
\[
\| (e_T a - a) \xi \|_p = \| (e_T - 1) a \xi \|_p < \tfrac{1}{6} \ep.
\]
Taking the supremum over all such~$\xi$ gives
$\| e_T a - a \| \leq \tfrac{1}{6} \ep < \tfrac{1}{3} \ep.$
This proves the claim.

We now have the statement in the case $p = 1.$
So from now on assume $p \in (1, \I).$
For $j \in S$ let $\dt_j$ denote
the corresponding standard basis vector in $l^p (S).$
There are functionals $\om_j$ in the dual space $l^p (S)'$
for $j \in T_0$
such that for all $\xi \in l^p (S)$
we have
\begin{equation}\label{Eq_3912_St}
e_{T_0} a \xi = \sum_{j \in T_0} \om_j (\xi) \dt_j.
\end{equation}
Let $q \in (1, \I)$ be the conjugate exponent,
that is, $\frac{1}{p} + \frac{1}{q} = 1.$
There are $\et_j \in l^q (S)$ for $j \in T_0$
such that for all $\xi \in l^p (S)$ we have
\[
\om_j (\xi) = \sum_{l \in S} (\et_j)_l \xi_l.
\]

For $T \in {\mathcal{F}}$
let $f_T \in L (l^q (S))$ be multiplication by~$\ch_T.$
There is $T_1 \in {\mathcal{F}}$ such that $T_0 \S T_1$
and such that for $T \in {\mathcal{F}}$ with $T_1 \S T,$
and $j \in T_0,$ we have
\[
\| f_T \et_j - \et_j \|_q < \frac{\ep}{6 \cdot \card (T_0)}.
\]
Since $\xi \mapsto \om_j (e_T \xi)$
is the linear functional corresponding to $f_T \et_j \in l^q (S),$
we get,
for all $\xi \in l^p (S),$ all $j \in T_0,$
and all $T \in {\mathcal{F}}$ containing~$T_1,$
\[
| \om_j (\xi) - \om_j (e_T \xi) |
  \leq \| \xi \|_p \cdot \| \et_j - f_T \et_j \|_q
  \leq \frac{\ep \| \xi \|_p}{6 \cdot \card (T_0)}.
\]
Therefore, by~(\ref{Eq_3912_St}),
\[
\| e_{T_0} a e_T \xi -  e_{T_0} a \xi \|_p
  \leq \sum_{j \in T_0} | \om_j (\xi) - \om_j (e_T \xi) |
  \leq \frac{\ep \| \xi \|_p}{6}.
\]
Taking the supremum over all such~$\xi$ gives
$\| e_{T_0} a e_T - e_{T_0} a \|
   \leq \tfrac{1}{6} \ep < \tfrac{1}{3} \ep.$
Therefore
\begin{align*}
\| e_T a e_T - a \|
& \leq \| e_T ( a - e_{T_0} a) e_T \|
         + \| e_{T_0} a e_T - e_{T_0} a \|
         + \| e_{T_0} a - a \|
       \\
& < \| e_T \| \cdot \| a - e_{T_0} a \| \cdot \| e_T \|
         + \tfrac{1}{3} \ep
         + \tfrac{1}{3} \ep
  < \tfrac{1}{3} \ep + \tfrac{1}{3} \ep + \tfrac{1}{3} \ep
  = \ep.
\end{align*}
This completes the proof that
$\lim_{T \in \mathcal{F}} \| e_T a e_T - a \| = 0.$
\end{proof}

\begin{cor}\label{C-UMnIsK}
Let $p \in (1, \I),$
and let $S$ be a set.
Then ${\ov{M}}^p_S = K (l^p (S)).$
\end{cor}

\begin{proof}
This is immediate from Proposition~\ref{P-UMnDense}.
\end{proof}

We also get the result that the finite rank operators
are dense in $K (l^p (S))$ whenever $p \in [1, \I).$
This is surely known.
The case in which $S$ is countable is known in much greater
generality,
namely for every Banach space with a Schauder basis.

It is not true that ${\ov{M}}^1_S = K (l^1 (S)),$
even when $S$ is countable.
In fact,
if $S$ is infinite, then ${\ov{M}}_S^1$ does not
even contain all rank one operators in $L (l^1 (S)).$

\begin{exa}\label{E-UMn1}
Let $S$ be an infinite set,
and fix $s_0 \in S.$
For $j \in S$ let $\dt_j$ denote
the corresponding standard basis vector in $l^1 (S).$
Define $a \colon l^1 (S) \to l^1 (S)$
by
\[
a \xi = \left( \ssum{j \in S} \xi_j \right) \dt_{s_0}.
\]
Then one easily checks that $a$ is a rank one operator
in $L (l^1 (S))$
and that $\| a \| = 1.$
We show that $a \not\in {\ov{M}}_S^1.$

Let $T \S S$ be any finite set,
and let $e_T$ be as in Proposition~\ref{P-UMnDense}.
We first observe that $\| a - a e_T \| \geq 1.$
Indeed, there is $j \in S$ such that $j \not\in T,$
and taking $\xi = \dt_j$ gives $\| \xi \| = 1,$
$\| a \xi \| = 1,$ and $a e_T \xi = 0.$

We now claim that $b \in {\ov{M}}_S^1$ implies
$\| a - b \| \geq 1.$
It suffices to prove this for $b \in M_S^1.$
Thus we may assume that there is a finite set $T \S S$
such that $e_T b e_T = b.$
Since $b (1 - e_T) = 0,$
we get
\[
\| a - b \|
  = \| a - b \| \cdot \| 1 - e_T \|
  \geq \| (a - b) (1 - e_T) \|
  = \| a (1 - e_T) \|
  \geq 1,
\]
as desired.
\end{exa}

\begin{exa}\label{E-pUHF}
Let $p \in [1, \I).$
Let $P$ be the set of prime numbers,
and let $N \colon P \to \Nz \cup \{ \I \}$
be a function such that $\sum_{t \in P} N (t) = \I.$
(Such a function is a {\emph{supernatural number}}.)
Then the spatial $L^p$~UHF algebra~$D$ of type~$N,$
as in Definition~3.9(2) of~\cite{PhLp2a}
(and whose uniqueness is proved in Theorem~3.10 of~\cite{PhLp2a})
is an $L^p$~operator algebra.
Recall from Definition~3.5 and Theorem~3.10 of~\cite{PhLp2a}
that $D$ is characterized as follows.
For any sequence $\big( r (0), \, r (1), \, r (2), \, \ldots \big)$
in~$\N$
with $r (0) = 1,$ $r (n) \mid r (n + 1)$ for $n \in \Nz,$
and such that for all $t \in P$ we have
\[
N (t) = \sup \big( \big\{ m \in \Nz \colon
 {\mbox{there is $n \in \Nz$ such that $t^m \mid r (n)$}} \big\} \big),
\]
there are subalgebras
$D_0 \S D_1 \S \cdots \S D$
such that ${\ov{\bigcup_{n = 0}^{\I} D_n}} = D$
and such that $D_n$ is isometrically isomorphic to $M^p_{r (n)}$
for all $n \in \Nz.$

As a special case, for $d \in \{ 2, 3, \ldots \},$
we get the spatial $L^p$~UHF algebra of type~$d^{\infty}$
by taking $N (t) = \I$ for those primes~$t$ which divide~$d$
and $N (t) = 0$ otherwise.
Then we can take $r (n) = d^n$ for $n \in \Nz.$
\end{exa}

Theorem 5.14 of~\cite{PhLp2b}
shows that for every $p \in [1, \I)$
and every supernatural number~$N,$
there are uncountably many
nonisomorphic nonspatial $L^p$~UHF algebras of type~$N.$

\begin{exa}\label{E-Odp}
Let $p \in [1, \I),$
and let $d \in \N.$
Then the algebra $\OP{d}{p}$
of Definition~8.8 of~\cite{PhLp1}
is an $L^p$~operator algebra.

We recall its definition.
For convenience in Section~\ref{Sec_OpdCP},
we use slightly different notation.
We let $L_d$ denote the Leavitt algebra over~$\C,$
as in Definition~1.1 of~\cite{PhLp1},
with standard generators
(note the change in labelling)
$s_0, s_1, \ldots, s_{d - 1}, t_0, t_1, \ldots, t_{d - 1}$
satisfying the relations
\begin{equation}\label{Eq:Leavitt1}
t_j s_j = 1
\,\,\,\,\,\,\,\,\,\,\,\,
{\mbox{for $j \in \{ 0, 1, \ldots, d - 1 \},$}}
\end{equation}
\begin{equation}\label{Eq:Leavitt2}
t_j s_k = 0
\,\,\,\,\,\,\,\,\,\,\,\,
{\mbox{for $j, k \in \{ 0, 1, \ldots, d - 1 \}$ with $j \neq k,$}}
\end{equation}
and
\begin{equation}\label{Eq:Leavitt3}
\sum_{j = 0}^{d - 1} s_j t_j = 1.
\end{equation}
Then $\OP{d}{p}$ is the completion of
$L_d$ in the norm coming from any spatial \rpn{}
(in the sense of Definition 7.4(2) of~\cite{PhLp1})
on a space $L^p (X, \mu)$ for a \sfm\  $\XBM.$
By Theorem~8.7 of~\cite{PhLp1},
all such \rpn{s} give the same norm on~$L_d,$
so $\OP{d}{p}$ is well defined.
Since injective spatial \rpn{s} of $L_d$ exist
(Lemma~7.5 of~\cite{PhLp1})
we may, and do,
regard $L_d$ as a subalgebra of~$\OP{d}{p}.$
\end{exa}

\begin{exa}\label{E-CX}
Let $p \in [1, \I],$
and let $X$ be a locally compact Hausdorff space.
Then $C_0 (X),$ with the usual supremum norm,
is an $L^p$~operator algebra.
To see this, let $\mu$ be counting measure on~$X.$
Then the map $f \mapsto m (f),$
sending $f \in C_0 (X)$ to the multiplication operator
$(m (f) \xi) (x) = f (x) \xi (x)$ for $\xi \in L^p (X, \mu)$
and $x \in X,$
is an isometric bijection from $C_0 (X)$
to a norm closed subalgebra of $\LLp.$

If $X$ is compact metrizable
(and in some other cases),
we can find a finite Borel measure $\nu$ on~$X$
such that $\nu (U) > 0$ for every nonempty open set $U \S X.$
Then $C (X)$ is isometrically isomorphic to the corresponding
algebra of multiplication operators on $L^p (X, \nu).$
If $X$ is compact metrizable and $p \neq \I,$
then $L^p (X, \nu)$ is separable.
\end{exa}

Our final example is the spatial $L^p$~operator tensor product
of $L^p$~operator algebras.
For this,
we need to recall briefly the tensor product of operators
on $L^p$~spaces.
What we need is summarized in Theorem~2.16 of~\cite{PhLp1},
but is taken from Chapter~7 of~\cite{DF}.
Those sources assume the measures are \sft,
but this hypothesis is not needed, by Section~1 of~\cite{FIP}.

When we need them,
we will use the symbol $\otimes_{\mathrm{alg}}$
for algebraic (that is, not completed)
tensor products of both Banach spaces and Banach algebras.

\begin{rmk}\label{R-LpTens}
For proofs of the following, see Theorem~2.16 of~\cite{PhLp1},
and remove the hypothesis of $\sm$-finiteness there
by using, in~\cite{FIP}, Theorem~1.1
and the discussion at the beginning of Section~1.

For $p \in [1, \I)$
and for \msp{s} $\XBM$ and $\YCN,$
there is an $L^p$~tensor product
$L^p (X, \mu) \otimes_p L^p (Y, \nu)$
which can be canonically identified
with $L^p (X \times Y, \, \mu \times \nu)$
via $(\xi \otimes \et) (x, y) = \xi (x) \et (y).$
Moreover, if
\[
(X_1, \cB_1, \mu_1), \,\,\,\,\,\,
(X_2, \cB_2, \mu_2), \,\,\,\,\,\,
(Y_1, \cC_1, \nu_1),
\andeqn
(Y_2, \cC_2, \nu_2)
\]
are \msp s,
and
\[
a \in L \big( L^p (X_1, \mu_1), \, L^p (X_2, \mu_2) \big)
\andeqn
b \in L \big( L^p (Y_1, \nu_1), \, L^p (Y_2, \nu_2) \big),
\]
then there is a corresponding tensor product operator
\[
a \otimes b
   \in L \big( L^p (X_1 \times Y_1, \, \mu_1 \times \nu_1),
          \, L^p (X_2 \times Y_2, \, \mu_2 \times \nu_2 \big),
\]
which has the expected properties:
bilinearity,
$(a_1 \otimes b_1) (a_2 \otimes b_2) = a_1 a_2 \otimes b_1 b_2,$
and $\| c \| = \| a \| \cdot \| b \|.$
\end{rmk}

We exclude $p = \I$
because usually $L^{\I} (X \times Y, \, \mu \times \nu)$
is much bigger than the closure of the image of
$L^{\I} (X, \mu) \otimes_p L^{\I} (Y, \nu).$

\begin{exa}\label{E-LpTP}
Let $p \in [1, \I),$
let $(X_1, \cB_1, \mu_1)$ and $(X_2, \cB_2, \mu_2)$
be \msp s,
and let $A_1 \S L ( L^p (X_1, \mu_1) )$
and $A_2 \S L ( L^p (X_2, \mu_2) )$
be norm closed subalgebras.
Define the algebra
\[
A_1 \otimes_p A_2
  \S L \big( L^p (X_1 \times X_2, \, \mu_1 \times \mu_2) \big)
\]
to be the closed linear span of all $a_1 \otimes a_2$
(in the sense of Remark~\ref{R-LpTens})
for $a_1 \in A_1$ and $a_2 \in A_2,$
as in Definition~1.9 of~\cite{PhLp2a}.
Then $A_1 \otimes_p A_2$ is an $L^p$~operator algebra.

Thus, for example,
if $A \S \LLp$ is an $L^p$~operator algebra,
we can form an $L^p$~matrix algebra:
let $d \in \N,$ let $Z$ and $\ld$ be as in Example~\ref{E-Mpd},
and get
\[
M^p_d \otimes_p A \S L \big( L^p (Z \times X, \, \ld\times \mu) \big).
\]
Remark~\ref{R-UnNormalize} is easily adapted to show that
we get the same result using counting measure instead of~$\ld.$
\end{exa}

If $A_1$ and $A_2$ are $L^p$~operator algebras,
we can choose
\msp s $(X_1, {\mathcal{B}}_1, \mu_1)$
and $(X_2, {\mathcal{B}}_2, \mu_2),$
and isometric \rpn{s}
(in the sense of \Def{D-LpRep} below)
\[
\pi_1 \colon A_1 \to L (L^p (X_1, \mu_1))
\andeqn
\pi_2 \colon A_2 \to L (L^p (X_2, \mu_2)).
\]
Following Example~\ref{E-LpTP},
we can then form the $L^p$~operator tensor product
$\pi_1 (A_1) \otimes_p \pi_2 (A_2).$
In general,
the resulting $L^p$~operator tensor product
can depend on the choice of $\pi_1$ and~$\pi_2,$
even when $p = 2.$
The following suggestion is due to Vern Paulsen.
Choose operator spaces $E_1, E_2, F_1, F_2$ on Hilbert spaces
such that $E_1$ is isometrically isomorphic to $F_1,$
such that $E_2$ is isometrically isomorphic to $F_2,$
but such that the spatial tensor products
$E_1 \otimes E_2$ and $F_1 \otimes F_2$
are distinct.
For example,
let $H$ be the two dimensional Hilbert space $l^2 ( \{ 1, 2 \} ).$
Take $E_2$ and $F_2$ to be
the column Hilbert space $H^{\text{c}}$
and the row Hilbert space $H^{\text{r}}$
associated to~$H$
(1.2.23 of~\cite{BM}).
Take $E_1 = F_1 = M_2 = L ( H ).$
Then $M_2 (E_2) \cong M_2 \otimes E_2$
and $M_2 (F_2) \cong M_2 \otimes F_2$
by 1.5.2 of~\cite{BM}.
There are $x, y \in M_2,$
such as $x = e_{1, 1}$ and $y = e_{1, 2},$
for which $\| x^* x + y^* y \| \neq \| x x^* + y y^* \|.$
Using 1.2.5 and 1.2.24 of~\cite{BM},
we see that the map $M_2 \otimes E_2 \to M_2 \otimes F_2$
is not isometric.
Define
\[
A_1 = \left\{ \left( \begin{matrix}
  \ld_1     &  a        \\
    0       &  \ld_2
\end{matrix} \right) \colon
   {\mbox{$a \in E_1$ and $\ld_1, \ld_2 \in \C$}} \right\}
 \S L (H \oplus H).
\]
Similarly define $A_2,$ $B_1,$ and $B_2$
using $E_2,$ $F_1,$ and $F_2$ in place of~$E_1.$
Then $A_1$ is isometrically isomorphic to~$B_1$
and $A_2$ is isometrically isomorphic to~$B_2,$
but the obvious map
$A_1 \otimes_{\mathrm{alg}} A_2 \to B_1 \otimes_{\mathrm{alg}} B_2$
does not extend to an isometric isomorphism of
these algebras on $(H \oplus H) \otimes (H \oplus H).$
It is an isomorphism by finite dimensionality,
but using $l^2$ in place of $H$ will give an example
for which the obvious map does not extend to an isomorphism of the closures.

This can't happen for \ca{s},
as pointed out by Narutaka Ozawa.
An isometric \hm{} from a \ca~$A$ to $L (H),$
for any Hilbert space~$H,$
must be a *-\hm,
by Proposition A.5.8 of~\cite{BM}.

Because of the rigidity
of contractive \rpn{s} of $\OP{d}{p}$
(see the remark after \Lem{L-OdIsoMdOd}),
it turns out that
we do not need a general theory
of $L^p$~operator tensor products.
Therefore we do not develop such a theory here.

\begin{lem}\label{L-MatpSMap}
Let $p \in [1, \I).$
Let $S$ and $T$ be sets,
and let $g \colon S \to T$ be an injective function.
Let $\XBM$ be a \msp,
and let $A \S \LLp$ be a closed subalgebra.
Then there is a unique isometric \hm{}
$\gm_{g, A} \colon
 {\ov{M}}^p_S \otimes_p A \to {\ov{M}}^p_T \otimes_p A$
such that
$\gm_{g, A} (e_{j, k} \otimes a) = e_{g (j), \, g (k) } \otimes a$
for all $j, k \in S$ and all $a \in A.$
Its range is ${\ov{M}}^p_{g (S)} \otimes_p A.$
The assignment $g \mapsto \gm_g$ is functorial.
\end{lem}

We do not claim functoriality in~$A,$
because the norm on ${\ov{M}}^p_S \otimes_p A$
in general depends on how $A$ is represented.

The proof avoids possible issues with the product
of $\mu$ and counting measure on~$S$
by making serious use of it only on finite subsets of~$S.$

\begin{proof}[Proof of Lemma~\ref{L-MatpSMap}]
Regard $M^p_S \otimes_{\mathrm{alg}} A$
and $M^p_T \otimes_{\mathrm{alg}} A$
as subalgebras of
${\ov{M}}^p_S \otimes_p A$ and ${\ov{M}}^p_T \otimes_p A,$
with the restricted norms.
By definition,
they are dense.
There is an obvious algebra \hm{}
$\gm_{g, A}^{(0)} \colon
 M^p_S \otimes_{\mathrm{alg}} A \to M^p_T \otimes_{\mathrm{alg}} A$
such that
$\gm_{g, A}^{(0)} (e_{j, k} \otimes a) = e_{g (j), \, g (k)} \otimes a$
for all $j, k \in S$ and all $a \in A,$
which is is functorial in~$g.$
The closure of its range
is obviously ${\ov{M}}^p_{g (S)} \otimes_p A.$
The proof is therefore completed by showing that $\gm_{g, A}^{(0)}$
is isometric
and extending by continuity.

It suffices to show that, for any finite set $F \S S,$
the restriction
of $\gm_{g, A}^{(0)}$ to $M^p_F \otimes_{\mathrm{alg}} A$
is isometric.
Thus, we may as well assume that $S$ is finite.
Let $\nu_S$ and $\nu_T$ be counting measure on $S$ and~$T,$
with all subsets taken to be measurable,
and equip $S \times X$ and $T \times X$ with the product
$\sm$-algebras and measures.
Equip $g (S) \times X \S T \times X$
with the restricted $\sm$-algebra and measure.
We then have a bijection $h = g \times \id_X$
from $S \times X$ to $g (S) \times X$
which preserves measurable sets and the measure.
The map $h$ induces is an isometric bijection
\[
\ph \colon
  L \big( L^p (S \times X, \, \nu_S \times \mu) \big)
  \to L \big( L^p (g (S) \times X, \, \nu_T \times \mu) \big).
\]
Let $f \in L \big( L^p (T \times X, \, \nu_T \times \mu) \big)$
be multiplication by the characteristic function of $g (S) \times X.$
Then there is an obvious isometric identification
\[
L \big( L^p (g (S) \times X, \, \nu_T \times \mu) \big)
 = f L \big( L^p (T \times X, \, \nu_T \times \mu) \big) f,
\]
and thus an isometric inclusion
\[
\io \colon L \big( L^p (g (S) \times X, \, \nu_T \times \mu) \big)
  \to L \big( L^p (T \times X, \, \nu_T \times \mu) \big).
\]
The restriction of $\io \circ \ph$ to $M^p_S \otimes_{\mathrm{alg}} A$
agrees with $\gm_{g, A}^{(0)}.$
So $\gm_{g, A}^{(0)}$ is isometric,
as desired.
\end{proof}

There are many other examples
of $L^p$~operator algebras.
We will see a few later,
namely $L^p$~operator crossed products.

\begin{dfn}\label{D-LpRep}
Let $p \in [1, \I],$
and let $A$ be an $L^p$~operator algebra.
\begin{enumerate}
\item\label{D_3914_LpRep_Rep}
A {\emph{representation}} of~$A$
(on $L^p (X, \mu)$)
is a \ct\  \hm\  $\pi \colon A \to \LLp$
for some measure space $\XBM.$
\item\label{D_3914_LpRep_Contr}
The \rpn{} $\pi$ is said to be {\emph{contractive}}
if $\| \pi (a) \| \leq \| a \|$ for all $a \in A,$
and {\emph{isometric}}
if $\| \pi (a) \| = \| a \|$ for all $a \in A.$
\item\label{D_3914_LpRep_Spb}
If $p \neq \I,$
we say that the \rpn{} $\pi \colon A \to \LLp$
is {\emph{separable}} if $L^p (X, \mu)$ is separable,
and that $A$ is {\emph{separably representable}}
if it has a separable isometric representation.
\item\label{D_3914_LpRep_Sft}
We say that $\pi$ is {\emph{$\sigma$-finite}}
if $\mu$ is \sft,
and that $A$ is {\emph{$\sigma$-finitely representable}}
if it has a $\sigma$-finite isometric representation.
\item\label{D_3914_LpRep_Ndg}
We say that $\pi$ is {\emph{nondegenerate}}
if
\[
\pi (A) E = \spn \big( \big\{ \pi (a) \xi \colon
     {\mbox{$a \in A$ and $\xi \in E$}} \big\} \big)
\]
is dense in~$E.$
We say that $A$ is {\emph{nondegenerately (separably) representable}}
if it has a nondegenerate (separable) isometric representation,
and {\emph{nondegenerately $\sm$-finitely representable}}
if it has a nondegenerate $\sm$-finite isometric representation.
\end{enumerate}
\end{dfn}

We will only be interested in contractive \rpn{s},
but it seems potentially confusing to incorporate
contractivity into the definition.

\begin{rmk}\label{R_3916_SepSft}
Let $p \in [1, \I).$
The corollary to Theorem~3 in Section~15 of~\cite{Lc}
implies that any separable $L^p$~space
is isometrically isomorphic to a $\sigma$-finite $L^p$~space.
So separably (nondegenerately) representable
implies $\sigma$-finitely (nondegenerately) representable.
\end{rmk}

We do not require representations to be unital,
even when the algebra is unital.
But a nonunital representation of a unital algebra
is necessarily degenerate.

\begin{exa}\label{E_3419_Deg}
Use the notation of Example~\ref{E-Mpd}.
Take $S = \{ 1, 2 \},$
so that the algebra there is $\MP{2}{p}.$
Take $A = \C \cdot e_{1, 2}.$
Then $A$ is an $L^p$~operator algebra which has a separable isometric
representation.
We claim that $A$ has no nondegenerate representation
on any nonzero Banach space.

Let $E$ be a nonzero Banach space,
and let $\pi \colon A \to L (E)$
be a representation.
Then $\pi (e_{1, 2})^2 = 0,$
so $\pi (e_{1, 2}) E \S \ker ( \pi (e_{1, 2}) ).$
If $\pi (e_{1, 2}) \neq 0,$
then $\ker (\pi (e_{1, 2}))$ is a proper closed subspace
of $E$ which contains $\pi (A) E,$
while if $\pi (e_{1, 2}) = 0$ then $\pi (A) E = 0.$
\end{exa}

Lemma~\ref{L-LpSum} below is essentially the same as
Lemma~3.5 of~\cite{PhLp2b},
except that we do not assume that our direct sums are countable.
We first define a version of the disjoint union of measure spaces which
is suitable for our purposes.

\begin{dfn}\label{D_3914_DjUnion}
Let $S$ be a set,
and for $j \in S$ let $(X_j, \cB_j, \mu_j)$
be a \msp{} with $\mu_j (X_j) > 0.$
The
{\emph{disjoint union}}
$\coprod_{j \in S} (X_j, \cB_j, \mu_j)$
is the \msp{} $\XBM$ defined as follows.
Set $X = \coprod_{j \in S} X_j.$
Take $\cB$ to be the collection
of subsets $Y \S X$ such that $Y \cap X_j \in \cB_j$
for all $j \in S$
and such that $Y \cap X_j = \E$
for all but countably many $j \in S$
or $Y \cap X_j = X_j$
for all but countably many $j \in S.$
Define a measure $\mu$ on $X$
by $\mu (Y) = \sum_{j \in S} \mu_j (Y \cap X_j)$
for $Y \in \cB.$
\end{dfn}

\begin{lem}\label{L_3914_PrpDj}
Let the notation be as in Definition~\ref{D_3914_DjUnion}.
Then:
\begin{enumerate}
\item\label{D_3914_DjUnion_IsM}
$\XBM = \coprod_{j \in S} (X_j, \cB_j, \mu_j)$
is a \msp.
\item\label{D_3914_DjUnion_Supp}
If $Y \S X$ satisfies $Y \cap X_j \neq \E$
for uncountably many $j \in S,$
then $\mu (Y) = \infty.$
\item\label{D_3914_DjUnion_Sep}
If $S$ is countable and $L^p (X_j, \mu_j)$ is separable for all $j \in S,$
then $L^p (X, \mu)$ is separable.
\item\label{D_3914_DjUnion_Sft}
If $S$ is countable and $\mu_j$ is \sft{} for all $j \in S,$
then $\mu$ is \sft.
\end{enumerate}
\end{lem}

\begin{proof}
All parts are immediate.
\end{proof}

\begin{lem}\label{L-LpSum}
Let $p \in [1, \I),$
and let $A$ be an $L^p$~operator algebra.
Let $S$ be a set.
For $j \in S$ let $(X_j, \cB_j, \mu_j)$
be a \msp{} with $\mu_j (X_j) > 0$ and
let $\pi_j \colon A \to L (L^p (X_j, \mu_j))$
be a contractive \rpn.
Equip the algebraic direct sum $E_0 = \bigoplus_{j \in S} L^p (X_j, \mu_j)$
with the norm
\[
\big\| (\xi_j )_{j \in S} \big\|
 = \left( \ssum{j \in S} \| \xi_j \|_p^p \right)^{1 / p},
\]
and let $E$ be the completion of $E_0$ in this norm.
Set $\XBM = \coprod_{j \in S} (X_j, \cB_j, \mu_j).$
Then $E \cong L^p (X, \mu),$
and there is a unique contractive \rpn{} $\pi \colon A \to L (E)$
such that
\[
\pi (a) \big( (\xi_j )_{j \in S} \big)
        = \big( (\pi_j (a) )_{j \in S} \big)
\]
for $a \in A$ and $\xi_j \in L^p (X_j, \mu_j)$
for $j \in S.$
\end{lem}

\begin{proof}
The proof is immediate.
\end{proof}

We can of course replace contractivity
with
$\| \pi_j \| \leq M$ (with $M$ independent of~$j$) in the hypothesis
and $\| \pi \| \leq M$ in the conclusion.

\begin{dfn}\label{D-LpDirectSum}
The \rpn~$\pi$ of \Lem{L-LpSum}
is called the {\emph{$L^p$-direct sum}} of the \rpn{s}~$\pi_j,$
and written $\pi = \bigoplus_{j \in S} \pi_j.$
(We suppress~$p$ in the notation.)
\end{dfn}

\begin{lem}\label{L_3914_SmNdg}
Let $A$ be an $L^p$~operator algebra.
Then the $L^p$-direct sum of nondegenerate contractive \rpn{s}
is nondegenerate.
\end{lem}

\begin{proof}
Let the notation be as in Lemma~\ref{L-LpSum}.
It is enough to show that if $j_0 \in S$
and $\et = (\et_j )_{j \in S}$ is an element of the
algebraic direct sum $E_0 = \bigoplus_{j \in S} L^p (X_j, \mu_j)$
such that $\et_j = 0$ for $j \neq j_0,$
then
$\xi \in {\overline{\pi (A) E_0}}.$
This is easily verified using elements
$\xi \in E_0$ such that $\xi_j = 0$ for $j \neq j_0.$
\end{proof}

We are interested in separably representable algebras
for technical reasons which will become apparent
in the example to which we want to apply the theory.
(See Section~\ref{Sec_OpdCP}.)

\begin{prp}\label{P-SepImpSepRep}
Let $p \in [1, \I),$
and let $A$ be a separable $L^p$~operator algebra.
Then $A$ is separably representable.
If $A$ is nondegenerately representable,
then $A$ is separably nondegenerately representable.
\end{prp}

The proof is harder than for $p = 2$
because a closed subspace of $L^p (X, \mu)$
need not be isomorphic to any $L^p (Y, \nu).$
Also, we don't have anything which plays the role of the adjoint,
so invariant subspaces need not be reducing.
This means that we can't use the most obvious form of
the method of decomposing a \rpn{} into
cyclic \rpn{s}.

Of course, there are nonseparable $L^p$~operator algebras
which are separably representable,
such as $L (l^p (\N)).$

\begin{proof}[Proof of Proposition~\ref{P-SepImpSepRep}]
By hypothesis,
there is a \msp\  $\XBM$ and an isometric
\rpn{} $\rh \colon A \to \LLp,$
nondegenerate if $A$ is nondegenerately representable.
Recall that $L^p (X, \mu)$ is a complex Banach lattice,
in the sense of Definition~3 in Section~1 of~\cite{Lc}
and Definition~1 in Section~3 of~\cite{Lc}.
The real part consists of the real valued functions
in $L^p (X, \mu),$
the order is the pointwise almost everywhere order,
and $\xi \vee \et$ is defined by
$(\xi \vee \et) (x) = \max (\xi (x), \et (x) )$ for $x \in X.$
We write ${\mathrm{Re}} (\xi)$ and ${\mathrm{Im}} (\xi)$
with the obvious meanings.

We claim that for every separable subset $Q \S L^p (X, \mu),$
there is a separable closed sublattice $F \S L^p (X, \mu)$
which contains $Q$ and which is invariant in the sense
that $\rh (a) F \S F$ for all $a \in A.$

We prove the claim.
Choose a countable dense subset $S \S A.$

For any set $T \S L^p (X, \mu),$
we define $G (T) \S L^p (X, \mu)$ as follows.
First,
let $G_0 (T)$ consist of all functions
${\mathrm{Re}} (\xi),$ ${\mathrm{Im}} (\xi),$
and $| \xi |,$
for $\xi \in T$ or $\xi = \rh (b) \et$ with $\et \in T$ and $b \in S.$
Then take
\[
G_1 (T) = \big\{ \xi_1 \vee \xi_2 \vee \cdots \vee \xi_n \colon
     {\mbox{$n \in \N$ and $\xi_1, \xi_2, \ldots, \xi_n \in G_0 (T)$}}
            \big\}.
\]
Finally, take $G (T)$ to be the
$\Q [i]$-linear span of $G_1 (T).$
Note that $T \S G (T),$
that $\rh (b) \xi \in G (T)$ for $\xi \in T$ and $b \in S,$
and that if $T$ is countable then so is $G (T).$

Now let $E_0$ be a countable dense subset of~$Q$
and, for $n \in \Nz,$
set $E_{n + 1} = G (E_n).$
Set $E = \bigcup_{n = 0}^{\I} E_n$
and $F = {\ov{E}}.$
It is clear that $F$ is a closed subspace of $L^p (X, \mu)$
and that (by density of $S$ in~$A$)
we have $\rh (A) F \S F.$
By continuity of the operations, if $\xi \in F$
then ${\mathrm{Re}} (\xi), \, {\mathrm{Im}} (\xi), \, | \xi | \in F,$
and if $\xi, \et \in F$ are real then $\xi \vee \et \in F.$
Therefore $F$ is a $\rh$-invariant closed sublattice of $L^p (X, \mu)$
which contains~$Q.$
(It is clearly the smallest such sublattice.)
It is separable because $E$ is countable.
This proves the claim.

We next claim that if $\rh$ is nondegenerate,
then the separable invariant closed sublattice $F \S L^p (X, \mu)$
in the claim above can be chosen to also be nondegenerate.
To prove this,
let $F_0$ be a sublattice as in the previous claim.
We construct by induction
separable invariant closed sublattices $F_m \S L^p (X, \mu)$
such that $F_m \S F_{m + 1}$
and $F_m \S {\ov{\spn}} (\rh (A) F_{m + 1})$ for all $m \in \Nz.$
Suppose we have $F_m.$
Let $P_m$ be a countable dense subset of $F_m$
and for every $\xi \in P_m$ and $n \in \N,$
use nondegeneracy of~$\rh$
to find $l_{\xi, n} \in \N,$
and $\et_{\xi, n, j} \in L^p (X, \mu)$
and $a_{\xi, n, j} \in A$ for $j = 1, 2, \ldots, l_{\xi, n},$
such that
\[
\left\| \xi
      -  \sssum{j = 1}{l_{\xi, n}} \pi ( a_{\xi, n, j} ) \et_{\xi, n, j}
    \right\|_p
 < \frac{1}{n}.
\]
Set
\[
Q_m = \big\{ \et_{\xi, n, j} \colon
    {\mbox{$n \in \N,$ $\xi \in P_m,$
       and $j = 1, 2, \ldots, l_{\xi, n}$}} \big\},
\]
which is a countable subset of $L^p (X, \mu)$
such that $F_m \S {\ov{\spn}} (\rh (A) Q_{m}).$
The previous claim provides a
separable invariant closed sublattice $F_{m + 1} \S L^p (X, \mu)$
which contains $F_m \cup Q_m.$
Clearly $F_m \S {\ov{\spn}} (\rh (A) F_{m + 1}).$
This completes the induction.

Set $F = {\ov{\bigcup_{m = 0}^{\I} F_m}}.$
Then $F$ is a
separable invariant closed sublattice of $L^p (X, \mu)$
which contains $Q$
and such that ${\ov{\spn}} (\rh (A) F) = F.$
This proves the claim.

We now claim that for every $a \in A$ and every $\ep > 0,$
there is a \msp\  $\YCN$ such that $L^p (Y, \nu)$ is separable,
and a contractive \rpn{} $\pi \colon A \to L ( L^p (Y, \nu) )$
such that $\| \pi (a) \| > \| a \| - \ep.$
Moreover, if $A$ is nondegenerately representable,
then $\pi$ may be chosen to be nondegenerate.
To prove the claim,
choose $\xi \in L^p (X, \mu)$ such that
$\| \xi \|_p = 1$ and $\| \rh (a) \xi \|_p > \| a \| - \ep.$
The claims above provide a
separable invariant closed sublattice $F \S L^p (X, \mu)$
which contains~$\xi,$
and which can be taken to satisfy
${\ov{\spn}} (\rh (A) F) = F$ if $\rh$ is nondegenerate.
Being a sublattice of $L^p (X, \mu),$
the space $F$ is clearly an abstract $L^p$-space
in the sense of Definition~1 in Section~15 of~\cite{Lc}.
According to Theorem~3 in Section~15 of~\cite{Lc},
it follows that there is a \msp\  $\YCN$
such that $F$ is isometrically isomorphic to $L^p (Y, \nu).$
Thus, $\pi (a) = \rh (a) |_{F}$ defines a contractive \rpn{} %
on a separable $L^p$~space,
which is nondegenerate if $\rh$ is nondegenerate.
Since $\xi \in F,$
we have $\| \pi (a) \| > \| a \| - \ep.$
This proves the claim.

Now we prove the proposition.
Let $S \S A$ be a countable dense subset.
For every $b \in S$ and $n \in \N,$
choose a contractive \rpn{}
$\pi_{b, n} \colon A \to L \big( L^p (Y_{b, n}, \nu_{b, n}) \big)$
as in the previous claim with $a = b$ and $\ep = \frac{1}{n}.$
Then let $\pi$ be the $L^p$~direct sum
(\Def{D-LpDirectSum}) of the \rpn{s} $\pi_{b, n},$
which is a contractive \rpn{} of $A$ on
a separable $L^p$~space.
The \rpn{} $\pi$ satisfies $\| \pi (b) \| = \| b \|$ for all $b \in S,$
and hence for all $b \in A.$
If $A$ is nondegenerately representable,
then $\pi$ is nondegenerate by Lemma~\ref{L_3914_SmNdg}.
\end{proof}

\section{Group actions and covariant representations}\label{Sec_GAlg}

\indent
We take group actions to be as in Section~2 of~\cite{PhLp2a}.

\begin{dfn}[Definition~2.3 of~\cite{PhLp2a}]\label{D:Aut}
Let $B$ be a Banach algebra.
An {\emph{automorphism}} of~$B$
is a \ct\  linear bijection $\ph \colon B \to B$
such that $\ph (a b) = \ph (a) \ph (b)$ for all $a, b \in B.$
(Continuity of $\ph^{-1}$ is automatic,
by the Open Mapping Theorem.)
We call $\ph$ an {\emph{isometric automorphism}}
if in addition $\| \ph (a) \| = \| a \|$ for all $a \in B.$
We let $\Aut (B)$ denote the set of automorphisms of~$B.$
\end{dfn}

\begin{dfn}\label{D:Action}
Let $A$ be a Banach algebra,
and let $G$ be a topological group.
An {\emph{action of $G$ on~$A$}}
is a \hm\  $g \mapsto \af_g$ from $G$ to $\Aut (A)$
such that for every $a \in A,$
the map $g \mapsto \af_g (a)$ is \ct\  from $G$ to~$A.$
We say that the action is {\emph{isometric}} if each $\af_g$ is.
If $p \in [1, \I]$ and $A$ is an $L^p$~operator algebra,
we call $(G, A, \af)$ a {\emph{$G$-$L^p$~operator algebra}},
and we call it an {\emph{isometric $G$-$L^p$~operator algebra}}
if in addition $\af$~is isometric.
We say $(G, A, \af)$ is {\emph{separable}},
{\emph{nondegenerately representable}},
{\emph{$\sm$-finitely representable}}, etc.,
if $A$ is,
in the sense of Definition~\ref{D-LpRep}.
\end{dfn}

We will only construct crossed products by
isometric actions of locally compact groups.
To avoid measure theoretic technicalities
(primarily related to product measures and Fubini's Theorem),
we mostly restrict to second countable groups
and to algebras on \sft{} $L^p$~spaces.
Since integrable functions have \sft{} supports,
we expect that, with care, these restrictions can be avoided.
However, we do not need the extra generality
and therefore give the simpler proofs which work without it.

\begin{rmk}\label{R-DynSys}
If $\af \colon G \to \Aut (A)$ is any action as in \Def{D:Action},
then $(G, A, \af)$ is a Banach algebra dynamical system
in the sense of Definition~2.10 of~\cite{DDW}.
\end{rmk}
The action needed for the computation of
$K_* \big( \OP{d}{p} \big)$
will be described in Notation~\ref{N-DICrPrd} below.
Dual actions on $L^p$~crossed products by abelian groups
will be constructed in Theorem~\ref{T-DualpFull}
and Theorem~\ref{T-DualpRed} below.
Meanwhile,
we give several other examples.
Except for the first,
they will not be used in the rest of this paper,
and the first one will not be used in the computation of
the K-theory of $\OP{d}{p}.$

\begin{exa}\label{E-GpOnSp}
Let $X$ be a locally compact Hausdorff space,
and let $G$ be a locally compact group
which acts continuously on~$X.$
Then $C_0 (X)$ is an $L^p$~operator algebra by Example~\ref{E-CX},
and the usual formula $\af_g (f) (x) = f (g^{-1} x)$
makes $(G, \, C_0 (X), \, \af)$ an isometric $G$-$L^p$ operator algebra.

As an important special case,
a \hme\  $h \colon X \to X$ gives an isometric action of~$\Z$
on $C_0 (X).$
\end{exa}

\begin{exa}\label{E-MdIso}
Let $p \in [1, \I) \SM \{ 2 \},$
let $d \in \N,$
and let $G$ be the group of isometries in the algebra $M^p_d$
of Example~\ref{E-Mpd}.
It follows from Theorem~6.9 and Lemma~6.15 of~\cite{PhLp1}
that $G$ consists of what we call complex permutation matrices,
that is, all $u \in M_d$ such that each row and each column of~$u$
contains exactly one nonzero entry,
and such that all nonzero entries have absolute value~$1.$
(These are the products of permutation matrices
and diagonal matrices whose diagonal entries have absolute value~$1.$)

The group $G$ is compact,
and its action on $M^p_d$ by conjugation
makes $M^p_d$ an isometric $G$-$L^p$ operator algebra.
\end{exa}

\begin{exa}\label{E-QF}
Let $p \in [1, \I) \SM \{ 2 \},$
let $d \in \N,$
and let $G$ be as in Example~\ref{E-MdIso}.
Let $\OP{d}{p}$ be as in Example~\ref{E-Odp}.
Then there is an isometric action
$\af \colon G \to \Aut \big( \OP{d}{p} \big)$
given as follows.
Write an element $g \in G$
as $g = \sum_{k = 0}^{d - 1} \ld_k e_{\sm (k), \, k}$
for a permutation $\sm$ of $\{ 0, 1, \ldots, d - 1 \}$
and complex numbers $\ld_0, \ld_1, \ldots, \ld_{d - 1}$
of absolute value~$1.$
Then $\af_g$ is determined by
\[
\af_g (s_j) = \ld_j s_{\sm (j)}
\andeqn
\af_g (t_j) = {\ov{\ld_j}} t_{\sm (j)}.
\]
One easily checks that these formulas define an action of~$G$
on~$L_d.$
It is isometric for the norm on $\OP{d}{p},$
because
if $\pi \colon \OP{d}{p} \to \LLp$
is a unital \rpn{} such that $\pi |_{L_d}$ is spatial
in the sense of Definition 7.4(2) of~\cite{PhLp1},
then $\pi \circ \af_g |_{L_d}$ is easily seen to be spatial.
Continuity of the action follows from continuity on the generators
via a standard $\frac{\ep}{3}$~argument.

Of course, individual automorphisms $\af_g$ generate
isometric actions of~$\Z.$
\end{exa}

Our definitions of crossed products
follow the general framework of~\cite{DDW}.
The following definition is the analog
in our situation (restricting to \rpn{s} on $L^p$~spaces)
of Definition 2.12 of~\cite{DDW}.

\begin{dfn}\label{D-CovRep}
Let $p \in [1, \I],$
let $G$ be a topological group,
and let $(G, A, \af)$ be a $G$-$L^p$~operator algebra.
Let $\XBM$ be a \msp.
Then a {\emph{covariant representation}} of
$(G, A, \af)$ on $L^p (X, \mu)$
is a pair $(v, \pi)$
consisting of a \rpn{} $g \mapsto v_g$
from $G$ to the invertible operators on $L^p (X, \mu)$
such that $g \mapsto v_g \xi$ is \ct\  for all $\xi \in L^p (X, \mu),$
and a \rpn{} $\pi \colon A \to \LLp,$
such that, in addition,
the following covariance condition is satisfied:
$\pi (\af_g (a)) = v_g \pi (a) v_g^{-1}$
for all $g \in G$ and $a \in A.$

A covariant representation $(v, \pi)$ of $(G, A, \af)$
is {\emph{contractive}} if $\| v_g \| \leq 1$ for all $g \in G$
and also $\pi$ is contractive.
It is {\emph{isometric}} if in addition $\pi$ is isometric.
It is {\emph{separable}}, {\emph{$\sigma$-finite}}, or
{\emph{nondegenerate}} if $\pi$ has the corresponding property.
\end{dfn}

If $(v, \pi)$ is contractive,
then necessarily $v_g$ is an isometric bijection for all $g \in G.$

\begin{rmk}\label{R-UnifBdd}
Let $p \in [1, \I],$
let $G$ be a locally compact group,
and let $(G, A, \af)$ be an isometric $G$-$L^p$~operator algebra.
Then the class of all contractive covariant representations
is uniformly bounded in the sense of Definition~3.1 of~\cite{DDW}.
\end{rmk}

For isometric actions of locally compact groups,
we show that covariant \rpn{s} exist
by constructing regular covariant \rpn{s}.
We need some preparation.

\begin{rmk}\label{R-CcGLp}
Let $p \in [1, \I),$
let $G$ be a locally compact group with left Haar measure~$\nu,$
and let $\XBM$ be a \msp.
Following Remark~\ref{R-LpTens},
we may identify the space
$C_{\mathrm{c}} \big( G, \, L^p (X, \mu) \big)$
of compactly supported \cfn s
from $G$ to $L^p (X, \mu)$
with a dense subspace of $L^p (G \times X, \, \nu \times \mu).$
Indeed, using Remark~\ref{R-LpTens},
this follows because the subspace contains all $\et \otimes \xi$
with $\et \in C_{\mathrm{c}} (G)$ and $\xi \in L^p (X, \mu),$
and $C_{\mathrm{c}} (G)$ is dense in $L^p (G, \nu).$
\end{rmk}

\begin{rmk}\label{R-lpG}
Let $p \in [1, \I),$
and let $G$ be a countable discrete group with counting measure~$\nu.$
For any Banach space~$E,$
let $l^p (G, E)$ be the Banach space of all functions
$\xi \colon G \to E$ such that the norm
\[
\| \xi \|_p = \left( \ssum{g \in G} \| \xi (g) \|^p \right)^{ 1 / p}
\]
is finite.
If now $\XBM$ is a \msp\  and $E = L^p (X, \mu),$
then for $\xi \in l^p (G, E)$ we get a function
$T (\xi)$ on $G \times X,$
given by $T (\xi) (g, x) = \xi (g) (x).$
One easily checks that $T$ is an isometric isomorphism
from $l^p \big( G, \, L^p (X, \mu) \big)$ to
$L^p (G \times X, \, \nu \times \mu).$
We therefore identify these spaces.
\end{rmk}

\begin{lem}\label{L-ExistRegRep}
Let $p \in [1, \I),$
let $G$ be a locally compact group with left Haar measure~$\nu,$
and let $(G, A, \af)$ be an isometric $G$-$L^p$~operator algebra.
Let $\XBM$ be a \msp,
and let $\pi_0 \colon A \to \LLp$
be a contractive \rpn.
Then:
\begin{enumerate}
\item\label{ExistRegRep-Exv}
There exists a unique \rpn~$v$ of $G$ on
$L^p (G \times X, \, \nu \times \mu)$ such that
$v_g (\xi) (h, x) = \xi (g^{-1} h, \, x)$
for $g, h \in G,$ $\xi \in L^p (G \times X, \, \nu \times \mu),$
and $x \in X.$
\item\label{ExistRegRep-VIsom}
The \rpn{} $v$ is isometric.
\item\label{ExistRegRep-Expi}
There is a unique \rpn~$\pi$
of $A$ on
$L^p (G \times X, \, \nu \times \mu)$ such that
for
\[
a \in A,
\,\,\,\,\,\,
h \in G,
\andeqn
\xi \in C_{\mathrm{c}} \big( G, \, L^p (X, \mu) \big)
    \S L^p (G \times X, \, \nu \times \mu)
\]
(see Remark~\ref{R-CcGLp}),
we have
\begin{equation}\label{Eq:RegRepFormula}
\big( \pi (a) \xi \big) (h)
 = \pi_0 ( \af_{h}^{-1} (a)) \big( \xi (h) \big).
\end{equation}
\item\label{ExistRegRep-PiContr}
The \rpn{} $\pi$ is contractive.
\item\label{ExistRegRep-GDiscr}
Suppose $G$ is discrete and countable.
For any
$\xi \in L^p (G \times X, \, \nu \times \mu),$
represented, following Remark~\ref{R-lpG}, as an element of
$l^p (G, \, L^p (X, \mu)),$
the formula~(\ref{Eq:RegRepFormula}) holds.
\item\label{ExistRegRep-Cov}
The pair $(v, \pi)$ is covariant.
\item\label{ExistRegRep-Ndg}
If $\pi_0$ is nondegenerate then $\pi$ is nondegenerate.
\item\label{ExistRegRep-Sft}
If $G$ is second countable
and $\mu$ is \sft,
then $\nu \times \mu$ is \sft.
\item\label{ExistRegRep-Sep}
If $G$ is second countable
and $L^p (X, \mu)$ is separable,
then $L^p (G \times X, \, \nu \times \mu)$ is separable.
\end{enumerate}
\end{lem}

\begin{proof}
Uniqueness of~$v$ is clear.
Now let $g \mapsto u_g$
be the usual left regular \rpn{} of $G$ on $L^p (G).$
Existence and contractivity of~$v$ follow from
the tensor product decomposition
\[
L^p (G \times X, \, \nu \times \mu)
  = L^p (G, \nu) \otimes_p L^p (X, \mu)
\]
as in Remark~\ref{R-LpTens},
which gives
$v_g = u_g \otimes 1$ for $g \in G.$

We now consider~$\pi.$
It is clear that for $a \in A,$
the formula~(\ref{Eq:RegRepFormula})
gives a well defined map
\[
T (a) \colon C_{\mathrm{c}} \big( G, \, L^p (X, \mu) \big)
    \to  C_{\mathrm{c}} \big( G, \, L^p (X, \mu) \big).
\]
Since $\pi_0$ and $\af_{h}^{-1}$ are contractive,
one readily checks that
$\| T (a) \xi \|_p \leq \| a \| \cdot \| \xi \|_p$
for $a \in A$ and
$\xi \in C_{\mathrm{c}} \big( G, \, L^p (X, \mu) \big).$
Moreover, $C_{\mathrm{c}} \big( G, \, L^p (X, \mu) \big)$
is dense in $L^p (G \times X, \, \nu \times \mu)$
by Remark~\ref{R-CcGLp}.
Therefore $T (a)$ has a unique extension to a linear operator
$\pi (a)$ on $L^p (G \times X, \, \nu \times \mu)$
with $\| \pi (a) \| = \| T (a) \| \leq \| a \|.$
One checks that $\pi$ is a \rpn{} by considering
its action on $C_{\mathrm{c}} \big( G, \, L^p (X, \mu) \big).$
This proves existence, uniqueness, and contractivity of~$\pi.$

For part~(\ref{ExistRegRep-GDiscr}),
one checks that when $G$ is discrete,
the formula~(\ref{Eq:RegRepFormula})
gives a well defined bounded operator $S (a)$ on
$L^p (G \times X, \, \nu \times \mu)$
which agrees with $T (a)$
on $C_{\mathrm{c}} \big( G, \, L^p (X, \mu) \big).$
Therefore $S (a) = \pi (a).$

Covariance can be checked by comparing values on
elements of $C_{\mathrm{c}} \big( G, \, L^p (X, \mu) \big).$

We prove part~(\ref{ExistRegRep-Ndg}).
It suffices to show that the linear span of the ranges
of the maps $T (a),$
for $a \in A,$
is dense in $C_{\mathrm{c}} \big( G, \, L^p (X, \mu) \big)$
in the norm on $L^p (G \times X, \, \nu \times \mu).$
So let $\xi \in C_{\mathrm{c}} \big( G, \, L^p (X, \mu) \big)$
and let $\ep > 0.$
Choose a compact set $K \S G$
such that $\supp (\xi) \S \sint (K).$
Set
\[
\ep_0 = \frac{\ep}{3 \nu (K)^{1/p} + 1}.
\]
Since $\pi_0$ is nondegenerate,
for $g \in G$ there are
\[
l (g) \in \N,
\,\,\,\,\,\,
\et_{g, 1}, \, \et_{g, 2}, \ldots, \et_{g, l (g)} \in L^p (X, \mu),
\andeqn
a_{g, 1}, \, a_{g, 2}, \ldots, a_{g, l (g)} \in A
\]
such that
\[
\left\| \xi (g)
 - \sssum{j = 1}{l (g)} \pi_0 (a_{g, j}) \xi_{g, j} \right\|_p
< \ep_0.
\]
Choose an open set $U (g) \S G$
containing~$g$
such that for $h \in U (G)$ we have
$\| \xi (h) - \xi (g) \| < \ep_0$
and
\[
\big\| (\pi_0 \circ \af_{h^{-1} g} ) (a_{g, j}) \et_{g, j}
        - \pi_0 (a_{g, j}) \xi_{g, j} \big\|_p
    < \frac{\ep_0}{l (g)}
\]
for $j = 1, 2, \ldots, l (g).$
For any $h \in U (g),$
we then get
\[
\left\| \xi (h)
 - \sssum{j = 1}{l (g)} \pi_0
              \big( \af_h^{-1} (\af_g (a_{g, j})) \big) \et_{g, j}
            \right\|_p
< 3 \ep_0.
\]

Choose elements $g_1, g_2, \ldots, g_n \in G$
and \cfn{s} $f_1, f_2, \ldots, f_n \colon G \to [0, 1]$
with compact supports $\supp (f_k) \S U (g_k) \cap \sint (K)$
such that $\sum_{k = 1}^n f_k (g) = 1$ for all $g \in \supp (\xi)$
and such that $\sum_{k = 1}^n f_k (g) \in [0, 1]$ for all $g \in G.$
For $k = 1, 2, \ldots, n$ and $j = 1, 2, \ldots, l (g_k),$
define
$\ld_{j, k} \in C_{\mathrm{c}} \big( G, \, L^p (X, \mu) \big)$
by $\ld_{j, k} (g) = f_k (g) \et_{g_k, j}$
and define $\xi_k \in C_{\mathrm{c}} \big( G, \, L^p (X, \mu) \big)$
by $\xi_k (g) = f_k (g) \xi (g).$
Then $\xi = \sum_{k = 1}^n \xi_k.$
Therefore
\begin{align*}
& \left\| \xi - \sssum{k = 1}{n} \sssum{j = 1}{l (g)}
  \pi (\af_{g_k} (a_{g_k, j})) \ld_{j, k} \right\|_p^p
\\
& \hspace*{3em} {\mbox{}}
 = \int_G \left\| \sssum{k = 1}{n} f_k (h)
      \left[ \xi (h)
 - \sssum{j = 1}{l (g)}
       \pi_0 \big( \af_h^{-1} (\af_{g_k} (a_{g_k, j})) \big)
                 \et_{g_k, j}
           \right] \right\|^p \, d \nu (h)
\\
& \hspace*{3em} {\mbox{}}
 \leq \int_G \left[ \sssum{k = 1}{n} f_k (h)
      \left\| \xi (h)
 - \sssum{j = 1}{l (g)}
       \pi_0 \big( \af_h^{-1} (\af_{g_k} (a_{g_k, j})) \big)
                \et_{g_k, j}
           \right\| \right]^p \, d \nu (h)
\\
& \hspace*{3em} {\mbox{}}
 \leq \int_G \left( \sssum{k = 1}{n} f_k (h) \cdot 3 \ep_0 \right)^p
        \, d \nu (h)
  \leq 3^p \ep_0^p \nu (K).
\end{align*}
So
\[
\left\| \xi - \sssum{k = 1}{n} \sssum{j = 1}{l (g)}
    \pi (\af_{g_k} (a_{g_k, j})) \ld_{j, k} \right\|_p
 \leq 3 \ep_0 \nu (K)^{1/p}
 < \ep.
\]
This completes the proof that $\pi$ is nondegenerate.

Part~(\ref{ExistRegRep-Sft})
and part~(\ref{ExistRegRep-Sep})
are well known.
\end{proof}

In \Lem{L-ExistRegRep},
we at least need $\sup_{g \in G} \| \af_g \|$ to be finite.
Otherwise,
the densely defined operators $T (a)$ in the proof might
not be bounded.

\begin{dfn}\label{D-RegRep}
Let the notation be as in \Lem{L-ExistRegRep}.
The covariant \rpn{} $(v, \pi)$
is called the {\emph{regular covariant representation}}
of $(G, A, \af)$ associated to~$\pi_0.$
We refer to any \rpn{} obtained this way as
a {\emph{regular contractive covariant representation}}.
We call it {\emph{separable}}, {\emph{$\sm$-finite}},
or {\emph{nondegenerate}}
if the original \rpn~$\pi_0$ has the corresponding property.
\end{dfn}

\section{Crossed products of $L^p$~operator algebras}\label{Sec_CP}

\indent
We present a rudimentary theory of crossed products
of algebras of bounded operators on $L^p$~spaces
by isometric actions of countable discrete groups.
Our intended purpose is the computation of the
K-theory of Cuntz algebras on $L^p$~spaces,
by realizing them as crossed products by actions of~$\Z.$
For the initial part of the theory,
however,
there is no extra work in allowing an arbitrary
second countable locally compact group,
and we therefore work in that generality.

As mentioned in the introduction to Section~\ref{Sec_LpOpAlg},
the completely isometric version of the theory may well be
at least as important.
One can copy everything done here in that context.
For example, in Definition~\ref{D-CrPrd},
one would use completely contractive \rpn{s}
instead of contractive \rpn{s}.
In fact, though,
the resulting theory is essentially a special case of the theory
we develop here and in the next section.
Let ${\ov{M}}^p_{\N}$ be as in Example~\ref{E-MatpS}.
Let $\et$ be counting measure on~$\N.$
Let the spatial $L^p$~operator tensor product be
as in Example~\ref{E-LpTP}.
Let $A \S \LLp$ be a norm closed subalgebra.
Then a \rpn{} $\pi \colon A \to L ( L^p (Y, \nu))$
is completely contractive \ifo\  the map
\[
\id_{{\ov{M}}^p_{\N}} \otimes \pi \colon
    {\ov{M}}^p_{\N} \otimes_{\mathrm{alg}} A
     \to L \big( L^p (\N \times Y, \, \et \times \nu) \big)
\]
extends to a contractive \hm\  defined on
${\ov{M}}^p_{\N} \otimes_{p} A,$
an action of $G$ on~$A$ is completely isometric \ifo\  it
induces an isometric action on ${\ov{M}}^p_{\N} \otimes_{p} A,$
etc.
So one can simply form the full or reduced crossed product,
as constructed here,
by the action on ${\ov{M}}^p_{\N} \otimes_{p} A,$
and cut down by $e_{1, 1} \otimes 1.$

\begin{ntn}\label{N-CcGA}
Let $A$ be a Banach algebra,
let $G$ be a locally compact group with left Haar measure~$\nu,$
and let $\af \colon G \to \Aut (A)$ be an action of $G$ on~$A.$
We let $C_{\mathrm{c}} (G, A, \af)$ be the vector space
of all compactly supported continuous functions from $G$ to~$A,$
made into an algebra over~$\C$ with the convolution product
\[
(a b) (g) = \int_G a (h) \af_h (b ( h^{-1} g)) \, d \nu (h)
\]
for $a, b \in C_{\mathrm{c}} (G, A, \af)$ and $g \in G,$
as at the beginning of Section~3 of~\cite{DDW},
or as at the beginning of Section~2.3 of~\cite{Wl}.

If $(v, \pi)$ is a covariant \rpn{} of $(G, A, \af)$
on a Banach space~$E$
in the general sense of Definition~2.12 of~\cite{DDW},
then we denote by
$v \ltimes \pi$ the \rpn{}  of $C_{\mathrm{c}} (G, A, \af)$
defined in equation~3.2 of \cite{DDW}
(called $\pi \rtimes v$ there), given by
\[
(v \ltimes \pi) (a) \xi = \int_G \pi (a (g)) v_g \xi \, d \nu (g)
\]
for $a \in C_{\mathrm{c}} (G, A, \af)$ and $\xi \in E.$
The integral is defined by duality,
as in Theorem~2.17 and Proposition~2.19 of~\cite{DDW}:
for all $\om$ in the dual space~$E'$ of~$E,$
we require that
\begin{equation}\label{Eq:DefOfInt}
\om \big( (v \ltimes \pi) (a) \xi \big)
  = \int_G \om \big( \pi (a (g)) v_g \xi \big) \, d \nu (g).
\end{equation}
In particular, we use the notation $v \ltimes \pi$
when $(v, \pi)$ is a covariant \rpn{} in the sense
of \Def{D-CovRep}.
\end{ntn}

\begin{rmk}\label{R-RegRepCP}
Let $p \in [1, \I),$
and let the notation be as in \Lem{L-ExistRegRep}.
Assume that $\mu$ is \sft.
Then a computation shows that for
\[
\xi \in C_{\mathrm{c}} \big( G, \, L^p (X, \mu) \big)
    \subset L^p (G \times X, \, \nu \times \mu)
\]
(following Remark~\ref{R-CcGLp})
and for $a \in C_{\mathrm{c}} (G, A, \af),$
we have
\[
\big( (v \ltimes \pi) (a) \xi \big) (g)
  = \int_G \pi_0 \big( \af_{g}^{-1} (a (h) ) \big)
      \big( \xi (h^{- 1} g) \big) \, d \nu (h).
\]
If $p \in (1, \I),$
we claim that
this integral is uniquely determined by the following requirement.
Let $\om \colon G \to L^p (X, \mu)'$ be an arbitrary
\cfn\  with compact support.
Then $\om$ determines a \ct\  linear functional on
$L^p (G \times X, \, \nu \times \mu),$ whose value
on an element $\et \in C_{\mathrm{c}} (G, A, \af)$
is given by $\int_G \om (g) ( \et (g)) \, d \nu (g).$
The requirement is then that the value of this functional
on $(v \ltimes \pi) (a) \xi$ be
\begin{align}\label{Eq:WeakIntRegRep}
& \int_G \om (g) \big( \big[ (v \ltimes \pi) (a) \xi \big] (g) \big)
    \, d \nu (g)
  \\
& \hspace*{3em} {\mbox{}}
  = \int_{G \times G}
    \om (g) \big( \pi_0 \big( \af_{g}^{-1} (a (h)) \big)
             \big( \xi (h^{- 1} g) \big) \big)
       \, d \nu (h) d \nu (g). \notag
\end{align}
(To see that these values determine $(v \ltimes \pi) (a) \xi,$
let $q \in (1, \I)$ be the conjugate exponent,
that is, $\frac{1}{p} + \frac{1}{q} = 1.$
Apply Remark~\ref{R-CcGLp} to
$L^q (G \times X, \, \nu \times \mu).$
Conclude that the set of such linear functionals is dense
in $L^p (G \times X, \, \nu \times \mu)'.$)
\end{rmk}

At this point, we must make a choice:
do we require covariant \rpn{s} to be nondegenerate?
For C*~crossed products,
both versions are used.
For example, nondegeneracy is required in 7.6.5 of~\cite{Pd1},
but not in Lemma~2.27 of~\cite{Wl}.
Lemma~2.31 of~\cite{Wl} shows that one gets the same
crossed product either way.
In our context,
with no analog of selfadjointness,
$L^p$~operator algebras
need not even be nondegenerately representable.
(See Example~\ref{E_3419_Deg}.)
Even if the algebra has an approximate identity,
we have not determined whether one gets the same crossed product
with the two choices.
We require nondegeneracy
because we want representations of unital algebras to be unital,
an outcome which is forced by nondegeneracy.

\begin{dfn}\label{D-CrPrd}
Let $p \in [1, \I),$
let $G$ be a second countable locally compact group,
and let $(G, A, \af)$ be an isometric $G$-$L^p$~operator algebra
which is nondegenerately $\sm$-finitely representable
(Definition \ref{D-LpRep}(\ref{D_3914_LpRep_Ndg})).
We define two crossed products following
Definition~3.2 of~\cite{DDW}
(justified in Lemma~\ref{L_3916_Exist} below),
but with different choices of the family~${\mathcal{R}}$
of covariant \rpn{s} which appears there:
\begin{enumerate}
\item\label{D-CrPrd-f}
We let $F^p (G, A, \af)$ be the crossed product obtained
by taking ${\mathcal{R}}$ to be the family ${\mathcal{R}}^p (G, A, \af)$
of all nondegenerate $\sm$-finite contractive covariant \rpn{s}.
We call it the {\emph{full $L^p$~operator crossed product}}
of $(G, A, \af).$
\item\label{D-CrPrd-r}
We let $F^p_{\mathrm{r}} (G, A, \af)$ be the crossed product obtained
by taking ${\mathcal{R}}$ to be the family
${\mathcal{R}}^p_{\mathrm{r}} (G, A, \af)$
of all regular covariant \rpn{s}
coming from nondegenerate $\sm$-finite contractive \rpn{s} of~$A.$
We call it the {\emph{reduced $L^p$~operator crossed product}}
of $(G, A, \af).$
\end{enumerate}
\end{dfn}

\begin{lem}\label{L_3916_Exist}
Let $p \in [1, \I),$
let $G$ be a second countable locally compact group,
and let $(G, A, \af)$ be a nondegenerately $\sm$-finitely representable
isometric $G$-$L^p$~operator algebra.
(See Definition~\ref{D:Action} and Definition~\ref{D-LpRep}.)
Then the crossed products $F^p (G, A, \af)$
and $F^p_{\mathrm{r}} (G, A, \af)$ of Definition~\ref{D-CrPrd} exist
as in Definition~3.2 of~\cite{DDW}.
\end{lem}

\begin{proof}
We must show that ${\mathcal{R}}^p (G, A, \af)$
and ${\mathcal{R}}^p_{\mathrm{r}} (G, A, \af)$
are uniformly bounded in the sense of Definition~3.1 of~\cite{DDW},
and are nonempty.
Uniform boundedness is obvious.
(In the notation of~\cite{DDW},
take $C = 1$ and $\nu (g) = 1$ for all $g \in G.$)
Lemma~\ref{L-ExistRegRep} shows that
${\mathcal{R}}^p_{\mathrm{r}} (G, A, \af) \neq \E.$
Then ${\mathcal{R}}^p (G, A, \af) \neq \E$
because
${\mathcal{R}}^p_{\mathrm{r}} (G, A, \af)
 \subset {\mathcal{R}}^p (G, A, \af).$
\end{proof}

Here is what happens for \ca{s}.
We are grateful to David Blecher
for pointing out the reference~\cite{BM}.

\begin{prp}\label{P-CompCStar}
Let $G$ be a locally compact group,
let $A$ be a \ca,
and let $\af \colon G \to \Aut (A)$ be an isometric action.
Then $A$  is nondegenerately $\sm$-finitely representable
and $\af_g$ is a *-automorphism for all $g \in G.$
Moreover, if $G$ is second countable,
then $F^2 (G, A, \af) = C^* (G, A, \af)$
and $F^2_{\mathrm{r}} (G, A, \af) = C^*_{\mathrm{r}} (G, A, \af).$
\end{prp}

\begin{proof}
To prove that $A$ is nondegenerately $\sm$-finitely representable,
we need to know that there are arbitrarily
large Hilbert spaces of the form $L^2 (X, \mu)$
for $\sm$-finite measures~$\mu.$
For a sufficiently large set~$S,$
take $X = [0, 1]^S$ and take $\mu$
to be the product of copies of Lebesgue measure.
By Proposition A.5.8 of~\cite{BM},
every contractive representation of $A$ on a Hilbert space
is necessarily a *-\hm.
It follows immediately that
$\af_g$ is a *-automorphism for all $g \in G,$
and, if $G$ is second countable, that
$F^p_{\mathrm{r}} (G, A, \af) = C^*_{\mathrm{r}} (G, A, \af).$
To get $F^p (G, A, \af) = C^* (G, A, \af),$
we also need the well known fact that
every isometric bijection on a Hilbert space is unitary.
\end{proof}

The following two theorems state parts of the universal
properties of our two crossed products.

\begin{thm}\label{T-UPropFull}
Let $p \in [1, \I),$
let $G$ be a second countable locally compact group,
and let $(G, A, \af)$ be a nondegenerately $\sm$-finitely representable
isometric $G$-$L^p$~operator algebra.
(See Definition~\ref{D:Action} and Definition~\ref{D-LpRep}.)
\begin{enumerate}
\item\label{T-UPropFull-1}
There is a \hm{}
\[
\io \colon C_{\mathrm{c}} (G, A, \af) \to F^p (G, A, \af)
\]
such that whenever $\XBM$ is a $\sm$-finite \msp{} and
$(v, \pi)$ is a nondegenerate contractive covariant \rpn{}
of $(G, A, \af)$ on $L^p (X, \mu),$
there is a unique contractive
\rpn{} $\rh_{v, \pi} \colon F^p (G, A, \af) \to \LLp$
such that $\rh_{v, \pi} \circ \io = v \ltimes \pi.$
\item\label{T-UPropFull-2}
For every $a \in F^p (G, A, \af),$
we have
\begin{align*}
\| a \|
& = \sup \big( \big\{ \| \rh_{v, \pi} (a) \| \colon
    {\mbox{$(v, \pi)$ is a nondegenerate \sft}}
      \\
& \hspace*{5em}
     {\mbox{contractive covariant \rpn{} of $(G, A, \af)$}}
            \big\} \big).
\end{align*}
\item\label{T-UPropFull-3}
The \hm\  $\io$ of~(\ref{T-UPropFull-1}) has dense range.
\end{enumerate}
\end{thm}

\begin{proof}
Let ${\mathcal{R}}^p (G, A, \af)$ be as in
Definition \ref{D-CrPrd}(\ref{D-CrPrd-f}),
and abbreviate it to~${\mathcal{R}}.$
The map we call~$\io$
comes from Definition~3.2 of~\cite{DDW},
in which the crossed product is defined to be a certain
Hausdorff completion
of $C_{\mathrm{c}} (G, A, \af),$
and Lemma~\ref{L_3916_Exist}.
(This map is called $q^{\mathcal{R}}$
in the discussion after Remark~3.3 of~\cite{DDW}.)
By construction,
it has dense range.
Existence of $\rh_{v, \pi}$
follows from the discussion after Remark~3.3 of~\cite{DDW}.
Uniqueness is clear from the fact that $\io$ has dense range.
Part~(\ref{T-UPropFull-2}) follows from
Proposition~3.4 of~\cite{DDW}.
\end{proof}

\begin{thm}\label{T-UPropRed}
Let $p \in [1, \I),$
let $G$ be a second countable locally compact group,
and let $(G, A, \af)$ be an isometric $G$-$L^p$~operator algebra
which is nondegenerately $\sm$-finitely representable
(Definition~\ref{D:Action}
and Definition \ref{D-LpRep}).
\begin{enumerate}
\item\label{T-UPropRed-1}
There is a \hm{}
\[
\io_{\mathrm{r}} \colon
   C_{\mathrm{c}} (G, A, \af) \to F^p_{\mathrm{r}} (G, A, \af)
\]
such that whenever $\XBM$ is a $\sm$-finite \msp{} and
$(v, \pi)$ is a contractive regular covariant \rpn{} of $(G, A, \af)$
on $L^p (X, \mu),$
there is a unique contractive
\rpn{} $\rh_{v, \pi} \colon F^p_{\mathrm{r}} (G, A, \af) \to \LLp$
such that $\rh_{v, \pi} \circ \io_{\mathrm{r}} = v \ltimes \pi.$
\item\label{T-UPropRed-2}
For every $a \in F^p_{\mathrm{r}} (G, A, \af),$
we have
\begin{align*}
\| a \|
& = \sup \big( \big\{ \| \rh_{v, \pi} (a) \| \colon
    {\mbox{$(v, \pi)$ is a nondegenerate \sft}}
      \\
& \hspace*{4em}
     {\mbox{contractive regular covariant \rpn{} of $(G, A, \af)$}}
            \big\} \big).
\end{align*}
\item\label{T-UPropRed-3}
The \hm\  $\io_{\mathrm{r}}$ of~(\ref{T-UPropFull-1}) has dense range.
\end{enumerate}
\end{thm}

\begin{proof}
The proof is essentially the same as that of Theorem~\ref{T-UPropFull}.
\end{proof}

\begin{dfn}\label{D-IndRpn}
We refer to the \rpn{s} $\rh_{v, \pi}$
in Theorem~\ref{T-UPropFull} and Theorem~\ref{T-UPropRed}
as the {\emph{representations induced by the covariant
representation $(v, \pi).$}}
When no confusion can arise,
we also denote them by $v \ltimes \pi.$
If $(v, \pi)$ is a regular covariant representation,
we refer to $v \ltimes \pi,$
as a \rpn{} of any of
$C_{\mathrm{c}} (G, A, \af),$
$F^p (G, A, \af),$
and $F^p_{\mathrm{r}} (G, A, \af),$
as a {\emph{regular representation}}
of the appropriate algebra.
A {\emph{contractive regular representation}}
is a regular representation for which the original \rpn{} of~$A$
is contractive.
\end{dfn}

It is clear that separability and $\sm$-finiteness
of $v \ltimes \pi$ are the same as the corresponding
properties of~$\pi.$
If $(v, \pi)$ is a regular covariant representation,
derived from a \rpn{} $\pi_0$ of~$A,$
then, by Lemma~\ref{L-ExistRegRep},
they are the same as the corresponding properties for $\pi_0.$

\begin{lem}\label{L_3916_IntIsNdg}
Let $p \in [1, \I),$
let $G$ be a second countable locally compact group,
and let $(G, A, \af)$ be a nondegenerately $\sm$-finitely representable
isometric $G$-$L^p$~operator algebra.
(See Definition~\ref{D:Action} and Definition~\ref{D-LpRep}.)
\begin{enumerate}
\item\label{L_3916_IntIsNdg_Crr}
Let $(v, \pi)$ be a covariant representation of $(G, A, \af)$
in which $\pi$ is nondegenerate.
Then $v \ltimes \pi$ is nondegenerate.
\item\label{L_3916_IntIsNdg_Reg}
Let $\pi_0$ be a nondegenerate \rpn{} of~$A,$
and let $(v, \pi)$ be
the associated regular covariant representation of $(G, A, \af).$
If $\pi_0$ is nondegenerate,
then $v \ltimes \pi$ is nondegenerate.
\end{enumerate}
\end{lem}

\begin{proof}
Part~(\ref{L_3916_IntIsNdg_Crr})
follows from Proposition 5.5(i) of~\cite{DDW}.
Part~(\ref{L_3916_IntIsNdg_Reg})
follows from part~(\ref{L_3916_IntIsNdg_Crr})
and Lemma \ref{L-ExistRegRep}(\ref{ExistRegRep-Ndg}).
\end{proof}

Under additional hypotheses,
there is a converse
to Lemma \ref{L_3916_IntIsNdg}(\ref{L_3916_IntIsNdg_Crr});
see Proposition 5.5(i) of~\cite{DDW}.
There is also,
again under additional hypotheses,
an analog of the statement for \ca{s}
that every nondegenerate \rpn{} of $F^p (G, A, \af)$
comes from a nondegenerate covariant \rpn{} of $(G, A, \af),$
and a related statement for $F^p_{\mathrm{r}} (G, A, \af).$
See Theorem~8.1 and Theorem~9.1 of~\cite{DDW}.
We do not reformulate these results for our situation here.

We also find it convenient to use the $L^1$ crossed product.
(As far as we know,
in general this crossed product need not be the same
as either of the $L^1$~operator crossed products.
However, for the special case $A = \C,$
it agrees with both the $L^1$~operator crossed products.
See Proposition~\ref{P_3217_L1OpGpAlg} below.)

\begin{ntn}\label{N-L1CrPrd}
Let $A$ be a Banach algebra,
let $G$ be a locally compact group with left Haar measure~$\nu,$
and let $\af \colon G \to \Aut (A)$
be an isometric action of $G$ on~$A.$
Then $L^1 (G, A, \af)$ denotes the convolution algebra
consisting of all $L^1$~functions from $G$ to~$A.$
\end{ntn}

The space $L^1 (G, A, \af)$
is as defined in the discussion on page~630 of~\cite{Sw2},
taking $\sm (g) = 1$ for all $g \in G.$
The proof that it is a Banach algebra under convolution
follows from the estimates in the proof of Theorem 2.2.6 of~\cite{Sw2}.
Alternatively, see Appendix~B of~\cite{Wl},
and note that everything there not involving the *~operation
goes through for an isometric action of $G$ on a Banach algebra.
The important point is that
$L^1 (G, A, \af)$ is the completion of $C_{\mathrm{c}} (G, A, \af)$
in the norm $\| a \|_1 = \int_G \| a (g) \| \, d \nu (g).$
A third source,
with a detailed proof, is Proposition IV.1.5 of~\cite{Jn2}.
(See Lemma IV.1.1 and Definition IV.1.4 of~\cite{Jn2}
for the notation.
Actions are assumed isometric;
see Definition I.2.2 of~\cite{Jn1}.)

The following proposition
and its corollary
will not be used.
They are included to show
one way in which $L^p$~operator crossed products
are well behaved.

\begin{prp}\label{P-RegCcInj}
Let $p \in [1, \I),$
let $G$ be a locally compact group with left Haar measure~$\nu,$
and let $(G, A, \af)$ be an isometric $G$-$L^p$~operator algebra.
Let $\XBM$ be a \msp, let $\pi_0 \colon A \to \LLp$
be an injective \rpn,
and let $(v, \pi)$ be the associated regular \rpn,
as in \Def{D-RegRep}.
Then the map
$v \ltimes \pi \colon
   C_{\mathrm{c}} (G, A, \af) \to
   L \big( L^p ( G \times X, \, \nu \times \mu) \big)$
of Notation~\ref{N-CcGA}
is injective.
\end{prp}

\begin{proof}
The proof is essentially the same as that of Lemma~2.26 of~\cite{Wl}.
(However, we avoid the initial translation simplification there,
since we don't need the properties of the translation map
for anything else.)

Thus, let $a \in C_{\mathrm{c}} (G, A, \af)$ be nonzero.
Choose $g_0 \in G$ such that $a (g_0) \neq 0.$
Then $\big( \pi_0 \circ \af_{g_0}^{-1} \big) (a (g_0) ) \neq 0.$
Choose $\xi_0 \in L^p (X, \mu)$ and $\om_0 \in L^p (X, \mu)'$
such that
\[
\om_0 \big( \big( \pi_0 \circ \af_{g_0}^{-1} \big) (a (g_0) ) \xi_0 \big)
 \neq 0.
\]
Define a \cfn{} $c \colon G \times G \to \C$ by
\[
c (g, h)
  = \om_0 \big( \big( \pi_0 \circ \af_{g}^{-1} \big) (a (h) ) \xi_0 \big)
\]
for $g, h \in G.$
Then $c (g_0, g_0) \neq 0.$
Therefore there is an open set $U \S G$
such that $g_0 \in U$ and
such that for all $g, h \in U,$
we have
$| c (g, h) - c (g_0, g_0) | < \tfrac{1}{2} | c (g_0, g_0) |.$
Choose an open set $W \S G$ such that $1 \in W$
and $W^2 \S U^{-1} g_0.$
Choose $f \in C_{\mathrm{c}} (G)$ such that $f \geq 0,$
$\supp (f) \S W,$
and $f (1) > 0.$
Let $\xi \in L^p (G \times X, \, \nu \times \mu)$
be given by the function $g \mapsto f (g) \xi_0$
from $G$ to $L^p (X, \mu).$
Let $\om \in C_{\mathrm{c}} (G, L^p (X, \mu)' )$
be the function $\om (g) = f (g^{-1} g_0) \om_0$
for $g \in G,$
and identify $\om$ with an element of
$L^p (G \times X, \, \nu \times \mu)'$ as in Remark~\ref{R-RegRepCP}.
Set
\[
\ld = \int_{G \times G}
   f (g^{-1} g_0) f (h^{-1} g) \, d \nu (h) d \nu (g).
\]
Two changes of variables,
first from $g$ to $g_0 g,$
then from $h$ to $g_0 h,$
show that
\[
\ld = \int_{G \times G}
   f (g^{-1}) f (h^{-1} g) \, d \nu (h) d \nu (g),
\]
which is clearly strictly positive.

Suppose $g, h \in G$
and $f (g^{-1} g_0) f (h^{-1} g) \neq 0.$
Then $g^{-1} g_0 \in W$ and $h^{-1} g \in W.$
So
\[
g^{-1} g_0 = g^{-1} g_0 \cdot 1 \in W^2 \S U^{-1} g_0
\andeqn
h^{-1} g_0 = h^{-1} g \cdot g^{-1} g_0 \in W^2 \S U^{-1} g_0.
\]
It follows that $g, h \in U,$
so
$| c (g, h) - c (g_0, g_0) | < \tfrac{1}{2} | c (g_0, g_0) |.$
Using this result at the second step,
and the formula~(\ref{Eq:WeakIntRegRep}) in Remark~\ref{R-RegRepCP}
at the first step,
we get
\begin{align*}
& \big| \om \big( (v \ltimes \pi) (a) \xi \big)
       - \ld c (g_0, g_0) \big|
\\
& \hspace*{3em} {\mbox{}}
 \leq \int_{G \times G}
   f (g^{-1} g_0) f (h^{-1} g) | c (g, h) - c (g_0, g_0) |
         \, d \nu (h) d \nu (g)
  \leq \frac{\ld | c (g_0, g_0) |}{2}.
\end{align*}
Since $\ld > 0$ and $c (g_0, g_0) \neq 0,$
we deduce that $(v \ltimes \pi) (a) \neq 0.$
\end{proof}

\begin{cor}\label{C-CcRedInj}
Let $p \in [1, \I),$
let $G$ be a second countable locally compact group,
and let $(G, A, \af)$ be a nondegenerately $\sm$-finitely representable
isometric $G$-$L^p$~operator algebra.
(See Definition~\ref{D:Action} and Definition~\ref{D-LpRep}.)
Then the map
$\io \colon
   C_{\mathrm{c}} (G, A, \af) \to F^p (G, A, \af)$
of Theorem~\ref{T-UPropFull}(\ref{T-UPropFull-1})
and the map
$\io_{\mathrm{r}} \colon
   C_{\mathrm{c}} (G, A, \af) \to F^p_{\mathrm{r}} (G, A, \af)$
of Theorem~\ref{T-UPropRed}(\ref{T-UPropRed-1})
are both injective.
\end{cor}

\begin{proof}
This is immediate from Proposition~\ref{P-RegCcInj}
and Theorem \ref{T-UPropFull}(\ref{T-UPropFull-2}) (for $\io$)
and Theorem \ref{T-UPropRed}(\ref{T-UPropRed-2})
(for $\io_{\mathrm{r}}$).
\end{proof}

\begin{lem}\label{L-CompOfCrPrd}
Let $p \in [1, \I),$
let $G$ be a second countable locally compact group,
and let $(G, A, \af)$ be a nondegenerately $\sm$-finitely representable
isometric $G$-$L^p$~operator algebra.
(See Definition~\ref{D:Action} and Definition~\ref{D-LpRep}.)
Let $\io \colon
   C_{\mathrm{c}} (G, A, \af) \to F^p (G, A, \af)$
be as in Theorem \ref{T-UPropFull}(\ref{T-UPropFull-1})
and let
$\io_{\mathrm{r}} \colon
   C_{\mathrm{c}} (G, A, \af) \to F^p_{\mathrm{r}} (G, A, \af)$
be as in
Theorem~\ref{T-UPropRed}(\ref{T-UPropRed-1}).
Then there are unique
\ct\  homomorphisms
\[
\kp \colon L^1 (G, A, \af) \to F^p (G, A, \af)
\andeqn
\kp_{\mathrm{r}} \colon F^p (G, A, \af) \to F^p_{\mathrm{r}} (G, A, \af)
\]
such that $\kp |_{C_{\mathrm{c}} (G, A, \af)} = \io$
and $\kp_{\mathrm{r}} \circ \io = \io_{\mathrm{r}}.$
Moreover,
$\kp$~and $\kp_{\mathrm{r}}$ are contractive and have dense ranges.
\end{lem}

\begin{proof}
Once everything else is done,
the statements about uniqueness and dense ranges follow from
Theorem~\ref{T-UPropFull}(\ref{T-UPropFull-3}),
Theorem~\ref{T-UPropRed}(\ref{T-UPropRed-3}),
and density of $C_{\mathrm{c}} (G, A, \af)$ in $L^1 (G, A, \af).$
The remaining assertions about~$\kp_{\mathrm{r}}$
are immediate from Theorem~\ref{T-UPropFull}(\ref{T-UPropFull-2})
and Theorem~\ref{T-UPropRed}(\ref{T-UPropRed-2}),
because every contractive regular covariant \rpn{} %
is a contractive covariant \rpn.

It remains to prove that $\kp$ exists and is contractive.
By Theorem~\ref{T-UPropFull}(\ref{T-UPropFull-2}),
it suffices to show that if $\XBM$ is a \msp,
$(v, \pi)$ is a contractive covariant \rpn{} of $(G, A, \af)$
on $L^p (X, \mu),$
and $a \in C_{\mathrm{c}} (G, A, \af),$
then $\| (v \ltimes \pi) (a) \| \leq \| a \|_1.$
We let $\xi \in L^p (X, \mu)$ and $\om \in L^p (X, \mu)',$
and use~(\ref{Eq:DefOfInt}) in Notation~\ref{N-CcGA}.
Since $\pi$ and $v$ are contractive, we get
\[
\big| \om \big( (v \ltimes \pi) (a) \xi \big) \big|
 \leq
   \int_G \big| \om \big( \pi (a (g)) v_g \xi \big) \big| \, d \nu (g)
 \leq \| \om \| \cdot \| a \|_1 \cdot \| \xi \|.
\]
The required estimate follows.
\end{proof}

When $p = 2$ and $A$ is a \ca,
the map
$\kp_{\mathrm{r}} \colon
   C^* (G, A, \af) \to C^*_{\mathrm{r}} (G, A, \af)$
is always surjective
(this follows from Theorem 1.5.7 of~\cite{Pd1}),
and is an isometric isomorphism if $G$ is amenable
(Theorem 7.7.7 of~\cite{Pd1}).
Moreover, for $A = \C,$
if $\kp_{\mathrm{r}}$ is an isomorphism,
then $G$ is amenable.
(See Theorem 7.3.9 of~\cite{Pd1}.)
We have not investigated whether any of these facts
carries over to $L^p$~operator crossed products for $p \neq 2,$
except for some very special cases.
We will see in Remark~\ref{R-FGpCP} below that if $G$ is finite
then $\kp_{\mathrm{r}}$ is bijective.
Also, the case $p = 1$ is apparently special,
at least when considering the $L^p$~operator group algebras,
which we denote by
$F^p (G)$ and $F^p_{\mathrm{r}} (G).$

\begin{prp}\label{P_3217_L1OpGpAlg}
Let $G$ be any second countable locally compact group.
Then the maps
\[
\kp \colon L^1 (G) \to F^1 (G)
\andeqn
\kp_{\mathrm{r}} \colon F^1 (G) \to F^1_{\mathrm{r}} (G)
\]
of Lemma~\ref{L-CompOfCrPrd}
are isometric isomorphisms of Banach algebras.
\end{prp}

\begin{proof}
By Lemma~\ref{L-CompOfCrPrd},
it suffices to prove that
$\| (\kp_{\mathrm{r}} \circ \kp ) (a) \| \geq \| a \|$
for all $a \in L^1 (G).$
For this,
by Theorem \ref{T-UPropRed}(\ref{T-UPropRed-2})
it suffices to find one \rpn{} $\pi$ of $\C$
on an $L^1$~space $L^1 (X, \mu)$
such that, if $(v, \pi)$ is the associated regular \rpn,
as in \Def{D-RegRep},
then the map
$v \ltimes \pi$
has the property that $\| (v \ltimes \pi) (a) \| \geq \| a \|.$
We take $X$ to be a one point space
and $\mu$ to be counting measure.
Then $v \ltimes \pi \colon L^1 (G) \to L ( L^1 (G))$
sends $a \in L^1 (G)$ to left convolution by~$a.$
This operator has norm $\| a \|$ because
$L^1 (G)$ has an approximate identity consisting
of elements of norm~$1.$
\end{proof}

In particular,
the map $\kp_{\mathrm{r}} \colon F^1 (G) \to F^1_{\mathrm{r}} (G)$
is an isomorphism regardless of whether $G$ is amenable.

In the rest of this section,
we define the dual action on an $L^p$~operator crossed product
(both full and reduced)
when the group is abelian.
We use the dual action for the case $G = \Z$
in Section~\ref{Sec_PV}
to obtain properties of a smooth subalgebra of the crossed product.

\begin{dfn}\label{D-DualCc}
Let $A$ be a Banach algebra,
let $G$ be a locally compact abelian group,
and let $\af \colon G \to \Aut (A)$ be an action of $G$ on~$A.$
For $\ta \in {\widehat{G}},$
we define
\[
{\widehat{\af}}_{\ta} \colon
   C_{\mathrm{c}} (G, A, \af) \to C_{\mathrm{c}} (G, A, \af)
\]
by
\[
{\widehat{\af}}_{\ta} (a) (g) = {\ov{\ta (g)}} a (g)
\]
for $a \in C_{\mathrm{c}} (G, A, \af)$ and $g \in G.$
The map $\ta \mapsto {\widehat{\af}}_{\ta}$ is called
the
{\emph{dual action of ${\widehat{G}}$ on $C_{\mathrm{c}} (G, A, \af)$}}.
\end{dfn}

The choice ${\ov{\ta (g)}}$ rather than $\ta (g)$
seems to be more common.
It agrees with the convention
in~\cite{Tki} (see the beginning of Section~3 there)
and~\cite{Wl} (see the beginning of Section~7.1 there),
but disagrees with the convention in~\cite{Pd1}
(see Proposition 7.8.3 there).

\begin{lem}\label{L-DualIsAct}
In the situation of \Def{D-DualCc},
the map $\ta \mapsto {\widehat{\af}}_{\ta}$
is a \hm\  from ${\widehat{G}}$
to the group of algebraic automorphisms of $C_{\mathrm{c}} (G, A, \af).$
\end{lem}

\begin{proof}
This is a well known computation.
\end{proof}

\begin{thm}\label{T-DualL1}
Let $A$ be a Banach algebra,
let $G$ be a locally compact abelian group,
and let $\af \colon G \to \Aut (A)$ be a isometric action of $G$ on~$A.$
Then there exists a unique
continuous isometric action,
also called $\ta \mapsto {\widehat{\af}}_{\ta},$
of ${\widehat{G}}$ on $L^1 (G, A, \af)$
such that
the inclusion of $C_{\mathrm{c}} (G, A, \af)$ in $L^1 (G, A, \af)$
is equivariant.
\end{thm}

\begin{proof}
When $\af$ is isometric,
it is easy to see that for all $a \in C_{\mathrm{c}} (G, A, \af)$
we have $\| {\widehat{\af}}_{\ta} (a) \|_1 = \| a \|_1.$
Therefore $\ta \mapsto {\widehat{\af}}_{\ta}$ extends uniquely
to a group \hm\  from ${\widehat{G}}$
to the isometric automorphisms of $L^1 (G, A, \af).$
One checks that if $a \in C_{\mathrm{c}} (G, A, \af)$
then $\ta \mapsto {\widehat{\af}}_{\ta} (a)$
is \ct\  in $\| \cdot \|_1.$
Since ${\widehat{\af}}_{\ta}$ is isometric
for all $\ta \in {\widehat{G}},$
a standard $\frac{\ep}{3}$~argument
shows that $\ta \mapsto {\widehat{\af}}_{\ta} (a)$
is \ct\  for all $a \in L^1 (G, A, \af).$
\end{proof}

\begin{thm}\label{T-DualpFull}
Let $p \in [1, \I),$
let $G$ be a second countable locally compact abelian group,
and let $(G, A, \af)$ be a nondegenerately $\sm$-finitely representable
isometric $G$-$L^p$~operator algebra.
(See Definition~\ref{D:Action} and Definition~\ref{D-LpRep}.)
Then there exists a unique
continuous isometric action,
also called $\ta \mapsto {\widehat{\af}}_{\ta},$
of ${\widehat{G}}$ on $F^p (G, A, \af)$
such that
the inclusion
$\io \colon C_{\mathrm{c}} (G, A, \af) \to F^p (G, A, \af)$
is equivariant.
\end{thm}

\begin{proof}
For a Banach space~$E,$
a \rpn{} $v$ of $G$ on~$E,$
and $\ta \in {\widehat{G}},$
we define a \rpn{} $\ta v \colon G \to L (E)$
by $(\ta v)_g = \ta (g) v_g$ for $g \in G.$
If $(v, \pi)$ is
a nondegenerate \sft{} contractive covariant \rpn{}
of $(G, A, \af),$
then one easily checks that $(\ta v, \, \pi)$ is too.
Since ${\ov{\ta}} ( \ta v) = v,$
the map $(v, \pi) \to ({\ov{\ta}} v, \, \pi)$ is a bijection
on the collection of
nondegenerate \sft{} contractive covariant \rpn{s}
of $(G, A, \af)$
on spaces of the form $L^p (X, \mu).$
One checks that
\[
({\ov{\ta}} v \ltimes \pi) (a)
  = (v \ltimes \pi) \big( {\widehat{\af}}_{\ta} (a) \big)
\]
for all $a \in C_{\mathrm{c}} (G, A, \af).$
It now follows from
Theorem~\ref{T-UPropFull}(\ref{T-UPropFull-2})
that $\| \io \big( {\widehat{\af}}_{\ta} (a) \big) \| = \| \io (a) \|$
for all $a \in C_{\mathrm{c}} (G, A, \af).$
By Theorem~\ref{T-UPropFull}(\ref{T-UPropFull-3}),
${\widehat{\af}}_{\ta}$
extends to an isometric endomorphism of $F^p (G, A, \af).$
These endomorphisms are easily seen to be an action of
${\widehat{G}}$ on $F^p (G, A, \af)$;
in particular,
they are automorphisms.
Since $\kp \colon L^1 (G, A, \af) \to F^p (G, A, \af)$
is contractive,
it follows from Theorem~\ref{T-DualL1}
that for $a \in L^1 (G, A, \af),$
the map
$\ta \mapsto {\widehat{\af}}_{\ta} ( \io (a)) \in F^p (G, A, \af)$
is \ct.
Since ${\widehat{\af}}_{\ta}$ is isometric
for all $\ta \in {\widehat{G}},$
continuity of $\ta \mapsto {\widehat{\af}}_{\ta} (b)$
for all $b \in F^p (G, A, \af)$
follows from density of the range of $\io$
by a standard $\frac{\ep}{3}$~argument.
\end{proof}

For the dual action on the reduced crossed product,
we need a lemma.

\begin{lem}\label{L-RedOneRep}
Let $p \in [1, \I),$
let $G$ be a second countable locally compact group,
and let $(G, A, \af)$ be a separable nondegenerately representable
isometric $G$-$L^p$~operator algebra.
(See Definition~\ref{D:Action} and Definition~\ref{D-LpRep}.)
Then there exists a \sfm{} $\XBM$
and a nondegenerate isometric \rpn{} $\pi_0 \colon A \to \LLp$
such that,
with $(v, \pi)$ being the associated regular covariant \rpn,
the \rpn{} $v \ltimes \pi$
is nondegenerate and isometric on $F^p_{\mathrm{r}} (G, A, \af).$
\end{lem}

\begin{proof}
Since $G$ is second countable and $A$ is separable,
it is easily checked that $F^p_{\mathrm{r}} (G, A, \af)$
is separable.
Let $S \S F^p_{\mathrm{r}} (G, A, \af)$ be a countable dense subset.

For $\ep > 0$ and $b \in F^p_{\mathrm{r}} (G, A, \af),$
use Theorem~\ref{T-UPropRed}(\ref{T-UPropRed-2})
to choose a \sfm{}
$(X_{b, \ep}, \, \cB_{b, \ep}, \, \mu_{b, \ep})$
and a nondegenerate contractive \rpn{}
\[
\pi_0^{b, \ep} \colon
   A \to L \big( L^p (X_{b, \ep}, \, \mu_{b, \ep} ) \big)
\]
such that,
with $\big( v^{b, \ep}, \, \pi^{b, \ep} \big)$
being the associated regular covariant \rpn,
we have
\[
\big\| \big( v^{b, \ep} \ltimes \pi^{b, \ep} \big) (b) \big\|
  > \| b \| - \ep.
\]
Further choose a \sfm{} $\XBM$
and a nondegenerate isometric \rpn{} $\sm_0 \colon A \to \LLp,$
and let $(u, \sm)$ be the associated regular covariant \rpn.
Let $\pi_0$ be the $L^p$~direct sum
(Definition~\ref{D-LpDirectSum})
of $\sm_0$ and all $\pi_0^{b, \ep}$
for $\ep \in \left\{ 1, \frac{1}{2}, \frac{1}{3}, \ldots \right\}$
and $b \in S.$
Then $\pi_0$ is nondegenerate (by Lemma~\ref{L_3914_SmNdg}), isometric,
and \sft{}
(being a countable direct sum of \sft{} \rpn{s}).
Also, if $(v, \pi)$ is the associated regular covariant \rpn,
then $v \ltimes \pi$ is the $L^p$~direct sum
of $u \ltimes \sm$ and all $v^{b, \ep} \ltimes \pi^{b, \ep}$
for $\ep \in \left\{ 1, \frac{1}{2}, \frac{1}{3}, \ldots \right\}$
and $b \in S.$
These representations are all nondegenerate
by Lemma~\ref{L_3916_IntIsNdg}(\ref{L_3916_IntIsNdg_Reg}).
Therefore $v \ltimes \pi$ is nondegenerate (by Lemma~\ref{L_3914_SmNdg}),
and is clearly isometric.
\end{proof}

\begin{thm}\label{T-DualpRed}
Let $p \in [1, \I),$
let $G$ be a second countable locally compact abelian group,
and let $(G, A, \af)$ be a separable nondegenerately representable
isometric $G$-$L^p$~operator algebra.
(See Definition~\ref{D:Action} and Definition~\ref{D-LpRep}.)
Then there exists a unique
continuous isometric action,
also called $\ta \mapsto {\widehat{\af}}_{\ta},$
of ${\widehat{G}}$ on $F^p_{\mathrm{r}} (G, A, \af)$
such that
the inclusion
$\io_{\mathrm{r}} \colon
 C_{\mathrm{c}} (G, A, \af) \to F^p_{\mathrm{r}} (G, A, \af)$
is equivariant.
\end{thm}

\begin{proof}
Let
\[
\pi_0 \colon A \to \LLp,
\,\,\,\,\,\,\,
\pi \colon A
   \to L \big( L^p (G \times X, \, \nu \times \mu) \big),
\]
and
\[
v \colon G
   \to L \big( L^p (G \times X, \, \nu \times \mu) \big)
\]
be as in \Lem{L-RedOneRep}.
For $\ta \in {\widehat{G}},$
define
\[
w_{\ta} \in L \big( L^p (G \times X, \, \nu \times \mu) \big)
\]
by
\[
(w_{\ta} \xi) (g, x)
  = {\ov{\ta (g) }} \xi (g, x)
\]
for $\xi \in L^p (G \times X, \, \nu \times \mu),$ $g \in G,$
and $x \in X.$
Then $w_{\ta}$ is an isometric bijection because it is
multiplication by a function with absolute value~$1.$
It is easy to check that $w_{\ta}$ commutes with $\pi (a)$
for all $\ta \in {\widehat{G}}$ and $a \in A,$
and that
$w_{\ta} v_g w_{\ta}^{-1} = {\ov{\ta (g) }} v_g$
for all $\ta \in {\widehat{G}}$ and $g \in G.$

It follows from the definition of $v \ltimes \pi$
that
\[
w_{\ta} ((v \ltimes \pi) \circ \io_{\mathrm{r}}) (b) w_{\ta}^{-1}
  = ((v \ltimes \pi) \circ \io_{\mathrm{r}})
          \big( {\widehat{\af}}_{\ta} (b) \big)
\]
for all $\ta \in {\widehat{G}}$ and $b \in C_{\mathrm{c}} (G, A, \af).$
Clearly $a \mapsto w_{\ta} a w_{\ta}^{-1}$
defines an isometric action of~${\widehat{G}}$
(with the discrete topology) on $F^p_{\mathrm{r}} (G, A, \af).$
The proof of continuity in the usual topology of~${\widehat{G}}$
is the same as in the proof of Theorem~\ref{T-DualL1}.
\end{proof}

\section{Reduced $L^p$~operator crossed products by countable discrete
   groups}\label{Sec_DiscCP}

\indent
In this section, we assume that the group~$G$ is discrete
and, starting with Lemma~\ref{L:FGpCP}, countable.
The Haar measure~$\nu$ on~$G$ will always be counting measure.
We will require that $A$ be separable,
since we use Lemma~\ref{L-RedOneRep}.
We define a Banach conditional expectation
from the reduced $L^p$~operator crossed product
$F^p_{\mathrm{r}} (G, A, \af)$ to~$A,$
and use it to define coordinates of elements
of $F^p_{\mathrm{r}} (G, A, \af),$
in a manner similar to what is done
for reduced crossed product \ca{s}.

We begin by writing the formula for a regular \rpn{} %
of $C_{\mathrm{c}} (G, A, \af)$ in a more convenient way.

\begin{ntn}\label{N-Sum}
Let $A$ be a Banach algebra,
let $G$ be a discrete group,
and let $\af \colon G \to \Aut (A)$ be an action of $G$ on~$A.$
Let ${\widetilde{A}}$ be the unitization of~$A$
(we do not add a new identity if $A$ is already unital),
and identify $C_{\mathrm{c}} (G)$ with a subspace
of $C_{\mathrm{c}} \big( G, {\widetilde{A}}, \af \big)$
in the obvious way.
We let $u_g$ be the characteristic function of~$\{ g \},$
regarded as an element of
$C_{\mathrm{c}} \big( G, {\widetilde{A}}, \af \big).$
We can then write an element $a \in C_{\mathrm{c}} (G, A, \af)$
as a finite sum
$a = \sum_{g \in G} a_g u_g$
(using $a_g$ rather than $a (g)$).
\end{ntn}

\begin{ntn}\label{N_3920_sgtg}
Let $G$ be a countable set with counting measure~$\nu,$
and let $\XBM$ be a measure space.
For $g \in G$ define
\[
s_g \colon L^p (X, \mu) \to L^p (G \times X, \, \nu \times \mu)
\andeqn
t_g \colon L^p (G \times X, \, \nu \times \mu) \to L^p (X, \mu)
\]
as follows.
Set
\[
s_g (\et) (h)
 = \begin{cases}
   \et & h = g
        \\
   0 & h \neq g
\end{cases}
\]
for $\et \in L^p (X, \mu)$ and $h \in G.$
For $\xi \in L^p (G \times X, \, \nu \times \mu),$
define $t_g (\xi)$ to be the composition of $\xi$
with the map $x \mapsto (g, x)$ from $X$ to $G \times X.$
\end{ntn}

\begin{lem}\label{L_3920_sgsp}
Let the notation be as in Notation~\ref{N_3920_sgtg}.
Then
$s_g$ is a spatial isometry
in the sense of Definition~6.4 of~\cite{PhLp1},
$t_g$ is its reverse in the sense of Definition~6.13 of~\cite{PhLp1},
$t_g s_g = 1,$
and $s_g t_g$ is multiplication by $\ch_{ \{ g \} \times X}.$
\end{lem}

\begin{proof}
Everything is immediate.
\end{proof}

\begin{lem}\label{L-StructRR}
Let $p \in [1, \I),$
let $G$ be a countable discrete group,
and let $(G, A, \af)$ be an isometric $G$-$L^p$~operator algebra.
Let $\XBM$ be a \msp,
and let $\pi_0 \colon A \to \LLp$
be a contractive \rpn.
Let $(v, \pi)$ be the associated regular covariant \rpn{} %
as in \Def{D-RegRep}.
For $g \in G,$
let $s_g$ and $t_g$ be as in Notation~\ref{N_3920_sgtg}.
Then:
\begin{enumerate}
\item\label{L-StructRR-1}
For $a \in C_{\mathrm{c}} (G, A, \af),$
$\xi \in C_{\mathrm{c}} \big( G, \, L^p (X, \mu) \big),$
and $h \in G,$ we have
\[
\big( (v \ltimes \pi) (a) \xi \big) (h)
 = \sum_{g \in G}
     \pi_0 \big( \af_h^{-1} (a_g) \big) \big( \xi (g^{-1} h ) \big).
\]
\item\label{L-StructRR-2b}
If $c \in L \big( L^p (G \times X, \, \nu \times \mu) \big)$
satisfies $t_h c s_k = 0$ for all $h, k \in G,$
then $c = 0.$
\item\label{L-StructRR-3}
For $a \in C_{\mathrm{c}} (G, A, \af)$
and $h, k \in G,$
we have
\[
t_h (v \ltimes \pi) (a) s_k
 = \pi_0 \big( \af_h^{-1} ( a_{h k^{-1} }) \big).
\]
\end{enumerate}
\end{lem}

\begin{proof}
These statements are all calculations.
\end{proof}

\begin{lem}\label{L:FGpCP}
Let $p \in [1, \I),$
let $G$ be a countable discrete group,
and let $(G, A, \af)$ be a separable nondegenerately representable
isometric $G$-$L^p$~operator algebra.
(See Definition~\ref{D:Action} and Definition~\ref{D-LpRep}.)
Let $\| \cdot \|$
be the norm on $F^p (G, A, \af),$
restricted to $C_{\mathrm{c}} (G, A, \af)$;
let $\| \cdot \|_{\mathrm{r}}$
be the norm on $F^p_{\mathrm{r}} (G, A, \af),$
restricted to $C_{\mathrm{c}} (G, A, \af)$;
and let $\| \cdot \|_{\I}$
be the supremum norm.
Then for every $a \in C_{\mathrm{c}} (G, A, \af),$
we have
$\| a \|_{\I} \leq \| a \|_{\mathrm{r}} \leq \| a \| \leq \| a \|_1.$
\end{lem}

\begin{proof}
The middle and last parts of this inequality
follow from \Lem{L-CompOfCrPrd}.

We prove the first inequality.
Let $a \in C_{\mathrm{c}} (G, A, \af).$
Write $a = \sum_{g \in G} a_g u_g,$
as in Notation~\ref{N-Sum}.
Let $g \in G.$
Choose an isometric \rpn{}  $\pi_0 \colon A \to \LLp$
as in \Lem{L-RedOneRep}.
Following Notation~\ref{N_3920_sgtg},
use \Lem{L-StructRR}(\ref{L-StructRR-3})
at the second step to get
\[
\| a_g \|
  = \| \pi_0 (a_g) \|
  = \| t_1 (v \ltimes \pi) (a) s_{g^{-1}} \|
  \leq \| (v \ltimes \pi) (a) \|
  \leq \| a \|_{\mathrm{r}}.
\]
This completes the proof.
\end{proof}

\begin{rmk}\label{R-Ident}
Lemma~\ref{L:FGpCP} implies that
the map $a \mapsto a u_1,$
from $A$ to $F^p_{\mathrm{r}} (G, A, \af),$
is isometric.
We routinely identify $A$
with its image in $F^p_{\mathrm{r}} (G, A, \af)$ under this map,
thus treating it as a subalgebra of $F^p_{\mathrm{r}} (G, A, \af).$

We do the same with the full \cp\  $F^p (G, A, \af).$
\end{rmk}

\begin{rmk}\label{R-FGpCP}
Adopt the notation of \Lem{L:FGpCP},
and assume that $G$ is finite.
Then $\| \cdot \|_1$ is equivalent to $\| \cdot \|_{\I},$
and $C_{\mathrm{c}} (G, A, \af)$ is complete in both these norms.
It follows that the map
$\kp_{\mathrm{r}} \colon
   F^p (G, A, \af) \to F^p_{\mathrm{r}} (G, A, \af)$
of \Lem{L-CompOfCrPrd} is bijective,
and also that the maps
\[
\io \colon
   C_{\mathrm{c}} (G, A, \af) \to F^p (G, A, \af)
\andeqn
\io_{\mathrm{r}} \colon
   C_{\mathrm{c}} (G, A, \af) \to F^p_{\mathrm{r}} (G, A, \af)
\]
are bijective.
Unlike in the C*~situation,
it does not immediately follow that $\kp_{\mathrm{r}}$
is isometric,
and we have not tried to determine whether it is.
\end{rmk}

\begin{prp}\label{P:CondExpt}
Let $p \in [1, \I),$
let $G$ be a countable discrete group,
and let $(G, A, \af)$ be a separable nondegenerately representable
isometric $G$-$L^p$~operator algebra.
(See Definition~\ref{D:Action} and Definition~\ref{D-LpRep}.)
Then for each $g \in G,$
there is a linear map $E_g \colon F^p_{\mathrm{r}} (G, A, \af) \to A$
with $\| E_g \| \leq 1$
such that if
\[
a = \sum_{g \in G} a_g u_g \in C_{\mathrm{c}} (G, A, \af),
\]
then $E_g (a) = a_g.$
Moreover,
for every \rpn{} $\pi_0$ of~$A$ as in \Lem{L-StructRR},
with $v$ and $\pi$ as in \Lem{L-StructRR}
and $s_g$ and $t_g$ as in Notation~\ref{N_3920_sgtg},
we have
\[
t_h (v \ltimes \pi) (a) s_k
 = \pi_0 \big( \af_h^{-1} (E_{h k^{-1}} (a)) \big)
\]
for all $h, k \in G.$
\end{prp}

\begin{proof}
The first part is immediate from the first inequality
in Lemma~\ref{L:FGpCP}.
The last statement follows from
\Lem{L-StructRR}(\ref{L-StructRR-3}) by continuity.
\end{proof}

It follows that coefficients~$a_g$ of
elements $a \in F^p_{\mathrm{r}} (G, A, \af)$ make sense.
One does not get convergence
of $\sum_{g \in G} a_g u_g,$
since this already fails for \ca{s} and $p = 2$ is allowed
in the proposition.
However,
we can prove that,
as in the C*~case,
$a$ is uniquely determined by the coefficients~$a_g.$

\begin{prp}\label{P:Faithful}
Let $p \in [1, \I),$
let $G$ be a countable discrete group,
and let $(G, A, \af)$ be a separable nondegenerately representable
isometric $G$-$L^p$~operator algebra.
(See Definition~\ref{D:Action} and Definition~\ref{D-LpRep}.)
Let the maps $E_g \colon F^p_{\mathrm{r}} (G, A, \af) \to A$
be as in Proposition~\ref{P:CondExpt}.
Then:
\begin{enumerate}
\item\label{P:Faithful:G}
If $a \in F^p_{\mathrm{r}} (G, A, \af)$
and $E_g (a) = 0$ for all $g \in G,$
then $a = 0.$
\item\label{P:Faithful:Inj}
If $\XBM$ is a \sfm{}
and $\pi_0 \colon A \to \LLp$
is a nondegenerate contractive representation
such that the $L^p$~direct sum (Definition~\ref{D-LpDirectSum})
$\bigoplus_{g \in G} \pi_0 \circ \af_g$ is injective,
then the regular representation~$\sm$ of $F^p_{\mathrm{r}} (G, A, \af)$
associated to $\pi_0$ is injective.
\end{enumerate}
\end{prp}

In \ca{s},
if $a \in C^*_{\mathrm{r}} (G, A, \af)$
and $E_1 (a^* a) = 0,$ then $a = 0.$
For $p \neq 2,$
we don't have adjoints,
and we do not know a sensible version of this
statement.

Also,
in part~(\ref{P:Faithful:Inj}),
we do not know whether
the regular representation of $F^p_{\mathrm{r}} (G, A, \af)$
is isometric,
not even if $\pi_0$ itself is assumed isometric.

\begin{proof}[Proof of Proposition~\ref{P:Faithful}]
We prove~(\ref{P:Faithful:G}).
Let $\XBM$ be a \sfm,
let $\pi_0 \colon A \to \LLp$
be a nondegenerate contractive representation,
and let $(v, \pi)$ be the associated regular covariant \rpn.
Let $s_g$ and $t_g$
be as in Notation~\ref{N_3920_sgtg}.
If $a \in F^p_{\mathrm{r}} (G, A, \af)$ satisfies $E_g (a) = 0$
for all $g \in G,$
then $t_h (v \ltimes \pi) (a) s_k = 0$ for all $h, k \in G,$
whence $(v \ltimes \pi) (a) = 0$
by \Lem{L-StructRR}(\ref{L-StructRR-2b}).
Since $\pi_0$ is arbitrary, it follows that $a = 0.$
This proves~(\ref{P:Faithful:G}).

For~(\ref{P:Faithful:Inj}),
suppose $a \in F^p_{\mathrm{r}} (G, A, \af)$ and $\sm (a) = 0.$
Fix $l \in G.$
Taking $h = g^{-1}$ and $k = l^{-1} g^{-1}$ in
Proposition~\ref{P:CondExpt},
we get $(\pi_0 \circ \af_g) (E_l (a)) = 0$
for all $g \in G.$
So $E_l (a) = 0.$
This is true for all $l \in G,$
so $a = 0.$
\end{proof}

\begin{cor}\label{C-L1Inj}
Let $p \in [1, \I),$
let $G$ be a countable discrete group,
and let $(G, A, \af)$ be a separable nondegenerately representable
isometric $G$-$L^p$~operator algebra.
(See Definition~\ref{D:Action} and Definition~\ref{D-LpRep}.)
With notation as in \Lem{L-CompOfCrPrd},
the map
$\kp_{\mathrm{r}} \circ \kp \colon
  L^1 (G, A, \af) \to F^p_{\mathrm{r}} (G, A, \af)$
is injective.
\end{cor}

\begin{proof}
This is immediate from
Proposition~\ref{P:Faithful}(\ref{P:Faithful:G}).
\end{proof}

\begin{dfn}\label{D:StdCond}
Let $p \in [1, \I),$
let $G$ be a countable discrete group,
and let $(G, A, \af)$ be a separable nondegenerately representable
isometric $G$-$L^p$~operator algebra.
(See Definition~\ref{D:Action} and Definition~\ref{D-LpRep}.)
The map $E_1 \colon F^p_{\mathrm{r}} (G, A, \af) \to A$
of Proposition~\ref{P:CondExpt},
determined by
\[
E_1 \left( \ssum{g \in G} a_g u_g \right) = a_1
\]
when $\sum_{g \in G} a_g u_g \in C_{\mathrm{c}} (G, A, \af),$
is called the
{\emph{standard conditional expectation
from $F^p_{\mathrm{r}} (G, A, \af)$ to~$A$}},
and is denoted by~$E.$
\end{dfn}

\begin{prp}\label{P-EIsCondExpt}
Let $p \in [1, \I),$
let $G$ be a countable discrete group,
and let $(G, A, \af)$ be a separable nondegenerately representable
isometric $G$-$L^p$~operator algebra.
(See Definition~\ref{D:Action} and Definition~\ref{D-LpRep}.)
Identify $A$ with its image in $F^p_{\mathrm{r}} (G, A, \af),$
as in Remark~\ref{R-Ident}.
Then the map $E \colon F^p_{\mathrm{r}} (G, A, \af) \to A$
is a Banach conditional expectation
in the sense of Definition~2.1 of~\cite{PhLp2a}.
\end{prp}

\begin{proof}
We already know that $E$ is linear and $\| E \| = 1.$
It is immediate that $E (a) = a$ for $a \in A.$
It remains to verify $E (a b c) = a E (b) c$
for $a, c \in A$ and $b \in F^p_{\mathrm{r}} (G, A, \af).$
A calculation shows that this is true when
$b \in C_{\mathrm{c}} (G, A, \af),$
and the general case follows by continuity.
\end{proof}

We will also need the following technical lemma
concerning right and left multiplication by~$u_g.$
It is easy when $A$ is unital,
since then $u_g$ is in the crossed product and
$\| u_g \| = 1$ by \Lem{L:FGpCP}.
However, we need it in the nonunital case.
Essentially,
it says that $u_g$ defines isometric elements
of the multiplier algebras
of $F^p (G, A, \af)$ and $F^p_{\mathrm{r}} (G, A, \af).$
Since we only need these particular multipliers,
we do not discuss the general theory of multiplier algebras.

\begin{lem}\label{L-NormOfMultU}
Let $p \in [1, \I),$
let $G$ be a countable discrete group,
and let $(G, A, \af)$ be a separable nondegenerately representable
isometric $G$-$L^p$~operator algebra.
(See Definition~\ref{D:Action} and Definition~\ref{D-LpRep}.)
Then for every $g \in G$ there are unique isometric maps
\[
L_g, R_g \colon F^p (G, A, \af) \to F^p (G, A, \af)
\andeqn
L_g, R_g \colon
 F^p_{\mathrm{r}} (G, A, \af) \to F^p_{\mathrm{r}} (G, A, \af)
\]
given by $L_g (a) = u_g a$
and $R_g (a) = a u_g$ for $a \in F^p (G, A, \af)$
and $a \in F^p_{\mathrm{r}} (G, A, \af).$
That is, for
\[
a = \sum_{h \in G} a_h u_h \in C_{\mathrm{c}} (G, A, \af),
\]
we have
\[
L_g (\io (a))
 = \io \left( \ssum{h \in G} \af_g (a_h) u_{g h} \right)
\andeqn
R_g (\io (a))
 = \io \left( \ssum{h \in G} a_h u_{h g} \right),
\]
and similarly with $\io_{\mathrm{r}}$ in place of~$\io.$
Moreover, for $g \in G$ and $a, b \in F^p (G, A, \af)$
or $a, b \in F^p_{\mathrm{r}} (G, A, \af),$
we have
\[
L_g (a b) = L_g (a) b,
\,\,\,\,\,\,
R_g (a b) = a R_g (b),
\andeqn
a L_g (b) = R_g (a) b.
\]
\end{lem}

\begin{proof}
By Theorem~\ref{T-UPropFull}(\ref{T-UPropFull-3})
and Theorem~\ref{T-UPropRed}(\ref{T-UPropRed-3}),
it suffices to work with elements of $C_{\mathrm{c}} (G, A, \af).$
In particular, uniqueness is clear.

Let $(v, \pi)$ be any
nondegenerate $\sm$-finite contractive covariant \rpn{} of
$(G, A, \af),$
and let $a \in C_{\mathrm{c}} (G, A, \af).$
Then
we have
\[
( v \ltimes \pi) \left( \ssum{h \in G} \af_g (a_h) u_{g h} \right)
  = \sum_{h \in G} \pi ( \af_g (a_h) ) v_{g} v_{h}
  = v_{g} \sum_{h \in G} \pi ( a_h) v_{h}
  = v_g ( v \ltimes \pi) (a)
\]
and
\[
( v \ltimes \pi) \left( \ssum{h \in G} a_h u_{h g} \right)
  = \sum_{h \in G} \pi ( a_h ) v_{h} v_{g}
  = ( v \ltimes \pi) (a) v_g.
\]
Since $v_g$ is an isometry,
it follows that
\[
\left\|
 ( v \ltimes \pi) \left( \ssum{h \in G} \af_g (a_h) u_{g h} \right)
  \right\|
  = \| ( v \ltimes \pi) (a) \|
\]
and
\[
\left\|
   ( v \ltimes \pi) \left( \ssum{h \in G} a_h u_{h g} \right) \right\|
  = \| ( v \ltimes \pi) (a) \|.
\]
By Theorem~\ref{T-UPropFull}(\ref{T-UPropFull-2}),
taking the supremum over all
nondegenerate $\sm$-finite contractive covariant \rpn{s} of
$(G, A, \af)$
gives the result for the full crossed product.
By Theorem~\ref{T-UPropRed}(\ref{T-UPropRed-2}),
restricting to
nondegenerate $\sm$-finite contractive regular
covariant \rpn{s} gives the result for the reduced crossed product.
\end{proof}

\begin{ntn}\label{N-Multug}
Let the notation be as in \Lem{L-NormOfMultU}.
We write $u_g a$ for $L_g (a)$ and $a u_g$ for $R_g (a).$
This notation is consistent with Notation~\ref{N-Sum}.
\end{ntn}

\section{Structure theorems for reduced $L^p$~operator
  crossed products by discrete groups}\label{Sec_FpCX}

\indent
In this section,
we prove three results about the structure
of reduced $L^p$~operator crossed products by discrete groups.
None of them will be used in the computation of
the K-theory of $\OP{d}{p},$
but they are fairly easy to prove using known techniques,
and it therefore seems appropriate to include them.
The results are as follows.
First,
crossed products by isometric actions of countable discrete amenable groups
on unital nondegenerately $\sm$-finitely representable
$L^p$~operator algebras
preserve Banach algebra amenability
(in the sense of~\cite{Rnd}).
Second,
if a countable discrete group~$G$ acts freely and minimally
on a compact metrizable space~$X,$
then $F^p_{\mathrm{r}} (G, \, C (X))$ is simple.
Third,
if a countable discrete group~$G$ acts freely
on a compact metrizable space~$X,$
then the traces on $F^p_{\mathrm{r}} (G, \, C (X))$
all come from $G$-invariant measures on~$X.$
Besides these results,
it is proved in~\cite{PH}
that,
for $p \neq 1,$
the reduced $L^p$~operator crossed product
by a $G$-simple isometric action of a Powers group~$G$
is simple.

The definition of an amenable Banach algebra is given in
Definition 2.1.9 of~\cite{Rnd};
see Theorem 2.2.4 of~\cite{Rnd}
for two standard equivalent conditions.
The amenability result has essentially
the same proof as has been already
used for \ca{s}.
(See Theorem~1 of~\cite{Rs} or Proposition IV.4.3 of~\cite{Jn2}.)
We state the result on $L^1 (G, A, \af)$ from~\cite{Jn2}
for completeness.

\begin{prp}\label{P_3917_L1Amen}
Let $A$ be an amenable unital Banach algebra,
let $G$ be an amenable discrete group,
and let $\af \colon G \to \Aut (A)$
be an isometric action of $G$ on~$A.$
Then $L^1 (G, A, \af)$ is an amenable Banach algebra.
\end{prp}

\begin{proof}
This is Proposition IV.4.2 of~\cite{Jn2}.
See Lemma IV.1.1 and Definition IV.1.4 of~\cite{Jn2}
for the notation.
Actions in~\cite{Jn2} are assumed isometric;
see Definition I.2.2 of~\cite{Jn1}.
\end{proof}

\begin{thm}\label{T-CPAmen}
Let $p \in [1, \I),$
let $G$ be a countable discrete amenable group,
let $A$ be a unital $L^p$~operator algebra
which is nondegenerately $\sm$-finitely representable
(Definition \ref{D-LpRep}(\ref{D_3914_LpRep_Ndg})),
and let $\af \colon G \to \Aut (A)$
be an isometric action.
Suppose that $A$ is amenable as a Banach algebra.
Then $F^p (G, A, \af)$ and $F^p_{\mathrm{r}} (G, A, \af)$
are amenable as Banach algebras.
\end{thm}

\begin{proof}
We follow the proof of Proposition IV.4.3 of~\cite{Jn2}.
(The paper~\cite{Rs}
does not have Proposition~\ref{P_3917_L1Amen}
as an intermediate step.)
\Lem{L-CompOfCrPrd}
provides contractive  homomorphisms
\[
\kp \colon L^1 (G, A, \af) \to F^p (G, A, \af)
\andeqn
\kp_{\mathrm{r}} \colon F^p (G, A, \af) \to F^p_{\mathrm{r}} (G, A, \af)
\]
which have dense ranges.
So amenability of $F^p (G, A, \af)$
follows from
Proposition 2.3.1 of~\cite{Rnd}
and
Proposition~\ref{P_3917_L1Amen}.
Amenability of $F^p_{\mathrm{r}} (G, A, \af)$
then follows from amenability of $F^p (G, A, \af)$
and Proposition 2.3.1 of~\cite{Rnd}.
\end{proof}

Now we consider the results on free actions.
We need a lemma,
which is based on the main part of
the proof of Lemma VIII.3.7 of~\cite{Dv}.
We simplify and generalize it in several ways.
In particular,
we do not need the Rokhlin Lemma,
and the proof works for free minimal actions
of arbitrary discrete groups.

\begin{lem}\label{L-RmvOneElt}
Let $G$ be a discrete group,
let $X$ be a free compact $G$-space,
and let $F \in G \setminus \{ 1 \}$ be finite.
Then there exist $n \in \N$ and $s_1, s_2, \ldots, s_n \in C (X)$
such that $| s_k (x) | = 1$ for $k = 1, 2, \ldots, n$ and all $x \in X,$ 
and such that for all $x \in U$ and $g \in F,$
we have
\[
\frac{1}{n} \sum_{k = 1}^n s_k (x) {\overline{s_k (g^{-1} x) }} = 0.
\]
\end{lem}

\begin{proof}
For use only within this proof,
we define a pair $(F, U),$
consisting of a finite subset $F \subset G \setminus \{ 1 \}$
and an open subset $U \subset X,$
to be inessential
if there exist $n \in \N$ and $s_1, s_2, \ldots, s_n \in C (X)$
such that $| s_k (x) | = 1$ for $k = 1, 2, \ldots, n$ and all $x \in X,$ 
and such that for all $x \in U$ and $g \in F,$
we have
\[
\frac{1}{n} \sum_{k = 1}^n s_k (x) {\overline{s_k (g^{-1} x) }} = 0.
\]
Thus, we have to prove that $(F, X)$ is inessential
for every finite subset $F \subset G \setminus \{ 1 \}.$

We claim the following:
\begin{enumerate}
\item\label{L-Iness-1}
For every $x \in X$
and every $g \in G \SM \{ 1 \},$
there exists an open set $U \subset X$ with $x \in U$
such that $( \{ g \}, U )$ is inessential.
\item\label{L-Iness-2}
If $F \subset G \setminus \{ 1 \}$ is finite
and $U, V \S X$ are open,
and if $(F, U)$ and $(F, V)$ are both inessential,
then so is $(F, \, U \cup V).$
\item\label{L-Iness-3}
If $E, F \subset G \setminus \{ 1 \}$ are finite
and $U \S X$ is open,
and if $(E, U)$ and $(F, U)$ are both inessential,
then so is $(E \cup F, \, U).$
\end{enumerate}

To prove~(\ref{L-Iness-1}),
choose an open set $U \subset X$ with $x \in U$
such that ${\overline{U}} \cap g^{-1} {\overline{U}} = \varnothing.$
Take $n = 2,$
and take $s_1$ to be the constant function~$1.$
Choose a \cfn\  $r \colon X \to \R$
such that $r (x) = 0$ for $x \in {\overline{U}}$
and $r (x) = \pi$ for $x \in g^{-1} {\overline{U}}.$
Set $s_2 (x) = \exp (i r (x))$ for $x \in X.$
For $x \in U,$
we have
\[
\frac{1}{n} \sum_{k = 1}^n s_k (x) {\overline{s_k (g^{-1} x) }}
  = \frac{1}{2} \big[ 1 \cdot 1 + 1 \cdot (-1) \big]
  = 0.
\]
Thus $( \{ g \}, U )$ is inessential.

The proofs of (\ref{L-Iness-2}) and~(\ref{L-Iness-3})
are both based on the following calculation.
Let
\[
m, n \in \N,
\,\,\,\,\,\,
r_1, r_2, \ldots, r_m, s_1, s_2, \ldots, s_n \in C (X),
\,\,\,\,\,\,
x \in X,
\andeqn
g \in G.
\]
Then
\begin{align}\label{Eq:RmvProd}
\lefteqn{
\frac{1}{m n} \sum_{j = 1}^m \sum_{k = 1}^n
     (r_j s_k) (x) {\overline{(r_j s_k) (g^{-1} x) }}  }
        \\
& \hspace*{3em}
  = \left(
       \frac{1}{m} \sum_{j = 1}^m r_j (x) {\overline{r_j (g^{-1} x) }}
          \right)
    \left(
       \frac{1}{n} \sum_{k = 1}^n s_k (x) {\overline{s_k (g^{-1} x) }}
          \right).
       \notag
\end{align}
For~(\ref{L-Iness-2}),
choose $m, n \in \N$
and \cfn s
\[
r_1, r_2, \ldots, r_m, s_1, s_2, \ldots, s_n \colon X \to S^1
\]
such that for every $g \in F,$
we have
\begin{equation}\label{Eq:L-RmvUnion-Eq}
{\mbox{${\displaystyle{
  \frac{1}{m} \sum_{j = 1}^m r_j (x) {\overline{r_j (g^{-1} x) }} = 0}}$
   for $x \in U$}}
\andeqn
{\mbox{${\displaystyle{
  \frac{1}{n} \sum_{k = 1}^n s_k (x) {\overline{s_k (g^{-1} x) }} = 0}}$
   for $x \in V$}}.
\end{equation}
The functions $r_j s_k$ are \cfn s from $X$ to~$S^1,$
and (\ref{Eq:RmvProd})~implies that
for every $g \in F$ and $x \in U \cup V,$
we have
\[
\frac{1}{m n} \sum_{j = 1}^m \sum_{k = 1}^n
     (r_j s_k) (x) {\overline{(r_j s_k) (g^{-1} x) }}
 = 0.
\]
Similarly,
for~(\ref{L-Iness-3})
choose $m, n \in \N$
and \cfn s
\[
r_1, r_2, \ldots, r_m, s_1, s_2, \ldots, s_n \colon X \to S^1
\]
such that for every $x \in U,$
we have
\[
{\mbox{${\displaystyle{
  \frac{1}{m} \sum_{j = 1}^m r_j (x) {\overline{r_j (g^{-1} x) }} = 0}}$
   for $g \in E$}}
\andeqn
{\mbox{${\displaystyle{
  \frac{1}{n} \sum_{k = 1}^n s_k (x) {\overline{s_k (g^{-1} x) }} = 0}}$
   for $g \in F$}}.
\]
The claim again follows from~(\ref{Eq:RmvProd}).

We now claim that $( \{ g \}, X )$
is inessential
for all $g \in G \SM \{ 1 \}.$
Use compactness of~$X$ and~(\ref{L-Iness-1})
to find $n$ and open sets $U_1, U_2, \ldots, U_n \subset X$
such that $( \{ g \}, U_k)$ is inessential
for $k = 1, 2, \ldots, n$
and such that $\bigcup_{k = 1}^n U_k = X.$
Then apply~(\ref{L-Iness-2}) a total of $n - 1$ times.
This proves the claim.

For an arbitrary finite subset $F \subset G \setminus \{ 1 \},$
we use this claim and apply~(\ref{L-Iness-3}) repeatedly
to see that $(F, X)$ is inessential, as desired.
\end{proof}

We use the following analog of conventional notation for
transformation group \ca{s}.

\begin{ntn}\label{N-TransfGpLpAlg}
Let $X$ be a locally compact metrizable space,
and let $G$ be a second countable locally compact group which acts on~$X.$
Let $\af \colon G \to \Aut (C (X))$
be the action of Example~\ref{E-GpOnSp}.
We abbreviate
$F^p (G, A, \af)$ to $F^p (G, X)$
and $F^p_{\mathrm{r}} (G, A, \af)$ to $F^p_{\mathrm{r}} (G, X).$
If the action of $G$ on $X$ is called~$h,$
we write $F^p (G, X, h)$ and $F^p_{\mathrm{r}} (G, X, h).$
\end{ntn}

\begin{prp}\label{P-RmvEst}
Let $G$ be a countable discrete group,
let $X$ be a free compact $G$-space,
and let $E \colon F^p_{\mathrm{r}} (G, X) \to C (X)$
be the standard conditional expectation (Definition~\ref{D:StdCond}),
viewed as a map $F^p_{\mathrm{r}} (G, X) \to F^p_{\mathrm{r}} (G, X)$
via Remark~\ref{R-Ident}.
Then for every $a \in  F^p_{\mathrm{r}} (G, X)$
and $\ep > 0,$
there exist $n \in \N$ and $s_1, s_2, \ldots, s_n \in C (X)$
such that $| s_k (x) | = 1$ for $k = 1, 2, \ldots, n$ and all $x \in X,$ 
and such that
\[
\left\| E (a)
    - \frac{1}{n} \sum_{k = 1}^n s_k a {\overline{s_k}} \right\|
 < \ep.
\]
\end{prp}

\begin{proof}
Let $\af \colon G \to \Aut (C (X))$
be the action of Example~\ref{E-GpOnSp}.
Also, for $g \in G$ let $u_g \in F^p_{\mathrm{r}} (G, X)$
be as in Notation~\ref{N-Sum}.

Choose a finite set $F \subset G$
and elements $b_g \in C (X)$ for $g \in G$
such that,
with $b = \sum_{g \in F} b_g u_g,$
we have
$\left\| a - b \right\| < \tfrac{1}{2} \ep.$
\Wolog\  $1 \in F.$
By Lemma~\ref{L-RmvOneElt},
there exist $n \in \N$ and $s_1, s_2, \ldots, s_n \in C (X)$
such that $| s_k (x) | = 1$ for $k = 1, 2, \ldots, n$ and all $x \in X,$ 
and such that for all $x \in U$ and $g \in F \setminus \{ 1 \},$
we have
\begin{equation}\label{Eq_3922_Star}
\frac{1}{n} \sum_{k = 1}^n s_k (x) {\overline{s_k (g^{-1} x) }} = 0.
\end{equation}

Define $P \colon F^p_{\mathrm{r}} (G, X) \to F^p_{\mathrm{r}} (G, X)$
by
\[
P (c) = \frac{1}{n} \sum_{k = 1}^n s_k c {\overline{s_k}}
\]
for $c \in F^p_{\mathrm{r}} (G, X).$
We have to show that $\| E (a) - P (a) \| < \ep.$
Since $\| s_k \| = \| {\overline{s_k}} \| = 1$ for all~$k,$
we have $\| P \| \leq 1.$
Therefore
\begin{align*}
\| E (a) - P (a) \|
& \leq \| E (a) - E (b) \| + \| E (b) - P (b) \| + \| P (b) - P (a) \|
     \\
& < \tfrac{1}{2} \ep + \| E (b) - P (b) \| + \tfrac{1}{2} \ep
  = \| E (b) - P (b) \| + \ep.
\end{align*}
So it suffices to prove that $P (b) = E (b).$

Let $g \in F \setminus \{ 1 \}.$
Then, using~(\ref{Eq_3922_Star}) and the definition of~$\af_g,$
\[
P (b_g u_g)
  = \frac{1}{n} \sum_{k = 1}^n s_k b_g u_g {\overline{s_k}}
  = b_g \left( \frac{1}{n}
      \sum_{k = 1}^n s_k \af_g ({\overline{s_k}}) \right) u_g
  = 0.
\]
Also,
\[
P (b_1 u_1)
  = b_1 \cdot \frac{1}{n} \sum_{k = 1}^n s_k {\overline{s_k}}
  = b_1
  = E (b).
\]
Thus, $P (b) = E (b),$
as desired.
\end{proof}

\begin{thm}\label{T-FreeMinSimp}
Let $G$ be a countable discrete group,
and let $X$ be a free minimal compact metrizable $G$-space.
Then $F^p_{\mathrm{r}} (G, X)$ is simple.
\end{thm}

\begin{proof}
Let $I \subset F^p_{\mathrm{r}} (G, X)$ be a proper closed ideal.

We first claim that $I \cap C (X) = \{ 0 \}.$
If not,
let $f \in I \cap C (X)$ be nonzero.
Choose a nonempty open set $U \subset X$ on which $f$ does not vanish.
By minimality,
we have $\bigcup_{g \in G} g U = X.$
Since $X$ is compact, there is a finite set $S \subset G$
such that $\bigcup_{g \in S} g U = X.$
Define $b \in C (X)$
by
\[
b (x) = \sum_{g \in S} f (g^{-1} x) {\overline{ f (g^{-1} x) }}
\]
for $x \in X.$
Then $b (x) > 0$ for all $x \in X,$ so $b$ is invertible.
For $g \in G$ let $u_g \in F^p_{\mathrm{r}} (G, X)$
be as in Notation~\ref{N-Sum}.
Then $b = \sum_{g \in S} u_g f {\overline{f}} u_g^{-1} \in I.$
So $I$ contains an invertible element,
contradicting the assumption that $I$ is proper.
This proves the claim.

Let $E \colon F^p_{\mathrm{r}} (G, X) \to C (X)$
be the standard conditional expectation (Definition~\ref{D:StdCond}),
viewed as a map $F^p_{\mathrm{r}} (G, X) \to F^p_{\mathrm{r}} (G, X)$
(following Remark~\ref{R-Ident}).
We claim that $E (a) = 0$ for all $a \in I.$
It suffices to show that $E (a) \in I.$
To prove this,
let $\ep > 0.$
Use Proposition~\ref{P-RmvEst} to choose
$n \in \N$ and $s_1, s_2, \ldots, s_n \in C (X)$
such that the element
$b = \frac{1}{n} \sum_{k = 1}^n s_k a {\overline{s_k}}$
satisfies $\| E (a) - b \| < \ep.$
Clearly $b \in I.$
Since $\ep > 0$ is arbitrary,
this implies that $E (a) \in {\overline{I}} = I.$
The claim is proved.

Now let $a \in I.$
For all $g \in G,$
we have $a u_g^{-1} \in I,$ so $E (a u_g^{-1}) = 0.$
In the notation of Proposition~\ref{P:CondExpt},
this means that $E_g (a) = 0$ for all $g \in G.$
Proposition~\ref{P:Faithful}(\ref{P:Faithful:G})
now implies that $a = 0.$
\end{proof}

We can use the same methods to identify all the normalized traces
on $F^p_{\mathrm{r}} (G, X).$

\begin{dfn}\label{D-NTrace}
Let $A$ be a unital Banach algebra.
Then a
{\emph{normalized trace}}
on~$A$ is a linear functional satisfying the following three conditions:
\begin{enumerate}
\item\label{D-NTrace-1}
$\ta (1) = 1.$
\item\label{D-NTrace-2}
$\| \ta \| = 1.$
\item\label{D-NTrace-3}
$\ta (b a) = \ta (a b)$ for all $a, b \in A.$
\end{enumerate}
\end{dfn}

When $A$ is a unital \ca,
the normalized traces are exactly the tracial states.

Our result requires that the action be free,
but not necessarily minimal.
The main point is contained in the following proposition.
The proof follows the proof of Corollary VIII.3.8 of~\cite{Dv}.

\begin{prp}\label{P-TracesOnSubalg}
Let $G$ be a countable discrete group,
let $X$ be a free compact metrizable $G$-space,
and let $A \subset F^p_{\mathrm{r}} (G, X)$
be a subalgebra such that $C (X) \subset A.$
Let $E \colon F^p_{\mathrm{r}} (G, X) \to C (X)$
be the standard conditional expectation (Definition~\ref{D:StdCond}).
Let $\ta \colon A \to \C$ be a normalized trace
(Definition~\ref{D-NTrace}).
Then there exists a unique Borel probability measure $\mu$ on~$X$
such that for all $a \in A$ we have
\[
\ta (a) = \int_X E (a) \, d \mu.
\]
\end{prp}

\begin{proof}
We prove that $\ta = (\ta |_{C (X)} ) \circ E.$
The statement then follows by applying the Riesz Representation Theorem
to $\ta |_{C (X)}.$

Let $a \in A$ and let $\ep > 0.$
We prove that
$| \ta (a) - \ta (E (a)) | < \ep.$
Use Proposition~\ref{P-RmvEst} to choose
$n \in \N$ and $s_1, s_2, \ldots, s_n \in C (X)$
such that $| s_k (x) | = 1$ for $k = 1, 2, \ldots, n$ and all $x \in X,$ 
and such that
$\left\| E (a) - \frac{1}{n} \sum_{k = 1}^n s_k a {\overline{s_k}} \right\|
  < \ep.$
Since
$s_1, s_2, \ldots, s_n, {\overline{s_1}}, {\overline{s_2}}, \ldots,
 {\overline{s_n}} \in A,$
we have $\ta (s_k a {\overline{s_k}} ) = \ta (a)$ for $k = 1, 2, \ldots, n.$
Therefore
\[
\big| \ta (a) - \ta (E (a)) \big|
  = \left| \ta \left( \frac{1}{n} \sum_{k = 1}^n
                   s_k a {\overline{s_k}} \right)
           - \ta (E (a)) \right|
  \leq \left\| E (a) - \frac{1}{n} \sum_{k = 1}^n
                   s_k a {\overline{s_k}} \right\|
  < \ep.
\]
This completes the proof.
\end{proof}

\begin{thm}\label{T-11202Traces}
Let $G$ be a countable discrete group,
and let $X$ be a free compact metrizable $G$-space.
Let $E \colon F^p_{\mathrm{r}} (G, X) \to C (X)$
be the standard conditional expectation (Definition~\ref{D:StdCond}).
For a $G$-invariant Borel probability measure $\mu$ on~$X,$
define a linear functional $\ta_{\mu}$ on $F^p_{\mathrm{r}} (G, X)$
by
\[
\ta_{\mu} (a) = \int_X E (a) \, d \mu
\]
for all $a \in F^p_{\mathrm{r}} (G, X).$
Then $\mu \mapsto \ta_{\mu}$ is an affine bijection
from the $G$-invariant Borel probability measures on~$X$
to the normalized traces on $F^p_{\mathrm{r}} (G, X)$
(Definition~\ref{D-NTrace}).
Its inverse sends $\ta$ to the measure obtained from the
functional $\ta |_{C (X)}$ via the Riesz Representation Theorem.
\end{thm}

\begin{proof}
It is easy to check that
if $\mu$ is a $G$-invariant Borel probability measure on~$X,$
then $\ta_{\mu}$ is a normalized trace on $F^p_{\mathrm{r}} (G, X).$
Clearly $\ta_{\mu} (f) = \int_X f \, d \mu$ for $f \in C (X).$
This implies that $\mu \mapsto \ta_{\mu}$ is injective
and that the description of its inverse is correct on the range
of this map.

It remains only to prove that $\mu \mapsto \ta_{\mu}$ is surjective.
Let $\ta$ be a normalized trace on $F^p_{\mathrm{r}} (G, X).$
Proposition~\ref{P-TracesOnSubalg} provides
a Borel probability measure $\mu$ on~$X$
such that $\ta (a) = \int_X E (a) \, d \mu$
for all $a \in F^p_{\mathrm{r}} (G, X).$
For $g \in G$ and $f \in C (X),$
using the fact that $\ta$ is a trace at the second step,
we have
\[
\int_X f (g^{-1} x)  \, d \mu (x)
  = \ta \big( u_g f u_g^{-1} \big)
  = \ta (f)
  = \int_X f \, d \mu.
\]
Uniqueness in the Riesz Representation Theorem
now implies that $\mu$ is $G$-invariant.
This completes the proof.
\end{proof}

\section{The K-theory of direct limits and crossed
  products by~$\Z$}\label{Sec_PV}

\indent
We give two general results on K-theory that are needed
for the computation of $K_* \big( \OP{d}{p} \big).$
The first is the K-theory of direct limits of Banach algebras
with contractive homomorphisms,
and the second is the K-theory of
a reduced $L^p$~operator crossed product by~$\Z.$
For the main definitions and theorems related to
the K-theory of Banach algebras
(in fact, ``local Banach algebras''),
we refer to Sections 5, 8, and~9 of~\cite{Bl3}.

We start with direct limits.
Without some condition on norms,
we do not expect direct limits to exist in general.
For example,
the limit which describes $\| \ph_i (a) \|$
in Proposition~\ref{P-DLim} might be infinite,
or fail to exist because the net has more than one limit point.

\begin{prp}\label{P-DLim}
Let $I$ be a directed set,
and let $\big( (A_i)_{i \in I}, \, (\ph_{j, i})_{i \leq j} \big)$
be a direct system of Banach algebras with contractive homomorphisms.
Then the direct limit $A = \Dirlim A_i$
exists in the category of Banach algebras and contractive homomorphisms.
If the maps to the direct limit are called $\ph_i \colon A_i \to A,$
then
$\bigcup_{i \in I} \ph_i (A_i)$ is a dense subalgebra of~$A$
and
for all $i \in I$ and all $a \in A_i,$
we have $\| \ph_i (a) \| = \lim_{j \geq i} \| \ph_{j, i} (a) \|.$
\end{prp}

\begin{proof}
The proof is essentially the same as the proof of
Proposition 2.5.1 of~\cite{Ph0},
where the statement is proved for the category of \ca{s}
equipped with actions of a fixed group~$G.$
Also see Section 3.3 of~\cite{Bl3},
where a weaker boundedness condition on the maps is used,
but where the universal property of the direct limit
is not addressed.
\end{proof}

If the $A_i$ in Proposition~\ref{P-DLim}
are $L^p$~operator algebras,
we do not know whether it follows that
$\Dirlim A_i$ is an $L^p$~operator algebra.

\begin{cor}\label{C-ClUIsDLim}
Let $A$ be a Banach algebra, let $I$ be a directed set,
and let $(A_i)_{i \in I}$
be a family of closed subalgebras of~$A$ such that
$A_i \S A_j$ for $i \leq j,$
and such that ${\ov{\bigcup_{i \in I} A_i}} = A.$
For $i \leq j,$ let $\ph_{j, i} \colon A_i \to A_j$
be the inclusion.
Then the canonical map from $\Dirlim A_i$ to~$A$
is an isometric isomorphism.
\end{cor}

\begin{proof}
The formula for $\| \ph_i (a) \|$ in Proposition~\ref{P-DLim}
implies that the map $\Dirlim A_i \to A$ is isometric.
Therefore density of $\bigcup_{i \in I} A_i$ in $A$
implies surjectivity.
\end{proof}

For a Banach algebra $A$ and a locally compact Hausdorff space~$X,$
we let $C_0 (X, A)$
be the Banach algebra of all \cfn{s} $b \colon X \to A$
such that $x \mapsto \| b (x) \|$
vanishes at infinity on~$X,$
with pointwise multiplication and the supremum norm.
If $\ph \colon A \to B$
is a \ct{} \hm{} of Banach algebras,
let $C_0 (X, \ph) \colon C_0 (X, A) \to C_0 (X, B)$
be the \ct{} \hm{} determined by
$C_0 (X, \ph) (b) (x) = \ph (b (x))$
for $b \in C_0 (X, A)$ and $x \in X.$
The following lemma is well known,
but we do not know a reference.
The case we care about is $X = \R,$
but the proof in this case is no simpler.

\begin{lem}\label{L_3918_LimC0}
Let $X$ be a locally compact Hausdorff space.
Then $A \mapsto C_0 (X, A)$
is a functor
from the category of Banach algebras with contractive homomorphisms
to itself
which preserves direct limits.
\end{lem}

\begin{proof}
The only nonobvious part is that $C_0 (X, {-})$
preserves direct limits.
So let $I$ be a directed set,
and let $\big( (A_i)_{i \in I}, \, (\ph_{j, i})_{i \leq j} \big)$
be a direct system of Banach algebras with contractive homomorphisms.
Set $A = \Dirlim A_i,$
and for $i \in I$ let $\ph_i \colon A_i \to A$
be the map to the direct limit.
For $i \leq j,$ set
$\ps_{j, i} = C_0 (X, \ph_{j, i}) \colon C_0 (X, A_i) \to C_0 (X, A_j),$
and for $i \in I$ set
$\ps_i = C_0 (X, \ph_i) \colon C_0 (X, A_i) \to C_0 (X, A).$
We must show that the algebra $C_0 (X, A)$
and the maps $\ph_i$
satisfy the universal property of the direct limit.

We first claim that
for every $b \in C_0 (X, A)$
and every $\ep > 0$
there are $i \in I$ and $c \in C_0 (X, A_i)$
such that $\| \ps_i (c) - b \| < \ep.$
Choose a compact subset $K \S X$
such that $\| b (x) \| < \frac{\ep}{5}$ for all $x \in X \SM K.$
Choose a finite cover
$(U_1, U_2, \ldots, U_n)$ of $K$
consisting of nonempty open sets $U_m \S X$
such that $\| b (x) - b (y) \| < \frac{\ep}{5}$
for $m = 1, 2, \ldots, n$
and $x, y \in U_m,$
and choose \cfn{s} $f_1, f_2, \ldots, f_n \colon X \to [0, 1]$
with compact supports $\supp (f_m) \S U_m$
such that $\sum_{m = 1}^n f_m (g) = 1$ for all $x \in K$
and such that $\sum_{m = 1}^n f_m (g) \in [0, 1]$ for all $x \in X.$
For $m = 1, 2, \ldots, n,$
choose any $x_m \in U_m,$
and then choose $i (m) \in I$ and $a_m \in A_{i (m)}$
such that $\| \ph_{i (m)} (a_m) - b (x_m) \| < \frac{\ep}{5}.$
Define $c_m \in C_0 (X, A_{i (m)})$
by $c_m (x) = f_m (x) a_m$
for $x \in X.$
Choose $i \in I$ such that $i \geq i_m$ for $m = 1, 2, \ldots, n.$
Set $c = \sum_{m = 1}^{n} \ps_{i, i (m)} (c_m).$

We estimate $\| \ps_i (c) - b \|.$
If $x \not\in \bigcup_{m = 1}^n \supp (f_m),$
then $\| b (x) \| < \frac{\ep}{5}$ and
$c (x) = 0,$
so $\| \ph_i (c (x)) - b (x) \| < \frac{\ep}{5}.$
Otherwise,
let $J$
be the set of $m \in \{ 1, 2, \ldots, n \}$
such that $x \in U_m.$
If $x \not\in K,$
then for all $m \in J$ we have
\[
\| b (x_m) \|
 \leq \| b (x) \| + \| b (x) - b (x_m) \|
 < \frac{\ep}{5} + \frac{\ep}{5}
 = \frac{2 \ep}{5}.
\]
So
\begin{align*}
\| \ph_i (c (x)) \|
& \leq \| c (x) \|
  \leq \sum_{m \in J} f_m (x) \| \ph_{i (m)} (a_m) \|
  \\
& \leq \sum_{m \in J} f_m (x)
    \left(\| b (x_m) \| + \frac{\ep}{5} \right)
  \leq \sum_{m \in J} f_m (x)
    \left( \frac{2 \ep}{5} + \frac{\ep}{5} \right)
  \leq \frac{3 \ep}{5}.
\end{align*}
Since $\| b (x) \| < \frac{\ep}{5},$
we get $\| c (x) - b (x) \| < \frac{4 \ep}{5}.$
Finally,
for $x \in K,$
we have
\[
\| \ph_i (c (x)) - b (x) \|
  \leq \sum_{m \in J} f_m (x) \| \ph_{i (m)} (a_m) - b (x) \|
  \leq \sum_{m \in J} f_m (x) \cdot \frac{\ep}{5}
  = \frac{\ep}{5}.
\]
Therefore $\| \ps_i (c) - b \| \leq \frac{4 \ep}{5} < \ep.$
The claim is proved.

Next,
we claim that if $i \in I$
and $b \in C_0 (X, A_i),$
then $\lim_{j \geq i} \| \ps_{j, i} (b) \| = \| \ps_i (b) \|.$
To see this, let $\ep > 0.$
Choose a finite cover
$(U_1, U_2, \ldots, U_n)$ of $K$
consisting of nonempty open sets $U_m \S X$
such that $\| b (x) - b (y) \| < \frac{\ep}{3}$
for $m = 1, 2, \ldots, n$
and $x, y \in U_m.$
For $m = 1, 2, \ldots, n,$
choose $x_m \in U_m$
and choose $i (m) \in I$
such that
$\| \ph_{i (m), \, i} (b (x_m)) \|
 < \| \ph_i (b (x_m)) \| + \frac{\ep}{3}.$
Choose $j \in I$ such that $j \geq i (m)$ for $m = 1, 2, \ldots, n.$
Now let $x \in X.$
If there is $m$ such that $x \in U_m,$
then
\begin{align*}
\| \ph_{j, i} (b (x)) \|
& \leq \| \ph_{j, i} ( b (x) - b (x_m) ) \|
           + \big\| \ph_{j, i (m)}
                 \big( \ph_{i (m), \, i} (b (x_m)) \big) \big\|
             \\
& 
  < \frac{\ep}{3} + \| \ph_i (b (x_m)) \| + \frac{\ep}{3}
  \leq \| \ps_i (b) \| + \frac{2 \ep}{3}.
\end{align*}
Otherwise, $x \not\in K,$
so $\| b (x) \| < \frac{\ep}{3},$
whence $\| \ph_{j, i} (b (x)) \| < \frac{\ep}{3}.$
It follows that
$\| \ps_{j, i} (b) \|
  \leq \| \ps_i (b) \| + \frac{2 \ep}{3}
  < \| \ps_i (b) \| + \ep.$
The claim is proved.

We now prove the universal property.
Let $D$ be a Banach algebra,
and let $(\gm_i)_{i \in I}$
be a family of contractive homomorphisms
$\gm_i \colon C_0 (X, A_i) \to D$
such that $\gm_j \circ \ps_{j, i} = \gm_i$
whenever $i, j \in I$
satisfy $i \leq j.$
We need a contractive \hm{}
$\gm \colon C_0 (X, A) \to D$
such that $\gm \circ \ps_i = \gm_i$ for all $i \in I.$
The first claim implies that
$\bigcup_{i \in I} \ps_i ( C_0 (X, A_i) )$
is dense in $C_0 (X, A),$
so $\gm$ is unique if it exists.

Let $b \in \bigcup_{i \in I} \ps_i ( C_0 (X, A_i) ).$
We claim that whenever $i_1, i_2 \in I,$
and $b_1 \in C_0 (X, A_{i_1})$ and $b_2 \in C_0 (X, A_{i_2})$
satisfy $\ps_{i_1} (b_1) = b$ and $\ps_{i_2} (b_2) = b,$
then $\gm_{i_1} (b_1) = \gm_{i_2} (b_2).$
To prove this,
let $\ep > 0.$
Choose $j \in I$ such that $j \geq i_1$ and $j \geq i_2.$
Then $\ps_j ( \ps_{j, i_1} (b_1) - \ps_{j, i_2} (b_2) ) = 0.$
Use the second claim to choose $k \geq j$
such that
$\big\| \ps_{k, j} ( \ps_{j, i_1} (b_1) - \ps_{j, i_2} (b_2) ) \big\|
     < \ep.$
Then
\[
\| \gm_{i_1} (b_1) - \gm_{i_2} (b_2) \|
 = \big\| \gm_{k} \big( \ps_{k, i_1} (b_1) - \ps_{k, i_2} (b_2) \big) \big\|
 \leq \big\| \ps_{k, j} \big( \ps_{j, i_1} (b_1)
                 - \ps_{j, i_2} (b_2) \big) \big\|
 < \ep.
\]
Since $\ep > 0$ is arbitrary,
the claim follows.

There is therefore a well defined map
$\bt \colon \bigcup_{i \in I} \ps_i ( C_0 (X, A_i) ) \to D$
such that $\bt \circ \ps_i = \gm_i$ for all $i \in I.$
It is easy to check that $\bt$ is an algebra \hm.

We next claim that
for all $b \in \bigcup_{i \in I} \ps_i ( C_0 (X, A_i) ),$
we have $\| \bt (b) \| \leq \| b \|.$
This will complete the proof,
since we can take $\gm$ to be the extension of $\bt$
to $C_0 (X, A)$ by continuity.

Let $\ep > 0.$
Choose $i \in I$ and $c \in C_0 (X, A_i)$
such that $\ps_i (c) = b.$
Use the second claim to choose $j \in I$ such that $j \geq i$
and $\| \ps_{j, i} (c) \| < \| b \| + \ep.$
Then
\[
\| \gm (b) \|
 = \| \gm_j (\ps_{j, i} (c)) \|
 \leq \| \ps_{j, i} (c) \|
 < \| b \| + \ep.
\]
This completes the proof.
\end{proof}

\begin{thm}\label{P-KThDLim}
Let $I$ be a directed set,
and let $\big( (A_i)_{i \in I}, \, (\ph_{j, i})_{i \leq j} \big)$
be a direct system of Banach algebras with contractive homomorphisms.
Then the induced map
$\Dirlim K_* (A_i) \to K_* \big( \Dirlim A_i \big)$
is an isomorphism.
\end{thm}

\begin{proof}
It is shown in 5.2.4 of~\cite{Bl3}
that the result holds for the Murray-von Neumann semigroups
of idempotents over the algebra
in place of $K_0.$
It is easy to see that the result then also holds
for the Grothendieck groups of these semigroups.
The result for $K_0$ follows from
that for the Murray-von Neumann semigroups
by unitizing.
(See the proof of
Proposition 2.5.4 of~\cite{Ph0}.
One can also copy the proof of this proposition,
substituting Lemma 2.5.8 of~\cite{Ph0}
for the use of continuous functional calculus.)

To prove the result for $K_1,$
use the result for~$K_0,$
the natural isomorphism $K_1 (B) \cong K_0 (C_0 (\R, B))$
for any Banach algebra~$B$
(Theorem 8.2.2 of~\cite{Bl3}),
and Lemma~\ref{L_3918_LimC0}.
\end{proof}

\begin{thm}\label{P-KStdUHF}
Let $P$ be the set of prime numbers,
enumerated as $\{ p_1, p_2, \ldots \}.$
Let $N \colon P \to \Nz \cup \{ \I \}$ be any function
such that $\sum_{t \in P} N (t) = \I,$
and for $n \in \Nz$ define
\[
r (n) = p_1^{\min (N (1), n)} p_2^{\min (N (2), n)}
          \cdots p_n^{\min (N (n), n)}.
\]
Let $p \in [1, \I].$
Let $D$ be a spatial $L^p$~UHF algebra of type~$N,$
as in Example~\ref{E-pUHF}.
Then $K_1 (D) = 0$ and
$K_0 (D)
 \cong \bigcup_{n = 1}^{\I}
     r (n)^{-1} \Z
 \S \Q,$
via an isomorphism which sends the class $[1_D]$
of the idempotent $1_D$ to $1 \in \Q.$
\end{thm}

\begin{proof}
Theorem~3.10 and Definition~3.5 of~\cite{PhLp2a}
imply that there are
subalgebras
$D_0 \S D_1 \S \cdots \S D$
such that ${\ov{\bigcup_{n = 0}^{\I} D_n}} = D$
and such that
$D_n$ is isometrically isomorphic to $M^p_{r (n)}$
for all $n \in \Nz.$
Since $M^p_{r (n)}$ is $M_{r (n)}$ with an equivalent norm,
its K-theory is the same as for $M_{r (n)}.$
Thus, we have isomorphisms $\et_n \colon K_0 (D_n) \to r (n)^{-1} \Z,$
with $\et_n ([1]) = 1,$
and we have $K_1 (D_n) = 0,$
just as in the \ca{} case.
Apply Corollary~\ref{C-ClUIsDLim}
and Theorem~\ref{P-KThDLim}
to get the result.
\end{proof}

\begin{lem}\label{L-KthyTPM}
Let $p \in [1, \I),$
let $\XBM$ be a \msp,
and let $A \S \LLp$ be a norm closed subalgebra.
Let $S$ be a set,
let ${\ov{M}}^p_S$ be as in Example~\ref{E-MatpS},
and let $s_0 \in S.$
Let ${\ov{M}}^p_S \otimes_p A$
be as in Example~\ref{E-LpTP}.
Let $\ep \colon A \to {\ov{M}}^p_S \otimes_p A$
be the \hm{}
defined by $\ep (a) = e_{s_0, s_0} \otimes a$ for $a \in A.$
Then $\ep$ is an isomorphism on K-theory.
\end{lem}

\begin{proof}
Let ${\mathcal{F}}$ be the collection of
finite subsets of~$S$ which contain~$s_0,$
ordered by inclusion.
For $E \in {\mathcal{F}}$
let $M_E^p \S {\ov{M}}_S^p$
be as in Example~\ref{E-MatpS}.
If $F \in {\mathcal{F}}$ contains~$E,$
then
Lemma~\ref{L-MatpSMap} gives
isometric maps
$\ep_{F, E} \colon M_E^p \otimes_p A \to M^p_F \otimes_p A$
and
$\ep_E \colon M_E^p \otimes_p A \to {\ov{M}}^p_S \otimes_p A,$
coming from the inclusions.
Clearly $\ep_F \circ \ep_{F, E} = \ep_E.$
Also,
by the construction of ${\ov{M}}^p_S$
and density of elementary tensors in ${\ov{M}}^p_S \otimes_p A,$
the union of the ranges of the maps $\ep_F$ is dense.
Therefore we can make the identification
${\ov{M}}^p_S \otimes_p A
 = \Dirlim_{F \in {\mathcal{F}} } M^p_F \otimes_p A,$
using the maps $\ep_F$ for $F \in {\mathcal{F}}.$

For $E, F \in {\mathcal{F}}$ with $E \S F,$
we have $M^p_E \otimes_p A = M^p_E \otimes_{\mathrm{alg}} A$
and $M^p_F \otimes_p A = M^p_F \otimes_{\mathrm{alg}} A.$
The map $\ep_{F, E}$ is therefore clearly an isomorphism on K-theory.
It now follows from Theorem~\ref{P-KThDLim}
that $(\ep_F)_*$ is an isomorphism
for all $F \in {\mathcal{F}}.$
Take $F = \{ s_0 \}$ to get the result.
\end{proof}

We now give an analog of the
Pimsner-Voiculescu exact sequence~\cite{PV}
for reduced $L^p$~operator crossed products by~$\Z.$
Rather than repeating the proof of~\cite{PV}
(or any of a number of later proofs),
we show that the smooth crossed product,
for which the result is already known~\cite{PS},
has the same K-theory as the reduced $L^p$~operator crossed product.
Therefore we start with smooth crossed products.

The following definition is a special case of a much more
general definition.
(See Definition 2.1.0 of~\cite{Sw2}.)

\begin{dfn}\label{D-SmoothCrPrd}
Let $A$ be a Banach algebra,
and let $\af$ be an isometric automorphism of~$A.$
We use the same notation for the obvious isometric action
$\af \colon \Z \to \Aut (A).$
We define the {\emph{smooth crossed product}}
${\mathcal{S}} (\Z, A, \af)$
to be the vector space of all functions
$a \in L^1 (\Z, A, \af)$
such that for every $r \in \Nz$ the number
\[
\| a \|_{r, 1} = \sum_{n \in \N} (1 + | n | )^r \| a_n \|
\]
is finite.
We equip it with the topology given by the seminorms
$\| \cdot \|_{r, 1},$
and with the (convolution) multiplication
inherited from $L^1 (\Z, A, \af).$
Let
$\kp_{\I} \colon {\mathcal{S}} (\Z, A, \af) \to L^1 (\Z, A, \af)$
be the inclusion.
\end{dfn}

\begin{prp}\label{P-SIsFrechet}
The space ${\mathcal{S}} (\Z, A, \af)$
of \Def{D-SmoothCrPrd}
is a locally $m$-convex Fr\'{e}chet algebra,
and $\kp_{\I}$ is \ct.
\end{prp}

\begin{proof}
Since $\Z$ is a discrete group,
this algebra is the same as the one called $L_1^{\sm} (\Z, A)$
in Equation (2.1.3) in~\cite{Sw2},
with $\sm (n) = 1 + | n |$ for $n \in \Z.$
The conclusion now follows from
the second part of Theorem 3.1.7 of~\cite{Sw2}.
(The condition there, that $\sm_{-}$ bound~$\Ad,$
is trivial.
See the discussion after Definition 1.3.9 of~\cite{Sw2}.)
\end{proof}

\begin{lem}\label{L-SchwInj}
Let $p \in [1, \I),$
let $A$ be an $L^p$~operator algebra,
and let $\af \in \Aut (A)$ be isometric.
With notation as in \Lem{L-CompOfCrPrd}
and \Def{D-SmoothCrPrd},
the map
\[
\kp_{\mathrm{r}} \circ \kp \circ \kp_{\I} \colon
  {\mathcal{S}} (\Z, A, \af) \to F^p_{\mathrm{r}} (\Z, A, \af)
\]
is injective
and has dense range.
\end{lem}

\begin{proof}
Injectivity is immediate from
injectivity of $\kp_{\I}$ and Corollary~\ref{C-L1Inj}.
Density of the range
follows from Theorem~\ref{T-UPropRed}(\ref{T-UPropRed-3}).
\end{proof}

\begin{ntn}\label{N-ScwhSubalg}
Let $p \in [1, \I),$
let $A$ be an $L^p$~operator algebra,
and let $\af \in \Aut (A)$ be isometric.
Using \Lem{L-SchwInj},
from now on we identify ${\mathcal{S}} (\Z, A, \af)$
with the corresponding subalgebra of $F^p_{\mathrm{r}} (\Z, A, \af).$
\end{ntn}
The following definition is again a very special case of something
much more general.
See the Appendix in~\cite{Sw2},
especially Theorem~A.2.

\begin{dfn}\label{D-SmoothElts}
Let $B$ be a Banach algebra,
and let $\bt \colon S^1 \to \Aut (B)$
be an action of the circle $S^1$ on~$B.$
We say that $b \in B$ is {\emph{smooth}}
(or $\bt$-smooth if $\bt$ must be specified)
if the function $\ld \mapsto \bt_{\exp (i \ld)} (b)$
is a $C^{\infty}$~function from $\R$ to~$B.$
\end{dfn}

\begin{thm}\label{T-SmoothIffSchw}
Let $p \in [1, \I),$
let $A$ be a separable $L^p$~operator algebra,
and let $\af \in \Aut (A)$ be isometric.
Let
${\widehat{\af}} \colon
    S^1 \to \Aut \big( F^p_{\mathrm{r}} (\Z, A, \af) \big)$
be the dual action of Theorem~\ref{T-DualpRed}.
Then $b \in F^p_{\mathrm{r}} (\Z, A, \af)$
is ${\widehat{\af}}$-smooth
\ifo\  $b \in {\mathcal{S}} (\Z, A, \af).$
\end{thm}

\begin{proof}
We first prove that smooth elements are in ${\mathcal{S}} (\Z, A, \af).$

For $n \in \Z,$
let $E_n \colon F^p_{\mathrm{r}} (\Z, A, \af) \to A$
be as in Proposition~\ref{P:CondExpt}.
It follows from continuity of $E_n$ and the definition
of ${\widehat{\af}}$
(see Definition~\ref{D-DualCc} for the formula)
that
\begin{equation}\label{Eq:EnOnDual}
E_n \big( {\widehat{\af}}_{\exp (i \ld)} (b) \big)
  = \exp (- i n \ld) E_n (b)
\end{equation}
for all $b \in F^p_{\mathrm{r}} (\Z, A, \af)$ and $n \in \Z.$

We claim that if $b \in F^p_{\mathrm{r}} (\Z, A, \af)$
and the function $\ld \mapsto {\widehat{\af}}_{\exp (i \ld)} (b)$
is differentiable,
with derivative $f \colon \R \to F^p_{\mathrm{r}} (\Z, A, \af),$
then
\[
E_n (f (\ld)) = - i n \exp (- i n \ld) E_n (b)
\]
for all $n \in \Z.$

To prove the claim,
we first observe that $f$ is determined by the
condition
\[
\lim_{h \to 0} \frac{\big\| {\widehat{\af}}_{\exp (i (\ld + h))} (b)
                       - {\widehat{\af}}_{\exp (i \ld)} (b)
                       - h f (\ld) \big\|}{| h |}
         = 0.
\]
Apply~$E_n,$
and use (\ref{Eq:EnOnDual}) and boundedness of $E_n$ to get
\[
\lim_{h \to 0} \frac{\big\| \big[ \exp (- i n (\ld + h))
                       - \exp (- i n \ld) \big]  E_n (b)
                       - h E_n (f (\ld)) \big\|}{| h |}
         = 0.
\]
The claim follows.

The claim implies that
if the function $\ld \mapsto {\widehat{\af}}_{\exp (i \ld)} (b)$
is differentiable,
with derivative $f \colon \R \to F^p_{\mathrm{r}} (\Z, A, \af),$
then there is $c \in F^p_{\mathrm{r}} (\Z, A, \af)$
(namely $f (0)$)
such that $E_n (c) = - i n E_n (b)$
for all $n \in \Z.$
Moreover, from~(\ref{Eq:EnOnDual})
and Proposition~\ref{P:Faithful}(\ref{P:Faithful:G})
we get $f (\ld) = {\widehat{\af}}_{\exp (i \ld)} (c)$
for all $\ld \in \R.$

Now suppose $b \in F^p_{\mathrm{r}} (\Z, A, \af)$ is smooth.
Then, by induction, for every $r \in \Nz,$
there is $b_r \in F^p_{\mathrm{r}} (\Z, A, \af)$
such that $E_n (b_r) = (- i n)^r E_n (b)$
for all $n \in \N.$
Since $\sup_{n \in \Z} \| E_n (b_r) \| \leq \| b_r \|,$
it follows that for every $r \in \Nz$ we have
\[
\sup_{n \in \Z} | n |^r \| E_n (b) \| < \I.
\]
So for every $r \in \Nz$ we have,
using
$(1 + | n | )^{r + 2} \leq 2^{r + 2} \big( 1 + | n |^{r + 2} \big),$
\begin{align*}
\sum_{n \in \Z} (1 + | n | )^r \| E_n (b) \|
 & \leq \sum_{n \in \Z}
      \frac{2^{r + 2} \big( 1 + | n |^{r + 2} \big) \| E_n (b) \|}{
           (1 + | n | )^2}
 \\
 & \leq 2^{r + 2} \left( \sup_{n \in \Z} \| E_n (b) \|
          + \sup_{n \in \Z} | n |^{r + 2} \| E_n (b) \| \right)
            \sum_{n \in \Z} \frac{1}{(1 + | n | )^2}
 \\
 & < \I.
\end{align*}
It follows that $b \in {\mathcal{S}} (\Z, A, \af).$

Now let $b \in {\mathcal{S}} (\Z, A, \af).$
Since $\sum_{n \in \Z} \| b_n \| < \I,$
we can write $b = \sum_{n \in \Z} b_n u_n,$
with absolute convergence in $F^p_{\mathrm{r}} (\Z, A, \af).$
For $\ld \in \R,$ set $g (\ld) = {\widehat{\af}}_{\exp (i \ld)} (b).$
Then
\[
g (\ld) = \sum_{n \in \Z} \exp (- i n \ld) b_n u_n.
\]
It is easy to prove by induction on~$r$ that the function
\[
\ld \mapsto \sum_{n \in \Z} (- i n)^r \exp (- i n \ld) b_n u_n
\]
is well defined
(the series converges absolutely in $F^p_{\mathrm{r}} (\Z, A, \af)$),
takes values in ${\mathcal{S}} (\Z, A, \af),$
and is equal to $g^{(r)} (\ld).$
Thus $b$ is smooth.
\end{proof}

Recall
(see the second paragraph of Definition~1.1 of~\cite{Sw0})
that if $B$ is a unital Banach algebra and $C$ is a dense subalgebra,
then $C$ is spectrally invariant in~$B$
if whenever $c \in C$ has an inverse $b \in B,$
then $b \in C.$
In the nonunital case,
we ask that the unitization of~$C$
be spectrally invariant in the unitization of~$B.$
(It is shown in Lemma~1.2 of~\cite{Sw0}
that, under good conditions, $C$ is spectrally invariant in~$B$
\ifo\  $C$ is closed under holomorphic functional calculus in~$B.$)

\begin{cor}\label{C-SpecInv}
Let $p \in [1, \I),$
let $A$ be a $L^p$~operator algebra,
and let $\af \in \Aut (A)$ be isometric.
Then ${\mathcal{S}} (\Z, A, \af)$ is
spectrally invariant in $F^p_{\mathrm{r}} (\Z, A, \af).$
\end{cor}

\begin{proof}
It is easy to check that if $b \in F^p_{\mathrm{r}} (\Z, A, \af)$
is invertible and smooth,
then $b^{-1}$ is smooth.
Apply Theorem~\ref{T-SmoothIffSchw}.
\end{proof}

In fact, Theorem~2.2 of~\cite{Sw1}
shows that ${\mathcal{S}} (\Z, A, \af)$ is
strongly spectrally invariant in $F^p_{\mathrm{r}} (\Z, A, \af)$
in the sense of Definition~1.2 of~\cite{Sw1}.

For a locally $m$-convex Fr\'{e}chet algebra~$A,$
we let $RK_* (A)$ denote the representable K-theory of~$A,$
as defined in~\cite{Ph1}.

\begin{cor}\label{C-KThyIso}
Let $p \in [1, \I),$
let $A$ be a $L^p$~operator algebra,
and let $\af \in \Aut (A)$ be isometric.
Let $\kp_{\I}$ be as in Definition~\ref{D-SmoothCrPrd}.
Then
\[
(\kp_{\I})_* \colon RK_* ( {\mathcal{S}} (\Z, A, \af) )
     \to RK_* ( F^p_{\mathrm{r}} (\Z, A, \af) )
\]
is an isomorphism.
\end{cor}

\begin{proof}
This is true for any \ct\  inclusion of a spectrally invariant
locally $m$-convex Fr\'{e}chet algebra,
by Lemma 1.1.9(1) of~\cite{PS}.
\end{proof}

\begin{thm}\label{T-PV}
Let $p \in [1, \I),$
let $A$ be a $L^p$~operator algebra,
and let $\af \in \Aut (A)$ be isometric.
Then we have the following natural six term exact sequence in
$K$-theory,
in which $\ep \colon A \to F^p_{\mathrm{r}} (\Z, A, \alpha)$ is the
inclusion of Remark~\ref{R-Ident}:
\begin{equation}\label{Eq:PVLp}
\begin{CD}
K_{0} (A) @>{{\operatorname{id}} - (\alpha^{-1})_{*}}>> K_{0} (A)
      @>{\ep_{*}}>>
      K_{0} \big( F^p_{\mathrm{r}} (\Z, A, \alpha) \big) \\
@A{\partial}AA @. @VV{\partial}V \\
K_{1} \big( F^p_{\mathrm{r}} (\Z, A, \alpha) \big) @<{\ep_{*}}<<
K_{1} (A) @<{{\operatorname{id}} - (\alpha^{-1})_{*}}<< K_{1}(A). \\
\end{CD}
\end{equation}
\end{thm}

\begin{proof}
The action satisfies the hypotheses of Theorem~2.6 of~\cite{PS}.
Therefore, with $\ep_{\I} \colon A \to {\mathcal{S}} (\Z, A, \alpha)$ being the inclusion analogous to that of Remark~\ref{R-Ident},
we get the following natural six term exact sequence:
\[
\begin{CD}
RK_{0} (A) @>{{\operatorname{id}} - (\alpha^{-1})_{*}}>> RK_{0} (A)
      @>{(\ep_{\I})_{*}}>>
      RK_{0} \big( {\mathcal{S}} (\Z, A, \alpha) \big) \\
@A{\partial}AA @. @VV{\partial}V \\
RK_{1} \big( {\mathcal{S}} (\Z, A, \alpha) \big)
 @<{(\ep_{\I})_{*}}<<
RK_{1} (A) @<{{\operatorname{id}} - (\alpha^{-1})_{*}}<< RK_{1}(A). \\
\end{CD}
\]
Now apply Corollary~\ref{C-KThyIso}
and use $\ep = \kp_{\I} \circ \ep_{\I},$
getting~(\ref{Eq:PVLp}) except with $RK_*$ in place of~$K_*.$
Since the algebras in~(\ref{Eq:PVLp})
are all Banach algebras,
Corollary 7.8 of~\cite{Ph1} allows us to replace $RK_*$ with~$K_*.$
\end{proof}

We suppose that there is also an analog of the Connes isomorphism
theorem
for reduced $L^p$~operator crossed products by isometric
actions of~$\R,$
but we have not investigated this question.
The analog for smooth crossed products
is Theorem 1.2.7 of~\cite{PS}.

\section{Realizing $\OP{d}{p}$ as a crossed
   product}\label{Sec_OpdCP}

\indent
In this section,
we show that $\OP{d}{p}$ is stably
isomorphic (in a suitable sense) to
a reduced $L^p$~operator crossed product by an isometric
action of~$\Z,$
in a manner analogous to the C*~case
(Section~2.1 of~\cite{Cu1}).
We use this isomorphism
to compute $K_* \big( \OP{d}{p} \big).$

The methods of the original computation,
in~\cite{Cu2},
do not seem to work.
We have only partial information about to what extent the description
of the K-theory of purely infinite simple \ca{s},
as in Section~1 of~\cite{Cu2},
carries over to purely infinite simple Banach algebras,
as in Definition~5.1 of~\cite{PhLp2a}.
See Corollary~5.15 of~\cite{PhLp2a} and Question~\ref{Q_3216_K1Inv0}.
This, however, is not the main difficulty,
since this description seems not to be essential
for the rest of the argument of~\cite{Cu2}.
Rather, the problem occurs
in the proof of Proposition~2.2 of~\cite{Cu2}.
The map $u \mapsto \ld_u$ in Proposition~2.1 of~\cite{Cu2},
from unitaries in ${\mathcal{O}}_d$
to endomorphisms of~${\mathcal{O}}_d,$
seems to have an analog in the $L^p$~situation
only when $u$ is an isometric bijection
($\| u \| \leq 1$ and $\| u^{-1} \| \leq 1$).
If $\| u \| > 1,$
we have potential trouble with the norm of $\ld_u (t_j^n)$
as $n \to \I,$
and if $\| u^{-1} \| > 1,$
we have potential trouble with the norm of $\ld_u (s_j^n)$
as $n \to \I.$
However, for $p \neq 2$ the isometries in $M_r (\C)$
are necessarily spatial.
(This is essentially Lamperti's Theorem~\cite{Lp};
see Theorem~6.9 and Lemma~6.15 of~\cite{PhLp1}.)
The isometry group is therefore not connected.
So we do not get the homotopy required in the
proof of Proposition~2.2 of~\cite{Cu2}.

We give careful details in this section
because it is often necessary to prove explicitly
that homomorphisms are isometric.
In particular,
we want it to be clear that we have provided such a proof
everywhere that one is needed.

The proof of stable isomorphism involves a number of objects,
including an algebra with notation for some of its elements,
an action of $\Z$ on this algebra,
and a number of homomorphisms.
We introduce notation for these as we proceed though
the construction,
stating the properties as lemmas.
Once established, notation implicitly stays in effect
for the rest of this section.

We represent the spatial $L^p$~UHF algebra of type $d^{\infty}$
(Example~\ref{E-pUHF}) following Example~3.8 of~\cite{PhLp2a},
taking $X_n$ there to be a $d$~point space
with normalized counting measure.
It is convenient here to use somewhat different notation.

\begin{ntn}\label{N-dinfty}
Fix $d \in \{ 2, 3, \ldots \}$
and $p \in [1, \I) \SM \{ 2 \}.$
(Here, and in the rest of the notation for this section,
we mostly suppress the dependence on both $d$ and~$p.$)
We describe a representation of
the spatial $L^p$~UHF algebra of
type $d^{\infty}$
(Example~\ref{E-pUHF}).
Define
\[
Z = \{ 0, 1, \ldots, d - 1 \}
\andeqn
X_0 = Z^{\N}
    = \prod_{n = 1}^{\I} Z.
\]
Equip $Z$ with the discrete topology
and normalized counting measure~$\ld,$
as in Example~\ref{E-Mpd}.
Write $\ld^n$ for the product of $n$ copies of $\ld$ on~$Z^n.$
Equip $X_0$ with the product topology
and the infinite product measure,
which we call~$\mu_0.$
For $r \in \N,$
we further let
$X_r = \prod_{n = r + 1}^{\I} Z,$
with the product topology and infinite product measure,
now called~$\mu_r,$
giving
$X_0 = Z^r \times X_r.$
Following Remark~\ref{R-LpTens},
we can then identify
\begin{equation}\label{Eq:LprTLpXr}
L^p (X_0, \mu_0) = L^p ( Z^r, \ld^r) \otimes_p L^p (X_r, \mu_r),
\end{equation}
and we can further decompose the first factor in various ways,
for example as
\[
L^p ( Z^r, \ld^r)
= L^p ( Z^{r - 1}, \, \ld^{r - 1}) \otimes_p L^p ( Z, \ld)
= L^p ( Z, \ld)^{\otimes r}.
\]
For $r \in \Nz,$
define
\[
1_{> r} = 1_{L^p (X_r, \mu_r)}
\andeqn
1_{\leq r} = 1_{L^p ( Z^r, \ld^r)},
\]
so that $1_{> 0} = 1_{\leq r} \otimes 1_{> r}.$
(These were called
$1_{ {\mathbb{N}}_{> r}}$ and $1_{ {\mathbb{N}}_{\leq r}}$
in Example~3.8 of~\cite{PhLp2a}.)
Take $M^p_d = L (L^p ( Z, \ld))$ as in Example~\ref{E-Mpd}.
For $r \in \N$ let $D_r^{(0)}$
be the set of all operators of the form
\begin{equation}\label{Eq:1TaT1}
1_{\leq r - 1} \otimes a \otimes 1_{> r}
 \in L \big( L^p ( Z^{r - 1}, \, \ld^{r - 1})
  \otimes_p L^p (Z, \ld) \otimes_p L^p (X_r, \mu_r) \big)
 = L \big( L^p ( X_0, \mu_0) \big)
\end{equation}
with $a \in M^p_d,$
and define
$\ph_r^{(0)} \colon
 M^p_d \to L \big( L^p ( X_0, \mu_0) \big)$
by $\ph_r ^{(0)}(a) = 1_{\leq r - 1} \otimes a \otimes 1_{> r}$
as in~(\ref{Eq:1TaT1}).
Let $D_r$ be the subalgebra of $L \big( L^p ( X_0, \mu_0) \big)$
generated by $\bigcup_{k = 1}^r D_k^{(0)},$
and define
$\ph_r \colon
 L \big( L^p (Z^r, \ld^r) \big) \to L \big( L^p ( X_0, \mu_0) \big)$
by $\ph_r (a) = a \otimes 1_{> r}$ according to
the tensor factorization in~(\ref{Eq:LprTLpXr}).
Now set $D = {\ov{ \bigcup_{r = 1}^{\I} D_r }}.$
\end{ntn}

\begin{lem}\label{L-DINisRight}
The objects defined in Notation~\ref{N-dinfty}
have the following properties:
\begin{enumerate}
\item\label{L-DINisRight-1}
For every $r \in \N,$ the map $\ph_r^{(0)}$ is an isometric bijection
from $M^p_d$ to $D_r^{(0)}.$
\item\label{L-DINisRight-1b}
For every $r \in \N,$ the map $\ph_r$ is an isometric bijection
from $L \big( L^p (Z^r, \ld^r) \big)$ to $D_r.$
\item\label{L-DINisRight-2}
We have $D_1 \S D_2 \S \cdots.$
\item\label{L-DINisRight-3}
The algebra~$D$ is isometrically isomorphic to
the spatial $L^p$~UHF algebra of
type $d^{\infty},$
as in Example~\ref{E-pUHF}.
\item\label{L-DINisRight-5}
The algebra~$D$ is separable.
\item\label{L-DINisRight-4}
For $r \in \N$
and $c_1, c_2, \ldots, c_r \in M^p_d,$
we have
\[
\ph_1^{(0)} (c_1) \ph_2^{(0)} (c_2) \cdots \ph_r^{(0)} (c_r)
 = c_1 \otimes c_2 \otimes
   \cdots \otimes c_r \otimes 1_{> r}
\]
with respect to the tensor factorization
\[
L^p (X_0, \mu_0)
  = L^p (Z, \ld)^{\otimes r} \otimes_p L^p (X_r, \mu_r).
\]
\end{enumerate}
\end{lem}

\begin{proof}
For~(\ref{L-DINisRight-1}),
for all $a \in L \big( L^p (Z, \ld) \big)$
we have $\| 1_{\leq r - 1} \otimes a \otimes 1_{> r} \| = \| a \|$
by Remark~\ref{R-LpTens}.
The proof of~(\ref{L-DINisRight-1b})
is the same.
Part~(\ref{L-DINisRight-2}) is obvious.
Given parts (\ref{L-DINisRight-1b}) and~(\ref{L-DINisRight-2}),
part~(\ref{L-DINisRight-3})
follows immediately from Definition~3.5 and Theorem 3.10 of~\cite{PhLp2a}.
Part~(\ref{L-DINisRight-5}) follows from the fact
that the $D_n$ are \fd\  and their union is dense in~$D.$
Part~(\ref{L-DINisRight-4}) is clear.
\end{proof}

\begin{ntn}\label{N-ActionOnD}
We follow Notation~\ref{N-dinfty}.
For a set~$S,$
let $\nu_S$ be counting measure on~$S.$
(Thus $\ld = d^{-1} \nu_Z.$)
Define $X = \Nz \times X_0,$
and equip $X$ with the measure $\mu = \nu_{\Nz} \times \mu_0$
and the product topology.
Define a function $h \colon X \to X$ as follows.
For $m \in \Nz$ and $k_1, k_2, \ldots \in Z,$
write $m = d m_0 + k_0$ with $m_0 \in \Nz$ and $k_0 \in Z,$
and then set
\[
h (m, k_1, k_2, \ldots)
  = (m_0, k_0, k_1, k_2, \ldots ).
\]
Define (justification in
\Lem{L-ActionIsRight}(\ref{L-ActionIsRight-2}) below)
$v \in \LLp$
by
\[
(v \xi) (x) = d^{1 / p} \xi (h^{-1} (x))
\]
for $\xi \in L_p (X, \mu)$ and $x \in X.$
\end{ntn}

\begin{ntn}\label{N-MatUnit}
Recall the algebras ${\ov{M}}_{S}^p$ of Example~\ref{E-MatpS},
and the matrix unit notation from there
and from Example~\ref{E-Mpd}.
Thus, the standard matrix units for
$M^p_d$ are $( e_{j, k} )_{j, k = 0}^{d - 1}.$
Define $B \S \LLp$ by $B = {\ov{M}}_{\Nz}^p \otimes_p D$
as in Example~\ref{E-LpTP}.

For $a \in {\ov{M}}_{\Nz}^p,$
we have the element
\[
a \otimes 1_{> 0}
 \in L \big( L^p ( \Nz \times X_0, \, \nu_{\Nz} \times \mu_0 ) \big)
 = \LLp
\]
as in Remark~\ref{R-LpTens},
but recalling that $1_{L^p (X_0, \mu_0)} = 1_{> 0}$
(Notation~\ref{N-dinfty}).
Similarly, for $b \in D$ we get $1 \otimes b \in \LLp.$
For $c \in M^p_d,$
we write $1 \otimes \ph_r^{(0)} (c)$
as
\[
1 \otimes \ph_r^{(0)} (c)
 = 1 \otimes 1_{\leq r - 1} \otimes c \otimes 1_{> r}.
\]
In the last expression, $c$~is really in position $r + 1$:
the first tensor factor is $1_{l^p (\Nz)},$
and $1_{\leq r - 1}$ is the tensor product of $r - 1$ copies
of $1_{M^p_d}.$
We use similar notation for products
as in \Lem{L-DINisRight}(\ref{L-DINisRight-4}).

For $n \in \Nz,$
we further define $f_n \in B$
by
\[
f_n = \sum_{m = 0}^{d^n - 1} e_{m, m} \otimes 1_{> 0}.
\]
\end{ntn}

\begin{lem}\label{L-ActionIsRight}
The objects defined in Notation~\ref{N-ActionOnD}
and Notation~\ref{N-MatUnit}
have the following properties:
\begin{enumerate}
\item\label{L-ActionIsRight-0}
The algebra~$B$ is separable.
\item\label{L-ActionIsRight-1}
The map $h$ is a \hme,
with inverse given by
\[
h^{-1} (m, k_1, k_2, k_3, \ldots)
  = (d m + k_1, \, k_2, \, k_3, \, \ldots )
\]
for $m \in \Nz$ and $k_1, k_2, \ldots \in Z.$
\item\label{L-ActionIsRight-2}
$v$ is a well defined bijective isometry in $\LLp$
which is spatial in the sense of
Definition~6.4 of~\cite{PhLp1}.
\item\label{L-ActionIsRight-3}
$v B v^{-1} = B.$
\item\label{L-ActionIsRight-4}
For
\[
j, k \in \Nz,
\,\,\,\,\,
r \in \N,
\andeqn
l_1, m_1, l_2, m_2, \ldots, l_r, m_r \in Z,
\]
we have
\begin{align*}
\lefteqn{v^{-1} \big( e_{j, k} \otimes e_{l_{1}, m_{1}}
   \otimes e_{l_{2}, m_{2}}
   \otimes \cdots
   \otimes e_{l_{r}, m_{r}}
   \otimes 1_{> r} \big) v}
  \\
& \hspace*{2em} \mbox{}
 = e_{d j + l_1, \, d k + m_1} \otimes e_{l_{2}, m_{2}}
   \otimes e_{l_{3}, m_{3}}
   \otimes \cdots
   \otimes e_{l_{r}, m_{r}}
   \otimes 1_{> r - 1}.
\end{align*}
If we write
$j = d j_0 + l_0$ and $k = d k_0 + m_0$ with $j_0, k_0 \in \Nz$
and $l_0, m_0 \in Z,$
then (allowing now $r = 0$)
\begin{align*}
\lefteqn{v \big( e_{j, k} \otimes e_{l_{1}, m_{1}}
   \otimes e_{l_{2}, m_{2}}
   \otimes \cdots
   \otimes e_{l_{r}, m_{r}}
   \otimes 1_{> r} \big) v^{-1}}
  \\
& \hspace*{2em} \mbox{}
 = e_{j_0, k_0} \otimes e_{l_{0}, m_{0}}
   \otimes e_{l_{1}, m_{1}}
   \otimes \cdots
   \otimes e_{l_{r}, m_{r}}
   \otimes 1_{> r + 1}.
\end{align*}
\item\label{L-ActionIsRight-5a}
For $n \in \Nz,$ the element $f_n$ is an idempotent.
\item\label{L-ActionIsRight-5b}
For $n \in \Nz,$ we have $f_{n + 1} f_n = f_{n} f_{n + 1} = f_n.$
\item\label{L-ActionIsRight-5}
For $n \in \Nz,$
we have $v^{-1} f_n v = f_{n + 1}.$
\item\label{L-ActionIsRight-6}
$\| v b v^{-1} \| = \| b \|$ for all $b \in B.$
\item\label{L-ActionIsRight-7}
We have
$f_0 B f_0 \subset f_1 B f_1 \subset f_2 B f_2 \subset \cdots$
and $B = {\ov{\bigcup_{n = 0}^{\I} f_n B f_n}}.$
\end{enumerate}
\end{lem}

\begin{proof}
Part~(\ref{L-ActionIsRight-0})
follows from separability of~$M^p_{\Nz}$
(which is clear from its definition)
and of~$D$
(\Lem{L-DINisRight}(\ref{L-DINisRight-5})).
Part~(\ref{L-ActionIsRight-1}) is clear.

We prove part~(\ref{L-ActionIsRight-2}).
Given part~(\ref{L-ActionIsRight-1}),
and since the formula for $v$ gives $v (\ch_E) = \ch_{h (E)},$
all we need to check is that
if $E \S X$ is measurable,
then $\mu (h (E)) = d^{-1} \mu (E).$
This is true for all sets of the form
\[
E = \{ m \} \times \{ ( k_1, k_2, \ldots, k_r ) \} \times X_r
\]
for $m \in \Nz,$ $r \in \N,$ and $k_1, k_2, \ldots, k_r \in Z.$
The characteristic functions of sets of this type
are \ct.
Moreover, for every $g \in C_{\mathrm{c}} (X),$
there is a compact set $K \S X$
and a sequence of continuous functions~$g_n,$
each a linear combination of functions $\ch_E$ with $E$ as above
and $E \S K,$
such that $g_n \to g$ uniformly.
The statement now follows from the Riesz Representation Theorem.

Part~(\ref{L-ActionIsRight-4}) is a computation.
Part~(\ref{L-ActionIsRight-3}) follows from
part~(\ref{L-ActionIsRight-4}),
because the elements considered span a dense subspace of~$B.$
Parts (\ref{L-ActionIsRight-5a})
and~(\ref{L-ActionIsRight-5b})
are immediate.
Part~(\ref{L-ActionIsRight-5}) also follows from
part~(\ref{L-ActionIsRight-4}).
To prove part~(\ref{L-ActionIsRight-6}),
use the fact that $\| v \| = \| v^{-1} \| = 1,$
which is a consequence of part~(\ref{L-ActionIsRight-2}).

We prove~(\ref{L-ActionIsRight-7}).
For $n \in \Nz,$
the inclusion $f_n B f_n \subset f_{n + 1} B f_{n + 1}$
follows from~(\ref{L-ActionIsRight-5b}).
Now define $z_n \in M^p_{\Nz}$ by
$z_n = \sum_{m = 0}^{d^n - 1} e_{m, m}$
for $n \in \Nz,$
so that $f_n = z_n \otimes 1_{> 0}.$
Then it follows directly from the definition
in Example~\ref{E-MatpS}
that
\[
{\ov{M}}^p_{\Nz} = {\ov{\bigcup_{n = 0}^{\I} z_n M^p_{\Nz} z_n}}.
\]
The desired conclusion now follows from the density in~$B$
of the linear span of the elementary tensors.
\end{proof}

\begin{ntn}\label{N-DICrPrd}
We follow Notation~\ref{N-dinfty}, Notation~\ref{N-ActionOnD},
and Notation~\ref{N-MatUnit}.
We define $\bt \in \Aut (B)$ by $\bt (b) = v b v^{-1}$
for all $b \in B.$
Since $\bt$ is an isometric automorphism of~$B$
(by \Lem{L-ActionIsRight}(\ref{L-ActionIsRight-3})
and \Lem{L-ActionIsRight}(\ref{L-ActionIsRight-6})),
we may define $A = F^p (\Z, B, \bt).$
We further identify an element $b \in B$
with $b u_0 \in F^p (\Z, B, \bt),$
as in Remark~\ref{R-Ident}.
(Since $B$ is not unital,
we do not have $u_n \in F^p (\Z, B, \bt).$)
In particular, the elements $f_n$
of Notation~\ref{N-MatUnit} are considered to be in~$A.$
For $a \in A$ and $n \in \Z,$
we write $u_n a$ for $L_n (a)$ and $a u_n$ for $R_n (a),$
as in Notation~\ref{N-Multug}.

We let $v$ also stand for the isometric \rpn{}
of $\Z$ on $L^p (X, \mu)$ given by $n \mapsto v^n,$
and we define
(justification in \Lem{L-CrPrdRight} below)
$\pi = v \ltimes \id_B \colon F^p (\Z, B, \bt) \to \LLp$
to be the representation associated,
as in Theorem~\ref{T-UPropFull}(\ref{T-UPropFull-1}),
with the covariant \rpn{} $(v, \id_B)$
of $(\Z, B, \bt).$
\end{ntn}

\begin{lem}\label{L-CrPrdRight}
The objects defined in Notation~\ref{N-DICrPrd}
have the following properties:
\begin{enumerate}
\item\label{L-CrPrdRight-5}
The algebra~$A$ is separable.
\item\label{L-CrPrdRight-1}
The pair $(v, \id_B)$ is a
contractive covariant \rpn{} of $(\Z, B, \bt).$
\item\label{L-CrPrdRight-2}
The \hm\  $\pi$ exists and is contractive.
\item\label{L-CrPrdRight_uv}
$u_{-1} b u_1 = v^{-1} b v$
and $u_1 b u_{-1} = v b v^{-1}$
for all $b \in B.$
\item\label{L-CrPrdRight-4}
We have
$f_0 A f_0 \subset f_1 A f_1 \subset f_2 A f_2 \subset \cdots$
and $A = {\ov{\bigcup_{n = 0}^{\I} f_n A f_n}}.$
\end{enumerate}
\end{lem}

\begin{proof}
Part~(\ref{L-CrPrdRight-5})
follows from separability of~$B$
(\Lem{L-ActionIsRight}(\ref{L-ActionIsRight-0}))
and the fact that the linear span of all $b u_n,$
with $b \in B$ and $n \in \Z,$
is dense in~$A$
(by Theorem~\ref{T-UPropFull}(\ref{T-UPropFull-3})).
Part~(\ref{L-CrPrdRight-1}) is immediate from
the definitions.
Part~(\ref{L-CrPrdRight-2}) follows
from part~(\ref{L-CrPrdRight-1})
and Theorem~\ref{T-UPropFull}(\ref{T-UPropFull-1}).
Part~(\ref{L-CrPrdRight_uv}) is the definition
of the product in $F^p (\Z, B, \bt).$

We prove~(\ref{L-CrPrdRight-4}).
For $n \in \Nz,$
the inclusion $f_n A f_n \subset f_{n + 1} A f_{n + 1}$
follows from
\Lem{L-ActionIsRight}(\ref{L-ActionIsRight-5b}).
For the density statement,
it suffices to show that for $b \in B$ and $k \in \Z,$
we have $b u_k \in {\ov{\bigcup_{n = 0}^{\I} f_n A f_n}}.$
Let $\ep > 0.$
Choose $m \in \Nz$ and $c \in f_m B f_m$
such that $\| c - b \| < \ep.$
\Wolog{} $m \geq - k.$
Then $u_k^{-1} f_m u_k = f_{m + k}$
by (\ref{L-CrPrdRight_uv})
and Lemma \ref{L-ActionIsRight}(\ref{L-ActionIsRight-5}).
Take $n = \max (m, m + k).$
Then, using Lemma \ref{L-ActionIsRight}(\ref{L-ActionIsRight-5b})
at the third step,
\[
c u_k
 = f_m c f_m u_k
 = f_m c u_k f_{m + k}
 = f_n f_m c u_k f_{m + k} f_n
 \in f_n A f_n,
\]
and $\| c u_k - b u_k \| < \ep.$
\end{proof}

\begin{ntn}\label{N-Omega}
Adopt the notation of Example~\ref{E-Odp}.
Let $\om \colon M_d^p \to L_d$
be the unital \hm\  determined by $\om (e_{j, k}) = s_j t_k$
for $j, k \in \{ 0, 1, \ldots, d - 1 \}.$
\end{ntn}

The following proposition is an easy consequence
of Theorems 7.2 and~7.7 of~\cite{PhLp1}
when $p \neq 1,$
but the case $p = 1$ is not explicitly in~\cite{PhLp1}.

\begin{prp}\label{P_3922_Spt}
Let $p \in [1, \I) \SM \{ 2 \}$
and let $d \in \{ 2, 3, \ldots \}.$
Let $\YCN$ be a \sfm.
Let $\rh \colon L_d \to L ( L^p (Y, \nu) )$
be a unital \hm.
Suppose that for $j = 0, 1, \ldots, d - 1,$
we have $\| \rh (s_j) \| \leq 1$ and $\| \rh (t_j) \| \leq 1.$
Suppose further that,
with $\om$ as in Notation~\ref{N-Omega}, the \hm{}
$\rh \circ \om \colon M_d^p \to L ( L^p (Y, \nu) )$
is contractive.
Then $\rh$ is spatial
in the sense of Definition 7.4(2) of~\cite{PhLp1}.
\end{prp}

\begin{proof}
The implication from (4) to~(5)
in Theorem~7.2 of~\cite{PhLp1}
provides a measurable partition
$Y = \coprod_{j = 0}^{d - 1} Y_j$
such that $\rh (s_j t_j)$
is multiplication by $\ch_{Y_j}$ for $j = 0, 1, \ldots, d - 1.$
Therefore $\rh$ is disjoint
in the sense of Definition 7.4(1) of~\cite{PhLp1}.
For $j = 0, 1, \ldots, d - 1,$
the relation $t_j s_j = 1$
and the bounds $\| \rh (s_j) \| \leq 1$ and $\| \rh (t_j) \| \leq 1$
imply that $\rh (s_j)$ is an isometry.
Lemma 7.12 of~\cite{PhLp1} now implies that $\rh$ is spatial.
\end{proof}

We will need a related nonunital result.

\begin{lem}\label{L-UniqLd}
Let the hypotheses and notation be as in Proposition~\ref{P_3922_Spt},
except that we require that $\rh$ be nonzero
but not necessarily unital.
Further assume that $L^p (Y, \nu)$ is separable.
Then $\rh$ extends uniquely to an isometric injective \hm{}
${\ov{\rh}} \colon \OP{d}{p}
          \to L \big( L^p (Y, \nu) \big).$
\end{lem}

\begin{proof}
It is clear that $e = \rh (1)$ is an idempotent in $\LLp.$
Set $E = {\mathrm{ran}} (e).$
The hypotheses imply that $\| e \| = 1.$
It follows from Theorem~3 in Section~17 of~\cite{Lc}
that there is a \msp\  $(Y_0, {\mathcal{C}}_0, \nu_0)$
such that $E$ is isometrically isomorphic to $L^p (Y_0, \nu_0).$
Since $L^p (Y, \nu)$ is separable, so is~$E,$
and therefore we may take $\nu_0$ to be \sft.
(See the corollary to Theorem~3 in Section~15 of~\cite{Lc}.)

The corestriction $\rh_0 \colon L_d \to L (E)$
is now a unital \rpn{} of $L_d$ on $L^p (Y_0, \nu_0)$
which satisfies the hypotheses of Proposition~\ref{P_3922_Spt}.
So $\rh_0$ is spatial.
Now apply Theorem~8.7 of~\cite{PhLp1}.
\end{proof}

\begin{lem}\label{L-OdIsoMdOd}
Let $\YCN$ be a \msp{}
such that $L^p (Y, \nu)$ is separable,
and let
$\rh \colon \OP{d}{p} \to L \big( L^p (Y, \nu) \big)$
be an isometric \hm\  (not necessarily unital).
Let $(T, {\mathcal{D}}, \et)$
be a \msp{}
such that $L^p (T, \et)$ is separable,
and let
$\gm \colon M^p_d \to L ( L^p (T, \et) )$
be an isometric \hm\  (not necessarily unital).
Then there is an isometric isomorphism
$\ps \colon \OP{d}{p}
 \to \gm ( M^p_d ) \otimes_p \rh \big( \OP{d}{p} \big)$
such that for $j = 0, 1, \ldots, d - 1,$
we have
\[
\ps (s_j) = \sum_{l = 0}^{d - 1} \gm (e_{j, l}) \otimes \rh (s_l)
\andeqn
\ps (t_j) = \sum_{l = 0}^{d - 1} \gm (e_{l, j}) \otimes \rh (t_l).
\]
Moreover, for
$j, k \in \{ 0, 1, \ldots, d - 1 \}$ and $a \in \OP{d}{p},$
we have $\ps (s_j a t_k) = \gm (e_{j, k}) \otimes \rh (a).$
\end{lem}

In particular,
the $L^p$~operator analog of the C*~minimal tensor product of
$M^p_d$ and $\OP{d}{p}$
does not depend on how these algebras are represented,
at least if we restrict to separable $L^p$~spaces.

\begin{proof}[Proof of \Lem{L-OdIsoMdOd}]
Recall that $p \in [1, \I) \SM \{ 2 \}.$

The hypotheses imply that both $\rh (1)$ and $\gm (1)$
are idempotents of norm~$1.$
As in the proof of \Lem{L-UniqLd},
we can use Theorem~3 in Section~17 of~\cite{Lc}
to see that $\rh (1) L^p (Y, \nu)$
and $\gm (1) L^p (T, \et)$ are
isometrically isomorphic to $L^p$~spaces
of \sfm s.
Taking the corestrictions of $\rh$ and $\gm,$
we thus reduce to the case in which both $\rh$ and~$\gm$
are unital.
It now follows from Proposition~\ref{P_3922_Spt}
that $\rh |_{L_d}$ is spatial,
and from Theorem~7.2 of~\cite{PhLp1}
that $\gm$ is spatial.
In particular,
we can write $Y = \coprod_{j = 0}^{d - 1} Y_j$
in such a way that for
for $j = 0, 1, \ldots, d - 1,$
the operator $\rh (s_j)$
is a spatial isometry with domain support~$Y,$
range support~$Y_j,$
and reverse $\rh (t_j)$
(in the sense of Definition 6.13 of~\cite{PhLp1}).
Also, we can write $T = \coprod_{j = 0}^{d - 1} T_j$
in such a way that for
$k, l \in \{ 0, 1, \ldots, d - 1 \},$
the operator $\gm (e_{k, l})$
is a spatial partial isometry with domain support~$T_l,$
range support~$T_k,$
and reverse $\gm (e_{l, k}).$

For $j = 0, 1, \ldots, d - 1,$
set
\[
v_j = \sum_{l = 0}^{d - 1} \gm (e_{j, l}) \otimes \rh (s_l)
\andeqn
w_j = \sum_{l = 0}^{d - 1} \gm (e_{l, j}) \otimes \rh (t_l).
\]
One easily checks that the elements
$v_j,$ playing the role of~$s_j,$
and $w_j,$ playing the role of~$t_j,$
satisfy the relations
(\ref{Eq:Leavitt1}), (\ref{Eq:Leavitt2}), and~(\ref{Eq:Leavitt3}).
Therefore there is a \hm{}
$\ph \colon L_d \to
 \gm ( M^p_d ) \otimes_p
            \rh \big( \OP{d}{p} \big)$
such that $\ph (s_j) = v_j$ and $\ph (t_j) = w_j$
for $j = 0, 1, \ldots, d - 1.$
Lemma~6.20 of~\cite{PhLp1}
implies that $\gm (e_{k, l}) \otimes \rh (s_j)$
is a spatial partial isometry with domain support~$T_l \times Y,$
range support~$T_k \times Y_j,$
and reverse $\gm (e_{l, k}) \otimes \rh (t_j).$
Now Lemma~3.8 of~\cite{PhLp2b}
implies that for $j = 0, 1, \ldots, d - 1,$
the operator $v_j$ is a spatial isometry
with
reverse~$w_j.$
That is, $\ph$ is a spatial \rpn{}
in the sense of Definition 7.4(2) of~\cite{PhLp1}.
It now follows from Theorem~8.7 of~\cite{PhLp1}
that $\ph$ extends to an isometric \hm{}
$\ps \colon \OP{d}{p}
 \to \gm ( M^p_d ) \otimes \rh \big( \OP{d}{p} \big).$

We now prove the formula
$\ps (s_j a t_k) = \gm (e_{j, k}) \otimes \rh (a).$
A calculation shows that
$\ps (s_j t_k) = \gm (e_{j, k}) \otimes \rh (1)$
for $j, k \in \{ 0, 1, \ldots, d - 1 \}.$
Now let $j, k, r \in \{ 0, 1, \ldots, d - 1 \}.$
Then
\[
\ps (s_j s_r t_k)
 = \ps (s_j) \ps (s_r t_k)
 = \sum_{l = 0}^{d - 1} [ \gm (e_{j, l}) \otimes \rh (s_l)]
                [ \gm (e_{r, k}) \otimes \rh (1)]
 = \gm (e_{j, k}) \otimes \rh (s_r)
\]
and
\[
\ps (s_j t_r t_k)
 = \ps (s_j t_r) \ps (t_k)
 = \sum_{l = 0}^{d - 1} [ \gm (e_{j, r}) \otimes \rh (1)]
                [ \gm (e_{l, k}) \otimes \rh (t_l)]
 = \gm (e_{j, k}) \otimes \rh (t_r).
\]
For $j, k \in \{ 0, 1, \ldots, d - 1 \},$
the map $a \mapsto s_j a t_j$ is a (nonunital)
continuous endomorphism of~$\OP{d}{p}.$
Since $s_0, s_1, \ldots, s_{d - 1}, t_0, t_1, \ldots, t_{d - 1}$
generate $\OP{d}{p}$
as a Banach algebra,
we conclude that
$\ps (s_j a t_j) = \gm (e_{j, j}) \otimes \rh (a)$
for $j = 0, 1, \ldots, d - 1$ and $a \in \OP{d}{p}.$
For $k = 0, 1, \ldots, d - 1,$
we then get
\[
\ps (s_j a t_k)
 = \ps (s_j a t_j) \ps (s_j t_k)
 = [\gm (e_{j, j}) \otimes \rh (a)] [\gm (e_{j, k}) \otimes \rh (1)]
 = \gm (e_{j, k}) \otimes \rh (a),
\]
as desired.

It remains to prove that $\ps$ is surjective.
The previous paragraph implies that the range of~$\ps$
contains $\gm (e_{j, k}) \otimes \rh (a)$
for all $j, k \in \{ 0, 1, \ldots, d - 1 \}$ and $a \in \OP{d}{p}.$
It follows that $\ps$ has dense range.
Since $\ps$ is isometric, it is surjective.
\end{proof}

\begin{lem}\label{L-MapToCorner}
There exists an isometric isomorphism
$\sm \colon \OP{d}{p} \to f_0 A f_0$
such that for $j = 0, 1, \ldots, d - 1$ we have
(recalling Notation~\ref{N-dinfty},
Notation~\ref{N-ActionOnD}, and Notation~\ref{N-DICrPrd})
\begin{equation}\label{Eq:DfnOfSm}
\sm (s_j) = u_1 ( e_{j, 0} \otimes 1_{> 0} )
\andeqn
\sm (t_j) = ( e_{0, j} \otimes 1_{> 0} ) u_{-1}.
\end{equation}
\end{lem}

\begin{proof}
We first check that the elements in~(\ref{Eq:DfnOfSm})
are in $f_0 A f_0.$
Using $u_1^{-1} f_0 u_1 = f_1$
(which follows from
Lemma \ref{L-CrPrdRight}(\ref{L-CrPrdRight_uv})
and Lemma \ref{L-ActionIsRight}(\ref{L-ActionIsRight-5}))
at the first step,
and the definitions of $f_0$ and $f_1$ at the second step,
for $j = 0, 1, \ldots, d - 1$
we get
\[
f_0 u_1 ( e_{j, 0} \otimes 1_{> 0} ) f_0
  = u_1 f_1 ( e_{j, 0} \otimes 1_{> 0} ) f_0
  = u_1 ( e_{j, 0} \otimes 1_{> 0} ).
\]
Therefore $u_1 ( e_{j, 0} \otimes 1_{> 0} ) \in f_0 A f_0.$
The proof that $( e_{0, j} \otimes 1_{> 0} ) u_{-1} \in f_0 A f_0$
is similar.

We next check that the elements in~(\ref{Eq:DfnOfSm})
give a unital \hm\  $\ta \colon L_d \to f_0 A f_0.$
This follows from the calculation,
for $j, k \in \{ 0, 1, \ldots, d - 1 \},$
\[
( e_{0, j} \otimes 1_{> 0} ) u_{-1}
   u_1 ( e_{k, 0} \otimes 1_{> 0} )
= e_{0, j} e_{k, 0}
= \begin{cases}
   0   & j \neq k
        \\
   f_0 & j = k,
\end{cases}
\]
and from the calculation
(using Lemma \ref{L-CrPrdRight}(\ref{L-CrPrdRight_uv})
and Lemma \ref{L-ActionIsRight}(\ref{L-ActionIsRight-5})
at the last step)
\[
\sum_{j = 0}^{d - 1}
  u_1 ( e_{j, 0} \otimes 1_{> 0} )
  ( e_{0, j} \otimes 1_{> 0} ) u_{-1}
= u_1 \left( \sssum{j = 0}{d - 1} e_{j, j} \otimes 1_{> 0} \right) u_{-1}
= u_1 f_1 u_{-1}
= f_0.
\]

Next, we claim that $\ta \circ \om \colon M^p_d \to f_0 A f_0$
is contractive.
To see this, first use
Lemma \ref{L-CrPrdRight}(\ref{L-CrPrdRight_uv})
and Lemma \ref{L-ActionIsRight}(\ref{L-ActionIsRight-4})
at the last step
to check that, for $j, k \in \{ 0, 1, \ldots, d - 1 \},$
we have
\begin{align*}
(\ta \circ \om) (e_{j, k})
& = \ta (s_j t_k)
  = u_1 ( e_{j, 0} \otimes 1_{> 0} )
  ( e_{0, k} \otimes 1_{> 0} ) u_{-1}
        \\
& = u_1 ( e_{j, k} \otimes 1_{> 0} ) u_{-1}
  = e_{0, 0} \otimes e_{j, k} \otimes 1_{> 1}.
\end{align*}
Therefore
$(\ta \circ \om) (a) = e_{0, 0} \otimes a \otimes 1_{> 1}$
for all $a \in M^p_d.$
By Remark~\ref{R-LpTens},
the \hm\  $\ta \circ \om$ is isometric from $M^p_d$ to~$B.$
Since the inclusion of $B$ in~$A$ is contractive
(by the inequality $\| a \| \leq \| a \|_1$ in \Lem{L:FGpCP}),
the claim follows.

The inequality $\| a \| \leq \| a \|_1$
and contractivity of multiplication by $u_1$ and $u_{-1}$
(Lemma~\ref{L-NormOfMultU})
imply that
$\| \ta (s_j) \| \leq 1$ and $\| \ta (t_j) \| \leq 1$
for $j = 0, 1, \ldots, d - 1.$

Since $A$ is separably representable
(using \Lem{L-CrPrdRight}(\ref{L-CrPrdRight-5})
and Proposition~\ref{P-SepImpSepRep}),
we may apply \Lem{L-UniqLd}
to conclude that $\ta$ extends to an
injective isometric \hm{}
$\sm \colon \OP{d}{p} \to f_0 A f_0.$
It remains only to prove that $\sm$ is surjective,
and for this it suffices to prove that
its range $\ran (\sm)$ is dense in $f_0 A f_0.$
By Theorem~\ref{T-UPropFull}(\ref{T-UPropFull-3}),
it is enough to show that
\begin{equation}\label{Eq:ForDRan}
f_0 \big( e_{l, m} \otimes 1_{\leq r - 1}
        \otimes e_{j, k} \otimes 1_{> r} \big)
     u_n f_0
 \in {\mathrm{ran}} (\sm)
\end{equation}
whenever $r \in \N,$ $n \in \Z,$
$j, k \in \{ 0, 1, \ldots, d - 1 \},$
and $l, m \in \Nz.$

We first claim that~(\ref{Eq:ForDRan}) holds when $n = 0.$
In this case, the expression is zero unless $l = m = 0.$
For $r \in \N$ and $j, k \in \{ 0, 1, \ldots, d - 1 \},$
we use Lemma \ref{L-CrPrdRight}(\ref{L-CrPrdRight_uv})
and Lemma \ref{L-ActionIsRight}(\ref{L-ActionIsRight-4})
to get
\begin{align*}
\lefteqn{
\sm (s_0)
   \big[ e_{0, 0} \otimes e_{0, 0} \otimes \cdots \otimes
    e_{0, 0} \otimes e_{j, k} \otimes 1_{> r} \big]
    \sm (t_0)}
    \\
& \hspace*{5em} \mbox{}
  = e_{0, 0} \otimes e_{0, 0} \otimes \cdots \otimes
    e_{0, 0} \otimes e_{0, 0} \otimes e_{j, k} \otimes
          1_{> r + 1}
\end{align*}
(shifting the tensor factor $e_{j, k}$ one space to the right).
Therefore
${\mathrm{ran}} (\sm)$ contains
\[
e_{0, 0} \otimes e_{j, k} \otimes 1_{> 1}
  = \sm (s_j t_k),
\,\,\,\,
e_{0, 0} \otimes e_{0, 0} \otimes e_{j, k} \otimes 1_{> 2},
\,\,\,\,
e_{0, 0} \otimes
	e_{0, 0} \otimes e_{0, 0} \otimes e_{j, k} \otimes 1_{> 3},
\,\,\,\,
\ldots.
\]
The closed subalgebra that these elements generate
is $\{ e_{0, 0} \otimes a \colon a \in D \},$
and the claim for $n = 0$ follows.

We next claim that for $n \in \N$ we have
$u_n f_0 = f_0 (u_1 f_0)^n$
and $f_0 u_{- n} = (f_0 u_{-1})^n f_0.$
The proof is by induction on~$n,$
and the case $n = 0$ is trivial.
For the induction step for the first, use
Lemma \ref{L-CrPrdRight}(\ref{L-CrPrdRight_uv})
and Lemma \ref{L-ActionIsRight}(\ref{L-ActionIsRight-5})
at the third step,
and \Lem{L-ActionIsRight}(\ref{L-ActionIsRight-5b})
at the last step, to get
\[
f_0 (u_1 f_0)^{n + 1}
  = u_n f_0 u_1 f_0
  = u_{n + 1} (u_1^{-1} f_0 u_1) f_0
  = u_{n + 1} f_1 f_0
  = u_{n + 1} f_0.
\]
Similarly,
\[
(f_0 u_{-1})^{n + 1} f_0
  = f_0 u_{-1} f_0 u_{-n}
  = f_0 (u_1^{-1} f_0 u_1) u_{- n - 1}
  = f_0 f_1 u_{- n - 1}
  = f_0 u_{- n - 1}.
\]
This proves the claim.

Now let $n \in \N.$
Then
\[
f_0 \big( e_{l, m} \otimes 1_{\leq r - 1}
        \otimes e_{j, k} \otimes 1_{> r} \big)
     u_n f_0
 = \big[ f_0 \big( e_{l, m} \otimes 1_{\leq r - 1}
        \otimes e_{j, k} \otimes 1_{> r} \big)
     f_0 \big] \cdot (u_1 f_0)^n.
\]
The first factor is in $\ran (\sm)$
by the case already done,
and $(u_1 f_0)^n = \sm (s_0)^n \in \ran (\sm),$
so (\ref{Eq:ForDRan}) holds.

Also,
using Lemma \ref{L-CrPrdRight}(\ref{L-CrPrdRight_uv})
and Lemma \ref{L-ActionIsRight}(\ref{L-ActionIsRight-5}),
\[
f_0 \big( e_{l, m} \otimes 1_{\leq r - 1}
        \otimes e_{j, k} \otimes 1_{> r} \big)
     u_{-n} f_0
 = f_0 \big( e_{l, m} \otimes 1_{\leq r - 1}
        \otimes e_{j, k} \otimes 1_{> r} \big)
     f_n u_{-n}.
\]
If $l \neq 0$ or $m \geq d^n,$
this expression is zero,
hence in $\ran (\sm).$
So we only need to consider
$f_0 \big( e_{0, m} \otimes 1_{\leq r - 1}
        \otimes e_{j, k} \otimes 1_{> r} \big) u_{-n} f_0,$
and under the assumption that there are
$q_0, q_1, \ldots, q_{n - 1} \in \{ 0, 1, \ldots, d - 1 \}$
such that $m = \sum_{i = 0}^{n - 1} q_i d^i.$
Then, using Lemma \ref{L-CrPrdRight}(\ref{L-CrPrdRight_uv})
and Lemma \ref{L-ActionIsRight}(\ref{L-ActionIsRight-4}) repeatedly
at the first step,
and $f_0 u_{- n} = (f_0 u_{-1})^n f_0$ at the second step,
\begin{align*}
& f_0 \big( e_{l, m} \otimes 1_{\leq r - 1}
        \otimes e_{j, k} \otimes 1_{> r} \big)
     u_{-n} f_0
     \\
& \hspace*{3em} {\mbox{}}
 = f_0 u_{-n} \big( e_{0, 0} \otimes e_{0, q_{n - 1}}
     \otimes e_{0, q_{n - 2}} \otimes e_{0, q_0} \otimes 1_{\leq r - 1}
        \otimes e_{j, k} \otimes 1_{> r + n} \big) f_0
     \\
& \hspace*{3em} {\mbox{}}
 = \sm (t_0)^n \cdot \big[ f_0 \big( e_{0, 0} \otimes e_{0, q_{n - 1}}
     \otimes e_{0, q_{n - 2}} \otimes e_{0, q_0} \otimes 1_{\leq r - 1}
        \otimes e_{j, k} \otimes 1_{> r + n} \big) f_0 \big].
\end{align*}
The case $n = 0$ of~(\ref{Eq:ForDRan}),
which we have already proved,
and the fact that $\ran (\sm)$ is an algebra,
show that this expression is in $\ran (\sm).$

Thus (\ref{Eq:ForDRan}) holds for all $j, k, l, m, r, n.$
This completes the proof of surjectivity.
\end{proof}

\begin{ntn}\label{N-DIForIntTw}
Fix a \sfm\  $\YCN$ such that $L^p (Y, \nu)$ is separable,
and a unital isometric \hm{}
$\af_0 \colon \OP{d}{p} \to L (L^p (Y, \nu))$
whose restriction to $L_d$ is spatial.
For $n \in \N,$
define
\[
\af_n \colon
    (M_d)^{\otimes n} \otimes_{\mathrm{alg}} \OP{d}{p}
   \to L \big( L^p (Z^n \times Y, \, \ld^n \times \nu) \big)
\]
to be the tensor product of $n$ copies of the
standard isomorphism $M_d \to L (L^p (Z, \ld))$ with~$\af_0.$
Make $(M_d)^{\otimes n} \otimes_{\mathrm{alg}} \OP{d}{p}$
into an $L^p$~operator algebra by defining $\| a \| = \| \af_n (a) \|$
for
$a \in (M_d)^{\otimes n} \otimes_{\mathrm{alg}} \OP{d}{p},$
and write
$(M^p_d)^{\otimes n} \otimes_p \OP{d}{p}$
for $(M_d)^{\otimes n} \otimes_{\mathrm{alg}} \OP{d}{p}$
equipped with this norm.

Let
$\ps_0 \colon \OP{d}{p}
       \to M^p_d \otimes_p \OP{d}{p}$
be the map $\ps$ of \Lem{L-OdIsoMdOd}.
Set $\et_0 = \id_{\OP{d}{p}}.$
Let $\sm_0 \colon \OP{d}{p} \to A$ be the map~$\sm$
of \Lem{L-MapToCorner},
followed by the inclusion of $f_0 A f_0$ in~$A.$
Define
$\ep_0 \colon \OP{d}{p}
       \to M^p_d \otimes_p \OP{d}{p}$
by $\ep_0 (a) = e_{0, 0} \otimes a$ for $a \in \OP{d}{p}.$

For $m \in \N,$
we adapt standard notation
by writing $\Ad (u_m)$ for the automorphism of~$A$
given by $a \mapsto u_m a u_{- m}$
(even though $u_m$ is not in~$A$).

Now, for $n \in \N,$ inductively define
(justifications in \Lem{L-DIIntTw} below)
\[
\ps_{n} = \id_{M^p_d} \otimes_p \ps_{n - 1}
  \colon (M^p_d)^{\otimes n} \otimes_p \OP{d}{p}
   \to (M^p_d)^{\otimes (n + 1)} \otimes_p \OP{d}{p},
\]
\[
\et_{n} = \ps_{n - 1} \circ \et_{n - 1}
  \colon \OP{d}{p}
   \to (M^p_d)^{\otimes n} \otimes_p \OP{d}{p},
\]
\[
\sm_{n} = \Ad (u_{-1}) \circ \sm_{n - 1} \circ \ps_{n - 1}^{-1}
   \colon (M^p_d)^{\otimes n} \otimes_p \OP{d}{p}
    \to A,
\]
and
\[
\ep_{n} = \ps_{n} \circ \ep_{n - 1} \circ \ps_{n - 1}^{-1}
  \colon (M^p_d)^{\otimes n} \otimes_p \OP{d}{p}
   \to (M^p_d)^{\otimes (n + 1)} \otimes_p \OP{d}{p}.
\]
\end{ntn}

\begin{lem}\label{L-DIIntTw}
For every $m \in \Z,$
the map $\Ad (u_m)$ in Notation~\ref{N-DIForIntTw} is a
well defined isometric automorphism of~$A.$
Moreover, for all $n \in \Nz,$
the maps defined in Notation~\ref{N-DIForIntTw} have the following
properties:
\begin{enumerate}
\item\label{L-DIIntTw-1}
$\ps_n$ is an isometric isomorphism.
\item\label{L-DIIntTw-3}
$\et_{n + 1} = ( \id_{M^p_d} \otimes \et_{n} ) \circ \ps_0.$
\item\label{L-DIIntTw-2}
$\et_{n + 1}$ is an isometric isomorphism.
\item\label{L-DIIntTw-4}
$\sm_n$ is isometric and ${\operatorname{ran}} (\sm_n) = f_n A f_n.$
\item\label{L-DIIntTw-5}
$\ep_n$ is an isometric \hm.
\item\label{L-DIIntTw-6}
$\sm_{n + 1} \circ \ep_{n} = \sm_{n}.$
\item\label{L-DIIntTw-7}
For all $a \in (M^p_d)^{\otimes n} \otimes_p \OP{d}{p},$
we have $\ep_n (a) = e_{0, 0} \otimes a.$
\end{enumerate}
\end{lem}

\begin{proof}
The part about $\Ad (u_m)$ follows from
Lemma~\ref{L-NormOfMultU}.

We prove the remaining statements simultaneously by induction on~$n.$

For $n = 0,$
part (\ref{L-DIIntTw-1})
is \Lem{L-OdIsoMdOd},
part~(\ref{L-DIIntTw-4})
is \Lem{L-MapToCorner},
part~(\ref{L-DIIntTw-5}) follows from Remark~\ref{R-LpTens},
and part~(\ref{L-DIIntTw-7}) is the definition of~$\ep_0.$

For parts (\ref{L-DIIntTw-3}) and~(\ref{L-DIIntTw-2}),
we use $\et_0 = \id_{\OP{d}{p}}$ to get
\[
\et_1
 = \ps_0 \circ \et_0
 = \ps_0
 = ( \id_{M^p_d} \otimes \et_{0} ) \circ \ps_0,
\]
as desired for~(\ref{L-DIIntTw-3}).
Also, $\et_1$ is isometric since $\ps_0$ is.

We now prove part~(\ref{L-DIIntTw-6}).
By Lemma~2.11 of~\cite{PhLp1},
it suffices to prove that
$(\sm_1 \circ \ep_0) (s_j) = \sm_0 (s_j)$
for $j = 0, 1, \ldots, d - 1.$

Fix $j \in \{ 0, 1, \ldots, d - 1 \}.$
We have $\ps_0 (s_0 s_j t_0) = e_{0, 0} \otimes s_j$
by \Lem{L-OdIsoMdOd},
so
$(\ps_0^{-1} \circ \ep_0) (s_j) = s_0 s_j t_0.$
Therefore,
using the formula for $\sm_0$ at the second step,
and Lemma \ref{L-CrPrdRight}(\ref{L-CrPrdRight_uv})
and Lemma \ref{L-ActionIsRight}(\ref{L-ActionIsRight-5})
at the third step,
\begin{align*}
(\sm_1 \circ \ep_0) (s_j)
& = \big( \Ad (u_{-1}) \circ \sm_0 \circ \ps_0^{-1} \circ \ep_0 \big)
           (s_0 s_j t_0)
  = (e_{0, 0} \otimes 1_{> 0}) u_1
           (e_{j, 0} \otimes 1_{> 0})
        \\
&
  = u_1 \left( \sssum{l = 0}{d - 1}
           e_{l, l} \otimes 1_{> 0} \right)
           (e_{j, 0} \otimes 1_{> 0})
        \\
& = u_1 (e_{j, 0} \otimes 1_{> 0})
  = \sm_0 (s_j).
\end{align*}
This completes the proof of the case $n = 0.$

Now let $n \in \N,$
and assume that all parts are known for $n - 1.$

The map $\ps_{n}$ is bijective because
$\ps_{n - 1}$ is bijective and $M^p_d$ is \fd.
The map $\et_{n + 1}$ is bijective
because $\ps_{n}$ and $\et_{n}$ are.
Now, to prove part~(\ref{L-DIIntTw-3}),
we use the induction hypothesis at the second step to get
\[
\et_{n + 1}
  = \ps_{n} \circ \et_n
  = \big( \id_{M^p_d} \otimes_p \ps_{n - 1} \big)
   \circ \big( \id_{M^p_d} \otimes \et_{n - 1} \big) \circ \ps_0
  = \big( \id_{M^p_d} \otimes \et_{n} \big) \circ \ps_0,
\]
as desired.

Now we prove that $\et_{n + 1}$ is isometric,
which will finish the proof of part~(\ref{L-DIIntTw-2}).
We know that $\et_{n}$ is isometric.
Therefore $\af_{n} \circ \et_{n}$ is isometric.
Proposition~\ref{P_3922_Spt}
implies that $(\af_{n} \circ \et_{n}) |_{L_d}$ is spatial.
For $j = 0, 1, \ldots, d - 1,$ one checks that
\begin{align*}
(\af_{n + 1} \circ \et_{n + 1}) (s_j)
& = \big[
   \af_{n + 1} \circ (\id_{M^p_d} \otimes \et_{n}) \circ \ps_0 \big]
     (s_j)
     \\
& = \sum_{l = 0}^{d - 1}
       e_{j, l} \otimes (\af_{n} \circ \et_{n}) (s_l)
  \in L \big( L^p (Z^{n + 1} \times Y, \, \ld^{n + 1} \times \nu) \big).
\end{align*}
This operator is a sum of spatial partial isometries
with disjoint domain supports and disjoint range supports.
By Lemma~3.8 of~\cite{PhLp2b},
it is a spatial partial isometry whose reverse
is the sum of the reverses of the summands, that is,
\[
\sum_{l = 0}^{d - 1}
       e_{l, j} \otimes (\af_{n} \circ \et_{n}) (t_l)
  = \big[
   \af_{n + 1} \circ (\id_{M^p_d} \otimes \et_{n}) \circ \ps_0 \big]
     (t_j)
  = (\af_{n + 1} \circ \et_{n + 1}) (t_j).
\]
Thus, $(\af_{n + 1} \circ \et_{n + 1}) |_{L_d}$
is a spatial representation of~$L_d.$
Theorem~8.7 of~\cite{PhLp1} now implies
that $\af_{n + 1} \circ \et_{n + 1}$ is isometric.
Therefore $\et_{n + 1}$ is isometric,
as desired.

It is now clear that $\ps_{n} = \et_{n + 1} \circ \et_{n}^{-1}$
is isometric,
which finishes part~(\ref{L-DIIntTw-1}).

The maps $\sm_{n}$ and $\ep_{n}$ are isometric,
since they are compositions of isometric maps.
The statement about the range of
$\sm_{n}$ follows from the induction hypothesis
and
$u_{-1} f_{n - 1} u_1 = f_{n}$
(Lemma \ref{L-CrPrdRight}(\ref{L-CrPrdRight_uv})
and Lemma \ref{L-ActionIsRight}(\ref{L-ActionIsRight-5})).

We now prove part~(\ref{L-DIIntTw-6}).
Using the induction hypothesis at the third step,
we get
\begin{align*}
\sm_{n + 1} \circ \ep_{n}
& = \big( \Ad (u_{-1}) \circ \sm_n \circ \ps_n^{-1} \big)
     \circ \big( \ps_{n} \circ \ep_{n - 1} \circ \ps_{n - 1}^{-1} \big)
         \\
& = \Ad (u_{-1}) \circ \sm_n \circ \ep_{n - 1} \circ \ps_{n - 1}^{-1}
  = \Ad (u_{-1}) \circ \sm_{n - 1} \circ \ps_{n - 1}^{-1}
  = \sm_{n}.
\end{align*}

It remains to prove part~(\ref{L-DIIntTw-7}).
We compute:
\begin{align*}
\ep_{n} (a)
& = \big( \ps_{n} \circ \ep_{n - 1} \circ \ps_{n - 1}^{-1} \big) (a)
  = \big( ( \id_{M^p_d} \otimes_p \ps_{n - 1} )
       \circ \ep_{n - 1} \circ \ps_{n - 1}^{-1} \big) (a)
        \\
& = \big( \id_{M^p_d} \otimes_p \ps_{n - 1} \big)
           \big( e_{0, 0} \otimes \ps_{n - 1}^{-1} (a) \big)
  = e_{0, 0} \otimes a,
\end{align*}
as desired.
\end{proof}

\begin{cor}\label{C_3924_ContrIsIso}
Let $\XBM$ be a \sfm{}
such that $L^p (X, \mu)$ is separable.
Let $\rh \colon A \to \LLp$
be a nonzero but not necessarily unital contractive \hm.
Then $\rh$ is isometric.
\end{cor}

\begin{proof}
We first claim that $\rh (f_n) \neq 0$ for all $n \in \Nz.$
{}From Lemma \ref{L-CrPrdRight}(\ref{L-CrPrdRight_uv})
and Lemma \ref{L-ActionIsRight}(\ref{L-ActionIsRight-6}),
we get $f_n = (u_{-n} f_0) f_0 (f_0 u_n)$
and $f_0 = (u_n f_n) f_n (f_n u_{-n}).$
Since $u_{-n} f_0,$ $f_0 u_n,$ $u_n f_n,$ and $f_n u_{-n}$
are all in~$A,$
we see that if $\rh (f_n) = 0$ for some $n \in \Nz,$
then $\rh (f_0) = 0,$
and then $\rh (f_n) = 0$ for all $n \in \Nz.$
Lemma \ref{L-CrPrdRight}(\ref{L-CrPrdRight-4})
and continuity of~$\rh$
then imply that $\rh = 0.$

For all $n \in \Nz,$
it follows that $\rh |_{f_n A f_n}$
is a nonzero contractive \hm.
Lemma \ref{L-DIIntTw}(\ref{L-DIIntTw-2})
implies that $f_n A f_n$ is isometrically isomorphic to~$\OP{d}{p}.$
So $\rh |_{f_n A f_n}$ is isometric by \Lem{L-UniqLd}.
Since $n$ is arbitrary,
$\rh$ is isometric by Lemma \ref{L-CrPrdRight}(\ref{L-CrPrdRight-4}).
\end{proof}

\begin{cor}\label{C-FullToRed}
The map
$\kp_{\mathrm{r}} \colon
  A = F^p (\Z, B, \bt) \to F^p_{\mathrm{r}} (\Z, B, \bt)$
is an isometric bijection.
\end{cor}

\begin{proof}
By \Lem{L-CompOfCrPrd},
this map is contractive and has dense range.
The algebra $F^p_{\mathrm{r}} (\Z, B, \bt)$
is clearly separable,
so separably representable by Proposition~\ref{P-SepImpSepRep}.
Apply Corollary~\ref{C_3924_ContrIsIso}.
\end{proof}

Although we won't need this,
it also follows
(using Lemma \ref{L-CrPrdRight}(\ref{L-CrPrdRight-2}))
that the \rpn~$\pi$ of Notation~\ref{N-DICrPrd} is isometric.

\begin{thm}\label{T_3924_Iso}
There is an isometric isomorphism
$\gm \colon {\ov{M}}^p_{\Nz} \otimes_{p} \OP{d}{p} \to A$
such that $\gm (e_{0, 0} \otimes 1) = f_0$
and $\gm (e_{0, 0} \otimes a) = \sm_0 (a)$
for all $a \in \OP{d}{p}.$
\end{thm}

\begin{proof}
For $n \in \Nz,$
set $T_n = \{ 0, 1, \ldots, d^n - 1 \} \S \Nz.$
Define $g_n \colon Z^n \to T_n$
by
\[
g_n (j_1, j_2, \ldots, j_n) = \sum_{l = 1}^n j_l d^{n - l}
\]
for $j_1, j_2, \ldots, j_n \in \Z.$
Lemma~\ref{L-MatpSMap} provides isometric homomorphisms
\[
\gm_{g_n, \OP{d}{p}} \colon
 M^p_{Z^n} \otimes_{p} \OP{d}{p} \to M^p_{T_n} \otimes_{p} \OP{d}{p}
\andeqn
\xi_n \colon
 M^p_{T_n} \otimes_{p} \OP{d}{p}
   \to M^p_{T_{n + 1}} \otimes_{p} \OP{d}{p},
\]
the first being an isomorphism and the second
coming from the inclusion of $T_n$ in $T_{n + 1}.$
There is an isometric isomorphism
$\io_n \colon (M^p_d)^{\otimes n} \otimes_p \OP{d}{p}
  \to M^p_{Z^n} \otimes_{p} \OP{d}{p}$
which comes from the identification of
$L^p (Z^n, \ld^n)$ with $L^p (Z, \ld)^{\otimes n},$
such that
for $a \in \OP{d}{p}$ and
$j_1, j_2, \ldots, j_n, k_1, k_2, \ldots, k_n \in \Z,$
we have
\[
\io_n \big( e_{(j_1, j_2, \ldots, j_n),
       \, (k_1, k_2, \ldots, k_n)} \otimes a \big)
 = e_{j_1, k_1} \otimes e_{j_2, k_2} \otimes \cdots \otimes e_{j_n, k_n}
      \otimes a. 
\]
Set $\gm_n = \gm_{g_n, \OP{d}{p}} \circ \io_n.$
One checks easily that $\xi_n \circ \gm_n = \gm_{n + 1} \circ \ep_n.$

Using Lemma \ref{L-DIIntTw}(\ref{L-DIIntTw-4})
and Lemma \ref{L-DIIntTw}(\ref{L-DIIntTw-6})
for the bottom part,
we now get a commutative diagram,
in which the maps in the bottom row are the inclusions,
all maps are isometric, and all vertical maps are bijective:
\[
\xymatrix{
M_{T_0}^p \otimes_p \OP{d}{p} \ar[r]^{\xi_0}
  & M_{T_1}^p \otimes_p \OP{d}{p} \ar[r]^{\xi_1}
  & M_{T_2}^p \otimes_p \OP{d}{p} \ar[r]^{\xi_2}
  & \cdots \\
\OP{d}{p} \ar[r]^{\ep_0} \ar[d]_{\sm_0} \ar[u]^{\gm_0}
  & M_d^p \otimes_p \OP{d}{p} \ar[r]^{\ep_1} \ar[d]_{\sm_1} \ar[u]^{\gm_1}
  & (M_d^p)^{\otimes 2} \otimes_p \OP{d}{p}
          \ar[r]^{\ep_2} \ar[d]_{\sm_2} \ar[u]^{\gm_2}
  & \cdots  \\
f_0 A f_0 \ar[r]
  & f_1 A f_1 \ar[r]
  & f_2 A f_2 \ar[r]
  & \cdots.
}
\]
Therefore the vertical maps induce isometric isomorphisms
of the direct limits.
By Corollary~\ref{C-ClUIsDLim} and Example~\ref{E-MatpS},
the direct limit of the top row
is ${\ov{M}}^p_{\Nz} \otimes_{p} \OP{d}{p}.$
By Corollary~\ref{C-ClUIsDLim} and 
Lemma \ref{L-CrPrdRight}(\ref{L-CrPrdRight-4}),
the direct limit of the bottom row is~$A.$
Thus we get isometric isomorphisms
\[
\gm_{\I} \colon
 \Dirlim_n (M_d^p)^{\otimes n} \otimes_p \OP{d}{p}
     \to {\ov{M}}^p_{\Nz} \otimes_{p} \OP{d}{p}
\andeqn
\sm_{\I} \colon
 \Dirlim_n (M_d^p)^{\otimes n} \otimes_p \OP{d}{p}
     \to A.
\]
Set $\gm = \sm_{\I} \circ \gm_{\I}^{-1}.$
\end{proof}

\begin{cor}\label{C-KThyA}
The map
$\kp_{\mathrm{r}} \circ \sm_0 \colon
  \OP{d}{p} \to F^p_{\mathrm{r}} (\Z, B, \bt)$
is an isomorphism on K-theory
such that
$(\kp_{\mathrm{r}} \circ \sm_0)_* ([1]) = [\kp_{\mathrm{r}} (f_0)].$
\end{cor}

\begin{proof}
It follows from Theorem~\ref{T_3924_Iso}
and \Lem{L-KthyTPM}
that $\sm_0 \colon \OP{d}{p} \to A$ is an isomorphism on K-theory,
and clearly
$(\sm_0)_* ([1]) = [f_0].$
Apply Corollary~\ref{C-FullToRed}.
\end{proof}

\begin{thm}\label{T-KThyOpd}
Let $p \in [1, \I)$
and let $d \in \{ 2, 3, \ldots \}.$
Then $K_1 \big( \OP{d}{p} \big) = 0$
and there is an isomorphism
$K_0 \big( \OP{d}{p} \big) \to \Z / (d - 1) \Z$
which sends $[1] \in K_0 \big( \OP{d}{p} \big)$
to the standard generator $1 + (d - 1) \Z.$
\end{thm}

\begin{proof}
It follows from \Lem{L-KthyTPM}
and Theorem~\ref{P-KStdUHF}
that $K_1 (B) = 0$
and that there is an isomorphism
$\et \colon \Z \big[ \tfrac{1}{d} \big] \to K_0 (B)$
such that $\et (1) = [f_0].$
Therefore $\et (d) = [f_1].$
We have $\bt (f_1) = f_0$
by \Lem{L-ActionIsRight}(\ref{L-ActionIsRight-5}),
so $(\bt^{-1})_*$ is multiplication by~$d.$
The exact sequence~(\ref{Eq:PVLp})
of Theorem~\ref{T-PV}
thus becomes
\[
0 \longrightarrow K_1 \big( F^p_{\mathrm{r}} (\Z, B, \bt) \big)
  \longrightarrow \Z \big[ \tfrac{1}{d} \big]
  \xrightarrow{\,\, 1 - d \,\,} \Z \big[ \tfrac{1}{d} \big]
  \longrightarrow K_0 \big( F^p_{\mathrm{r}} (\Z, B, \bt) \big)
  \longrightarrow 0.
\]
Therefore $K_1 \big( F^p_{\mathrm{r}} (\Z, B, \bt) \big) = 0$
and there is an isomorphism
from
$\Z \big[ \tfrac{1}{d} \big] / (d - 1) \Z \big[ \tfrac{1}{d} \big]$
to $K_0 \big( F^p_{\mathrm{r}} (\Z, B, \bt) \big)$
which sends $1$ to~$[f_0].$
Now apply Corollary~\ref{C-KThyA}
and use the fact that the map $\Z \to \Z \big[ \tfrac{1}{d} \big]$
induces an isomorphism
\[
\Z / (d - 1) \Z
 \to \Z \big[ \tfrac{1}{d} \big] / (d - 1) \Z \big[ \tfrac{1}{d} \big].
\]
This completes the proof.
\end{proof}

\section{Some open problems}\label{Sec_Problems}

\indent
There are many open problems suggested by the general theory.
We single out just a few.

\begin{qst}\label{Q-FullVsRed}
Let $p \in [1, \I) \SM \{ 2 \},$
let $(G, A, \af)$ be a nondegenerately $\sm$-finitely representable
isometric $G$-$L^p$~operator algebra.
(See Definition~\ref{D:Action} and Definition~\ref{D-LpRep}.)
Is the map
$\kp_{\mathrm{r}} \colon
 F^p (G, A, \af) \to F^p_{\mathrm{r}} (G, A, \af)$
of \Lem{L-CompOfCrPrd}
necessarily surjective?
If $G$ is amenable,
is this map necessarily injective?
Surjective? Isometric?
(In any of these questions,
does it help to assume that $G$ is discrete,
$G = \Z,$
or $A = \C$?)
If $G$ is finite,
does it follow that $\kp_{\mathrm{r}}$ is isometric?
(In Remark~\ref{R-FGpCP},
we showed that $\kp_{\mathrm{r}}$ is bijective,
but not that it is isometric.)
\end{qst}

Positive results may well hold only in fairly special circumstances.

\begin{qst}\label{Q-SimplMinHme}
Let $X$ be a \cms,
and let $h \colon X \to X$ be a \mh.
Define $\af \in \Aut (C (X))$ by $\af (f) = f \circ h^{-1}$
for $f \in C (X).$
As in Notation~\ref{N-TransfGpLpAlg},
abbreviate $F^p (\Z, \, C (X), \, \af)$ to $F^p (\Z, X, h)$
and
$F^p_{\mathrm{r}} (\Z, \, C (X), \, \af)$
to $F^p_{\mathrm{r}} (\Z, X, h).$

Is $F^p (\Z, X, h)$ simple?
(The algebra $F^p_{\mathrm{r}} (\Z, X, h)$ is simple
by Theorem~\ref{T-FreeMinSimp},
and we do not know whether it is different from $F^p (\Z, X, h).$)
Can there ever be a nonzero \ct\  \hm{}
from $F^{p_1} (\Z, X_1, h_1)$ to $F^{p_2} (\Z, X_2, h_2)$
or to $F^{p_2}_{\mathrm{r}} (\Z, X_2, h_2)$
with $p_1 \neq p_2$
and $h_1$ and $h_2$ both minimal?
\end{qst}

\begin{qst}\label{Q-InfoOnMin}
Let $h \colon X \to X,$
$F^p (\Z, X, h),$
and $F^p_{\mathrm{r}} (\Z, X, h)$ be as in Question~\ref{Q-SimplMinHme}.
What information about $h$ can one recover from the
isomorphism class or isometric isomorphism class
of $F^p (\Z, X, h)$ and $F^p_{\mathrm{r}} (\Z, X, h)$?
\end{qst}

\begin{qst}\label{Q_3217_Tsr}
Let $h \colon X \to X,$
$F^p (\Z, X, h),$
and $F^p_{\mathrm{r}} (\Z, X, h)$ be as in Question~\ref{Q-SimplMinHme}.
Suppose $X$ is the Cantor set.
Does it follow that the invertible elements of
$F^p_{\mathrm{r}} (\Z, X, h)$ are dense?
That is, does $F^p_{\mathrm{r}} (\Z, X, h)$ have stable rank one?
This is true for $p = 2,$
by~\cite{Pt2}.
If $X = S^1$ and $h$ is an irrational rotation,
does it follow that the invertible elements of
$F^p_{\mathrm{r}} (\Z, X, h)$ are dense?
This is true for $p = 2,$
by~\cite{Pt1}.
\end{qst}

In the case $p = 2,$
stable rank one holds much more generally.
For $p = 2,$ both the special examples in Question~\ref{Q_3217_Tsr}
also have real rank zero.
A unital C*-algebra
has real rank zero \ifo{} it is an exchange ring,
by Theorem~7.2 of~\cite{AGOP},
and the definition of an exchange ring
(see the beginning of Section~1 of~\cite{AGOP})
makes sense for general unital rings.
So it seems reasonable to ask the following:

\begin{qst}\label{Q_3218_CPExch}
In the examples of Question~\ref{Q_3217_Tsr},
is $F^p_{\mathrm{r}} (\Z, X, h)$ an exchange ring?
\end{qst}

\begin{qst}\label{Q_3216_K1Inv0}
Let $A$ be a purely infinite simple unital Banach algebra.
Let $\inv (A)$ denote its invertible group,
and let $\inv_0 (A)$ denote the connected component of $\inv (A)$
which contains~$1.$
Does it follow that the map $\inv (A) / \inv_0 (A) \to K_1 (A)$
is an isomorphism?
\end{qst}

The corresponding result for $K_0$ is true,
by an argument in the proof of Corollary 5.15 of~\cite{PhLp2a}.

\begin{qst}\label{Q-Takai}
Let $p \in [1, \I).$
Let $\af \colon G \to \Aut (A)$
be an isometric action of a second countable locally compact abelian group
on a separable nondegenerately representable $L^p$~operator algebra.
In Theorem~\ref{T-DualpFull} and Theorem~\ref{T-DualpRed},
we have constructed dual actions
\[
{\widehat{\af}} \colon
  {\widehat{G}} \to \Aut \big( F^p (G, A, \af) \big)
\andeqn
{\widehat{\af}} \colon
  {\widehat{G}} \to \Aut \big( F^p_{\mathrm{r}} (G, A, \af) \big).
\]
Is there an analog of Takai duality~\cite{Tki}
for the crossed products by these actions?
\end{qst}

\end{document}

Unused reference:

\bibitem{Bl3} B.~Blackadar,
{\emph{K-Theory for Operator Algebras}}, 2nd ed.,
MSRI Publication Series~{\textbf{5}},
Cambridge University Press, Cambridge, New York, Melbourne, 1998.